\documentclass[12pt,a4paper,reqno]{amsart}

\usepackage{a4wide}
\usepackage{amssymb,amsmath,amsthm,mathrsfs}
\usepackage{graphicx}
\usepackage{float}
\usepackage{colortbl}
\usepackage{enumerate}
\usepackage{upref}
\usepackage[bookmarks=false]{hyperref}

\theoremstyle{plain}% default
\newtheorem{theorem}{Theorem}
\newtheorem{lemma}{Lemma}[section]
\newtheorem{proposition}[lemma]{Proposition}
\newtheorem{corollary}[lemma]{Corollary}
\newtheorem*{question}{Question}
\newtheorem{maincorollary}[theorem]{Corollary}
\theoremstyle{definition}
\newtheorem{definition}{Definition}[section]

\newtheorem*{condition}{Condition}

\theoremstyle{remark}
\newtheorem{remark}{Remark}[section]

\usepackage{enumitem}
\makeatletter
\def\namedlabel#1#2{\begingroup
   #2%
 \def\@currentlabel{#2}%
   \phantomsection\label{#1}\endgroup
}

\def\nn{\ensuremath{\mathscr N}}
\def\R{\ensuremath{\mathbb R}}
\def\N{\ensuremath{\mathbb N}}

\def\Z{\ensuremath{\mathbb Z}}
\def\I{\ensuremath{{\bf 1}}}
\def\e{{\ensuremath{\rm e}}}

\def\S{\ensuremath{\mathcal S}}
\def\RR{\ensuremath{\mathcal R}}

\def\B{\ensuremath{\mathcal B}}
\def\M{\ensuremath{\mathcal M}}
\def\C{\ensuremath{\mathbb C}}

\def\l{{\rm Leb}}
\def\P{\ensuremath{\mathcal P}}
\def\p{\ensuremath{\mathbb P}}

\def\QQ{\ensuremath{\mathscr Q}}

\def\n{\ensuremath{n}}
\def\t{\ensuremath{t}}

\def\X{\mathcal{X}}
\def\ie{{\em i.e.}, }

\def\cv{\ensuremath{\text {Cor}}}

\def\o{\ensuremath{\underline{\omega}}}
\newcommand{\dif}{\mathrm{d}}

\DeclareMathOperator*{\esssup}{ess\;sup}
\DeclareMathOperator*{\essinf}{ess\;inf}

\def\dist{\ensuremath{\text{dist}}}

\def\eps{\varepsilon}

\def\V{\text{Var}}

\def\Iuc{\ensuremath{\I_{U_m^c}}}

\def\sm{\ensuremath{\mu_\varepsilon}}

\def\nm{\ensuremath{\theta^{\N}_{\varepsilon}}}

\def\Pe{\P_\varepsilon}

\numberwithin{equation}{section}

\begin{document}

\title{Laws of rare events for deterministic and random dynamical systems}

\author[H. Ayta\c{c}]{Hale Ayta\c{c}}
\address{Hale Ayta\c{c}\\ Centro de Matem\'{a}tica da Universidade do Porto\\ Rua do
Campo Alegre 687\\ 4169-007 Porto\\ Portugal}
\email{aytach@fc.up.pt}

\author[J. M. Freitas]{Jorge Milhazes Freitas}
\address{Jorge Milhazes Freitas\\ Centro de Matem\'{a}tica \& Faculdade de Ci\^encias da Universidade do Porto\\ Rua do
Campo Alegre 687\\ 4169-007 Porto\\ Portugal}
\email{jmfreita@fc.up.pt}
\urladdr{http://www.fc.up.pt/pessoas/jmfreita}

\author{ Sandro Vaienti}
\address{Sandro Vaienti\\ %(Temporary address) Centro de Modelamiento Matem\'{a}tico, Av. Blanco Encalada 2120 Piso 7, Santiago de Chile. (On leave from/Permanent address)
UMR-7332 Centre de Physique Th\'{e}orique, CNRS, Universit\'{e}
d'Aix-Marseille, Universit\'{e} du Sud, Toulon-Var and FRUMAM,
F\'{e}d\'{e}ration de Recherche des Unit\'{e}s des Math\'{e}matiques de Marseille,
CPT Luminy, Case 907, F-13288 Marseille CEDEX 9}
\email {vaienti@cpt.univ-mrs.fr}

\thanks{HA was partially supported by FCT (Portugal) grant SFRH/BD/33371/2008. JMF was partially supported by FCT grant SFRH/BPD/66040/2009 and by FCT project PTDC/MAT/099493/2008. HA and JMF were supported by the European Regional Development Fund through the programme COMPETE and by the Portuguese Government through the FCT and CMUP under the project PEst-C/MAT/UI0144/2011. All three authors were supported by FCT  project PTDC/MAT/120346/2010, which is financed by national and European Community structural funds through the programs  FEDER and COMPETE. SV was supported by the CNRS-PEPS {\em Mathematical Methods of Climate Theory} and by the ANR-Project {\em Perturbations}; part of this work was done while he was visiting the {\em Centro de Modelamiento Matem\'{a}tico, UMI2807}, in Santiago de Chile with a CNRS support (d\'el\'egation).}

\date{\today}


%37A50      Relations with probability theory and stochastic processes
%60G70      Extreme value theory; extremal processes
%37B20      Notions of recurrence
%60G10      Stationary processes
%37C25      Fixed points, periodic points, fixed-point index theory
%--------
%37A25      Ergodicity, mixing, rates of mixing
%37D25      Nonuniformly hyperbolic systems (Lyapunov exponents, Pesin theory, etc.)
%37D35      Thermodynamic formalism, variational principles, equilibrium states
%37C40      Smooth ergodic theory, invariant measures
%37H99	   Random dynamical systems

\begin{abstract}
The object of this paper is twofold. From one side we study  the dichotomy, in terms of the Extremal Index  of the possible Extreme Value Laws, when the rare events are centred around periodic or non periodic points. Then we build a general theory of Extreme Value Laws for randomly perturbed dynamical systems. We also address, in both situations, the convergence of Rare Events Point Processes. Decay of correlations against $L^1$ observables will play a central role in our investigations.
\end{abstract}

\maketitle

\tableofcontents

\section{Introduction}

Deterministic discrete dynamical systems are often used to model physical phenomena. In many situations, inevitable observation errors make it more realistic to consider random dynamics, where the mathematical model is adjusted by adding random noise to the iterative process in order to account for these practical imprecisions. The behaviour of such random systems has been studied thoroughly in the last decades. We mention, for example, \cite{K86a,KL06} for excellent expositions on the subject.

Laws of rare events for chaotic (deterministic) dynamical systems have also been exhaustively studied in the last years. When these results first appeared these notions were described as Hitting Times Statistics (HTS) or Return Times Statistics (RTS). In this setting, rare events correspond to entrances in small regions of the phase space and the goal is to prove distributional limiting  laws for the normalised waiting times before hitting/returning to  these asymptotically small sets. We refer to \cite{S09} for an excellent review. More recently, rare events have also been studied through Extreme Value Laws (EVLs), \ie the distributional limit of the partial maxima of stochastic processes arising from such chaotic systems simply by evaluating an observable function along the orbits of the system. Very recently, in \cite{FFT10,FFT11}, the two perspectives have been proved to be linked so that, under general conditions on the observable functions, the existence of HTS/RTS is equivalent to the existence of EVLs. These observable functions achieve a maximum (possibly $\infty$) at some chosen point $\zeta$ in the phase space so that the rare event of occurring an exceedance of a high level corresponds to an entrance in a small ball around $\zeta$. The study of rare events may be enhanced if we enrich the process by considering multiple exceedances (or hits/returns to target sets) that are recorded by Rare Events Point Processes (REPP), which count the number of exceedances (or hits/returns) in a certain time frame. Then one looks for limits in distribution for such REPP when time is adequately normalised. 

Surprisingly, not much is known about rare events for random dynamical systems. One of the main goals here is to establish what we believe to be the first result proving the existence of EVLs (or equivalently HTS/RTS) as well as the convergence of REPP, for randomly perturbed dynamical systems.  %Part of the difficulty in establishing this type of result derives from the fact that it is not immediately clear how to choose the best approach in order to prove the existence of HTS/RTS for random dynamics. On the other hand, from the EVL perspective, it is quite straightforward how to address the existence of EVLs even when we have to deal with randomly perturbed systems. We exploit this fact, and the connection between EVLs and HTS/RTS, in order to settle the strategy.

We remark that in the recent paper
\cite{MR11} the authors defined the meaning of first hitting/return time in the random dynamical setting. To our knowledge this was the first paper to address this issue of recurrence for random dynamics. There, the authors define the concepts of quenched and annealed return times for systems generated by the composition of random maps. Moreover, they prove that for super-polynomially mixing systems, the random recurrence rate is equal to the local dimension of the stationary measure.

In here, we are interested in establishing the right setting in order to have the connection between EVL and HTS/RTS, for random dynamics, and, eventually, to prove the existence of EVLs and HTS/RTS for random orbits. Moreover, we also study the convergence of the REPP for randomly perturbed systems. 
These achievements are, in our opinion, the main accomplishments of this paper.
%\textbf{HA: There are two definitions depending on the perspective on noise and we follow one of the definitions given there (where one fix the noise, so called 'quenched' case).}

%\textbf{HA: (6) Do you think the two sentences in bold above gave Mike a confusion?? I add a remark at the end. }

In general terms, we will consider uniformly expanding and piecewise expanding maps. Then we randomly perturb these discrete systems with additive, independent, identically distributed (i.i.d.)~noise introduced at each iteration. The noise distribution is absolutely continuous with respect to (w.r.t.)~Lebesgue measure. The details will be given in Section~\ref{sec:statement-results}.

The main ingredients will be decay of correlations against all $L^1$ observables (we mean decay of correlations of all observables in some Banach space against all observables in $L^1$, which will be made more precise in Definition~\ref{def:dc} below) and the notion of first return time from a set to itself.

We realised that the techniques we were using to study the random scenario also allowed us to give an answer to one of the questions raised in \cite{FFT12}. There the connection between periodicity, clustering of rare events and the Extremal Index (EI) was studied. In certain situations, like when rare events are defined as entrances in balls around (repelling) periodic points, the stochastic processes generated by the dynamics present clustering of rare events. The EI is a parameter $\vartheta\in[0,1]$ which quantifies the intensity of the clustering. In fact, in most situations the average cluster size is just $1/\vartheta$. No clustering means that $\vartheta=1$ and strong clustering means that $\vartheta$ is close to $0$. In  \cite[Section~6]{FFT12}, it is shown that, for uniformly expanding maps of the circle equipped with the Bernoulli measure,% like $f:{\mathcal S}^1\to {\mathcal S}^1$ with $f(z)=z^2$, 
there is a dichotomy in terms of the possible EVL: either the rare events are centred at (repelling) periodic points and $\vartheta<1$ or at non periodic points and the EI is $1$. This was proved for cylinders, in the sense that rare events corresponded to entrances into dynamically defined cylinders (instead of balls) and one of the questions it raised was if this dichotomy could be proved more generally for balls and for more general systems.  In \cite{FP12}, the authors build up on the work of \cite{H93} and eventually obtain the dichotomy for balls and for conformal repellers.

One of our results here, Theorem~\ref{thm:DC+R=>EVL}, allows to prove the dichotomy for balls and for systems with decay of correlations against $L^1$ which include, for example, piecewise expanding maps of the interval like Rychlik maps (Proposition~\ref{prop:Rychlik}) or piecewise expanding maps in higher dimensions, like the ones studied by Saussol, in \cite{S00}, (Proposition~\ref{prop:MPE}). Moreover, as an end product of our approach, we can express the dichotomy for these systems in the following more general terms (see Propositions~\ref{prop:Rychlik} and \ref{prop:MPE}): either we have, at non periodic points, the convergence of the REPP to the standard Poisson process or we have, at repelling periodic points, the convergence of REPP to a compound Poisson process consisting of an underlying asymptotic Poisson process governing the positions of the clusters of exceedances and a multiplicity distribution associated to each such Poisson event, which is determined by the average cluster size. In fact, at repelling periodic points, we always get that the multiplicity distribution is the geometric distribution (see \cite{HV09,FFT12a}).

We also consider discontinuity points of the map as centres of the rare events (see Proposition~\ref{prop:Rychlik.discontinuous}). A very interesting immediate consequence of this study is that, when we consider the REPP, we can obtain convergence to a compound Poisson process whose multiplicity distribution is not a geometric distribution. To our knowledge this is the first time these limits are obtained  for the general piecewise expanding systems considered and in the balls' setting (rather than cylinders), in the sense that exceedances or rare events correspond to the entrance of the orbits in topological balls.       

In the course of writing this paper we came across a paper by Keller, \cite{K12}, where he proved the dichotomy of expanding maps with a spectral gap for the corresponding Perron-Frobenius operator (which also include Rychlik maps and the higher dimensional piecewise expanding maps studied by Saussol \cite{S00}, for example). He uses of a powerful technique developed in \cite{KL09}, based on an eigenvalue perturbation formula. Our approach here is different since we use an EVL kind of argument and our assumptions are based on decay of correlations against $L^1$ observables. Moreover, in here, we also deal with the convergence of the REPP and obtain, in particular, the interesting fact that at discontinuity points we observe multiplicity distributions other than the geometric one.
 
 We also mention the very recent paper \cite{KR12}, where the dichotomy for cylinders is established for mixing countable alphabet shifts, but also in the context of nonconventional ergodic sums. It also includes examples of non-convergence of the REPP, in the cylinder setting.

We remark that in most situations, decay of correlations against $L^1$ observables is a consequence of the existence of a gap in the spectrum of the map's corresponding Perron-Frobenius operator. However, in  \cite{D98}, Dolgopyat proves exponential decay of correlations for certain Axiom A flows but along the way he proves it for semiflows against $L^1$ observables. This is done via estimates on families of twisted transfer operators for the Poincar\'e map, but without considering the Perron-Frobenius operator for the flow itself. This means that the discretisation of this flow by using a time 1 map, for example, provides an example of a system with decay of correlations against $L^1$ for which it is not known if there exists a spectral gap of the corresponding Perron-Frobenius operator. 
Apparently, the existence of a spectral gap for the map's Perron-Frobenius operator, defined in some nice function space, implies decay of correlations against $L^1$ observables. However, the latter is still a very strong property. In fact, from decay of correlations against $L^1$ observables, regardless of the rate, as long as it is summable, one can  actually  show that the system has exponential decay of correlations of H\"older observables against $L^\infty$. (See \cite[Theorem~B]{AFL11}). So an interesting question is:

\begin{question}
If a system presents summable decay of correlations against $L^1$ observables, is there a spectral gap for the system's Perron-Frobenius operator, defined in some appropriate function space?
\end{question}

We note that, as we point out in Remark~\ref{rem:L1stuff}, we do not actually need decay of correlations against $L^1$ in its full strength.

Returning to the random setting, our main result asserts that the dichotomy observed for deterministic systems vanishes and regardless of the centre being a periodic point or not, we always get standard exponential EVLs or, equivalently, standard exponential HTS/RTS (which means that $\vartheta=1$). Moreover, we also show that the REPP converges in distribution to a standard Poisson process. We will prove these results in Section~\ref{sec:random} using an EVL approach, where the main assumption will be decay of correlations against $L^1$.

Still in the random setting, motivated by the deep work of Keller, \cite{K12}, in Section~\ref{sec:Keller}, we prove  results in the same directions as before but based on the spectral approach used by Keller and Liverani to study deterministic systems. As a byproduct we get an HTS/RTS formula with sharp error terms for random dynamical systems (see Proposition~\ref{prop:hts-error-terms}). We will point out the differences between the two techniques (which we name here as {\em direct} and {\em spectral}, respectively), at the beginning of Section 5; let us simply stress that we implemented the spectral technique in random situation only for one-dimensional systems and the existence of EI was proved for a substantially large class of noises. On the other hand, the  direct technique worked for systems in higher dimensions as well, but it required additive noise with a continuous distribution. However, the latter was necessary to prove that EI is $1$ in the spectral approach too.

\section{Statement of Results}\label{sec:statement-results}

Consider a discrete time dynamical system $(\X,\B,\p,T)$ which will denote two different but interrelated settings throughout the paper. $\X$ is a topological space, $\mathcal B$ is the Borel $\sigma$-algebra, $T:\X\to\X$ is a measurable map and $\p$ is a $T$-invariant probability measure, \ie $\p(T^{-1}(B))=\p(B)$, for all $B\in \mathcal B$. Also, given any $A\in\mathcal B$ with $\p(A)>0$, let $\p_A$ denote the conditional measure on $A\in\B$, \ie $\p_A:=\frac{\p|_A}{\p(A)}$.

Firstly, it will denote a deterministic setting where $\X=\mathcal M$ is a compact Riemannian manifold, $\mathcal B$ is the Borel $\sigma$-algebra, $T=f:\M\to\M$ is a piecewise differentiable map and $\p=\mu$ is an $f$-invariant probability measure. Let $\dist(\cdot,\cdot)$ denote a Riemannian metric on $\mathcal M$ and $\l$ a normalized volume form on the Borel sets of $\mathcal M$ that we call Lebesgue measure.

 Secondly, it will denote a random setting which is constructed from the deterministic system via perturbing the original map with random additive noise. We assume that $\mathcal M$ is a quotient of a Banach vector space $\mathcal V$, like $\mathcal M=\mathbb T^d=\R^d/\Z^d$, for some $d\in\N$.  In the case $d=1$, we will also denote the circle $\mathbb T^1$ by ${\mathcal S}^1$. Let $\dist(\cdot,\cdot)$ denote the induced usual quotient metric on $\mathcal M$ and $\l$ a normalised volume form on the Borel sets of $\mathcal M$ that we call Lebesgue measure. Also denote the ball of radius $\varepsilon>0$ around $x\in \mathcal M$ by $B_\varepsilon(x):=\{y\in\mathcal M: \dist(x,y)<\varepsilon\}$. Consider the unperturbed deterministic system $f:\mathcal M\to\mathcal M$. For some $\varepsilon>0$, let $\theta_\varepsilon$ be a probability measure defined on the Borel subsets of $B_\varepsilon(0)$, such that
 \begin{equation}
 \label{eq:noise-distribution}
\theta_\varepsilon=g_\varepsilon\l\quad\mbox{and} \quad 0<\underline{g_\varepsilon}\leq g_\varepsilon\leq\overline{g_\varepsilon}<\infty.
 \end{equation}
 For each $\omega\in B_\varepsilon(0)$, we define the additive
 perturbation of $f$ that we denote by $f_\omega$ as the map
 $f_\omega:\mathcal M\to\mathcal M$, given by \footnote{In the general theory of
 randomly perturbed dynamical systems one could
 consider perturbations other than the additive ones and
 distributions $\theta_{\eps}$ which are not necessarily absolutely
 continuous. Our choice is motivated by the fact that our main
 result for the extreme values in presence of noise could be relatively easily
 showed with those assumptions, but it is also clear from the proof
 where possible generalizations could occur. We were especially
 concerned in constructing the framework and in finding the good
 assumptions for the theory, which is surely satisfied for more
 general perturbations and probability distributions. Let us notice
 that other authors basically used additive noise when they studied
 statistical properties of random dynamical systems \cite{BBM02,BBM03, AA03}, for instance.}
\begin{equation}
\label{eq:add-perturbation}
f_{\omega}(x)=f(x)+\omega. 
\end{equation}
Consider a sequence of i.i.d.\ random variables
(r.v.)~ $W_1, W_2,\ldots$ taking values on $B_\varepsilon(0)$
with common distribution given by $\theta_{\varepsilon}$.
 Let $\Omega=B_\varepsilon(0)^\N$ denote the space of realisations of such process and $\theta_\varepsilon^\N$ the product measure defined on its Borel subsets. Given a point $x\in\mathcal M$ and the realisation of the stochastic process  $\o=(\omega_1,\omega_2,\ldots)\in\Omega$, we define the random orbit of $x$ as $x, f_{\o}(x), f^2_{\o}(x),\ldots$ where, the evolution of $x$, up to time $n\in\N$, is obtained by the concatenation of the respective additive randomly perturbed maps in the following way:
\begin{equation}
\label{eq:concatenated}
f_{\o}^n(x)=f_{\omega_n}\circ f_{\omega_{n-1}}\circ\cdots\circ f_{\omega_1}(x),
\end{equation}
with $f_{\o}^0$ being the identity map on $\mathcal M$.
Next definition gives a notion that plays the role of invariance in the deterministic setting.
\begin{definition}
\label{def:stationary-measure}
Given $\varepsilon>0$, we say that the probability measure $\mu_\varepsilon$ on the Borel subsets of $\mathcal M$ is stationary if
\[
\iint \phi(f_{\omega}(x)) \,\dif\mu_\varepsilon(x)\,\dif\theta_\varepsilon(\omega)=\int \phi(x)\,\dif\mu_\varepsilon(x),
\]
for every $\phi:\mathcal M\to\R$ integrable w.r.t.\ $\mu_\varepsilon$.
\end{definition}
The previous equality could also be written as
$$
\int {\mathcal U}_{\eps}\phi\,\dif\mu_{\eps}=\int \phi\,\dif\mu_{\eps}
$$
where the operator ${\mathcal U}_{\eps}: L^{\infty}(\l)\rightarrow
L^{\infty}(\l)$, %($L^{\infty}$ to be intended w.r.t.\ the
%Lebesgue measure), 
is defined as $({\mathcal U}_{\eps}\phi)(x)=\int
_{B_{\eps}(0)}\phi(f_{\omega}(x))\,\dif\theta _{\eps}$ and it is called
the {\em random evolution operator}.\\ The adjoint of this operator
is called the {\em random Perron-Frobenius operator}, ${\mathcal
P}_{\eps}: L^1(\l)\rightarrow L^1(\l)$, %($L^1$ is to be intended w.r.t.\
%the Lebesgue measure),
 and it acts by duality as
$$
\int {\mathcal P}_{\eps}\psi\cdot\phi\, \dif\l = \int {\mathcal
U}_{\eps}\phi\cdot\psi\, \dif\l
$$
where $\psi\in L^1$ and $\phi \in L^{\infty}$.\\ It is immediate
from this definition to get another useful representation of this
operator, namely for $\psi\in L^1$: $$ ({\mathcal
P}_{\eps}\psi)(x)=\int_{B_{\eps}(0)}({\mathcal P}_{\omega}\psi)(x)
\,\dif\theta_{\eps}(\omega),$$ where ${\mathcal P}_{\omega}$ is the
Perron-Frobenius operator associated to $f_{\omega}$.\\
We recall that the stationary measure $\mu_{\eps}$ is absolutely
continuous w.r.t.\ the Lebesgue measure and with density $h_{\eps}$
if and only if such a density is a fixed point of the random
Perron-Frobenius operator: ${\mathcal P}_{\eps}h_{\eps}=h_{\eps}$ \footnote{The duality explains why we take ${\mathcal
P}_{\eps}$ acting on $L^1$ and ${\mathcal U}_{\eps}$ on $L^{\infty}$. Moreover our stationary measures will be absolutely continuous with density given by the fixed point of the Perron-Frobenius operator ${\mathcal P}_{\eps}$. }

We can give a deterministic representation of this random setting using the following skew product transformation:
\begin{equation} \label{def:skew-product}S:
\begin{array}[t]{ccc}
\mathcal M\times \Omega & \longrightarrow & \mathcal M\times \Omega\\
(x,\o)& \longmapsto & (f_{\omega_1},\sigma(\o)),
\end{array}
\end{equation}
where $\sigma:\Omega\to\Omega$ is the one-sided shift $\sigma(\o)=\sigma(\omega_1,\omega_2,\ldots)=(\omega_2, \omega_3, \ldots)$.
We remark that $\mu_\varepsilon$ is stationary if and only if the product measure $\mu_\varepsilon\times \theta_\varepsilon^\N$ is an $S$-invariant measure.

Hence, the random evolution can fit the original model $(\X,\mathcal B, \p, T)$ by taking the product space
$\X=\mathcal M\times \Omega$, with the corresponding product Borel $\sigma$-algebra $\mathcal B$,
 where the product measure $\p=\sm\times\theta_\epsilon^\N$ is defined. The system is then given by the skew product map $T=S$.

For systems we will consider, $\p$ has very good mixing properties, which in loose terms means that the system loses memory quite fast. In order to quantify the memory loss we look at the system's rates of decay of correlations w.r.t.\ $\p$.

\begin{definition}[Decay of correlations]
\label{def:dc}
Let \( \mathcal C_{1}, \mathcal C_{2} \) denote Banach spaces of real valued measurable functions defined on \( \X \).
We denote the \emph{correlation} of non-zero functions $\phi\in \mathcal C_{1}$ and  \( \psi\in \mathcal C_{2} \) w.r.t.\ a measure $\p$ as
\[
\cv_\p(\phi,\psi,n):=\frac{1}{\|\phi\|_{\mathcal C_{1}}\|\psi\|_{\mathcal C_{2}}}
\left|\int \phi\, (\psi\circ T^n)\, \dif\p-\int  \phi\, \dif\p\int
\psi\, \dif\p\right|.
\]

We say that we have \emph{decay
of correlations}, w.r.t.\ the measure $\p$, for observables in $\mathcal C_1$ \emph{against}
observables in $\mathcal C_2$ if, for every $\phi\in\mathcal C_1$ and every
$\psi\in\mathcal C_2$ we have
 $$\cv_\p(\phi,\psi,n)\to 0,\quad\text{ as $n\to\infty$.}$$
  \end{definition}

In the random setting, we will only be interested in Banach spaces of functions that do not depend on $\o\in\Omega$, hence, we assume that $\phi,\psi$ are actually functions defined on $\mathcal M$ and the correlation between these two observables can be written more simply as
\begin{align}
\cv_\p(\phi,\psi, n):&=\frac{1}{\|\phi\|_{\mathcal
C_{1}}\|\psi\|_{\mathcal C_{2}}} \left|\int \left(\int \psi\circ
f^n_{\o}\, \dif\theta_{\varepsilon}^{\N}\right)\phi\, \dif\sm-\int  \phi\,
\dif\sm\int \psi\, \dif\sm\right|\nonumber\\
&=\frac{1}{\|\phi\|_{\mathcal C_{1}}\|\psi\|_{\mathcal C_{2}}}
\left|\int {\mathcal U}_{\eps}^n\psi\cdot\phi\, \dif\mu_{\eps}-\int  \phi\,
\dif\sm\int \psi\, \dif\sm\right| \label{eq:cor-random}
\end{align}
where $(\mathcal{U}^n_{\varepsilon} \psi)(x)=\int\cdots\int
\psi(f_{\omega_n}\circ\cdots\circ f_{\omega_1}x)\,
\dif\theta_\varepsilon (\omega_n)\ldots \dif\theta_\varepsilon
(\omega_1)=\int \psi\circ f^n_{\o}(x)\, \dif\theta_{\varepsilon}^{\N}.$\\

We say that we have \emph{decay of correlations against $L^1$
observables} whenever  we have decay of correlations, with respect
to the measure $\p$, for observables in $\mathcal C_1$ against
observables in $\mathcal C_2$ and $\mathcal C_2=L^1(\l)$ is the
space of $\l$-integrable functions on $\mathcal M$ and
$\|\psi\|_{\mathcal C_{2}}=\|\psi\|_1=\int |\psi|\,\dif\l$. Note that
when $\mu$, $\mu_\varepsilon$ are absolutely continuous with respect
to $\l$ and the respective Radon-Nikodym derivatives are bounded
above and below by positive constants, then
$L^1(\l)=L^1(\mu)=L^1(\mu_\varepsilon)$.

The goal is to study the statistical properties of such systems regarding the occurrence of rare events. There are two approaches for this purpose which were recently proved to be equivalent.

We first turn to the existence of an EVL for the partial maximum of observations made along the time evolution of the system. To be more precise consider the time series $X_0,X_1,X_2,\dots$ arising from such a system simply by evaluating a given random variable (r.v.) $\varphi:\mathcal M\to\R\cup\{+\infty\}$ along the orbits of the system:
\begin{equation}
\label{eq:def-stat-stoch-proc-DS} X_n=\varphi\circ f^n,\quad \mbox{for
each } n\in {\mathbb N}.
\end{equation}
Note that when we consider the random dynamics, the process will be
\begin{equation}
\label{eq:def-rand-stat-stoch-proc-RDS2} X_n=\varphi\circ f^n_{\o},\quad \mbox{for
each } n\in {\mathbb N},
\end{equation}
which can also be written as $X_n=\bar\varphi\circ S^n$, where \begin{equation} \label{def:random-observable}\bar\varphi:
\begin{array}[t]{ccc}
\mathcal M\times \Omega & \longrightarrow & \R\cup\{+\infty\}\\
(x,\o)& \longmapsto & \varphi(x)
\end{array},
\end{equation}
Clearly, $X_0,X_1,\dots$ defined in this way is not an independent sequence.  However, invariance of $\mu$ and stationarity of $\mu_\varepsilon$ guarantee that the stochastic process is stationary in both cases.

We assume that the r.v.\ $\varphi:\M\to\R\cup\{\pm\infty\}$
achieves a global maximum at $\zeta\in \M$ (we allow
$\varphi(\zeta)=+\infty$). 
We also assume that $\varphi$ and $\p$ are sufficiently regular so that:

\begin{enumerate}

\item[\namedlabel{item:U-ball}{(R1)}] 
for $u$ sufficiently close to $u_F:=\varphi(\zeta)$,  the event 
\begin{equation*}
U(u)=\{X_0>u\}=\{x\in\M:\; \varphi(x)>u\}
\end{equation*} corresponds to a topological ball centred at $\zeta$. Moreover, the quantity $\p(U(u))$, as a function of $u$, varies continuously on a neighbourhood of $u_F$.

\end{enumerate}

In what follows, an \emph{exceedance} of the level $u\in\R$ at time $j\in\N$ means that the event $\{X_j>u\}$ occurs. We denote by $F$ the distribution function (d.f.)~of $X_0$, \ie $F(x)=\p(X_0\leq x)$. Given any d.f.\ $G$, let $\bar{G}=1-G$ and $u_G$ denote the right endpoint of the d.f.\ $G$, \ie
$
u_G=\sup\{x: G(x)<1\}.
$

The idea then is to consider the extremal behaviour of the system for which we define a new sequence of random variables  $M_1, M_2,\ldots$ given by
\begin{equation}
\label{eq:Mn-definition}
M_n=\max\{X_0,\ldots,X_{n-1}\}.
\end{equation}

\begin{definition}
We say that we have an \emph{EVL} for $M_n$ if there is a non-degenerate d.f.\ $H:\R\to[0,1]$ with $H(0)=0$ and,  for every $\tau>0$, there exists a sequence of levels $u_n=u_n(\tau)$, $n=1,2,\ldots$,  such that
\begin{equation}
\label{eq:un}
  n\,\p(X_0>u_n)\to \tau,\;\mbox{ as $n\to\infty$,}
\end{equation}
and for which the following holds:
\begin{equation}
\label{eq:EVL-law}
\p(M_n\leq u_n)\to \bar H(\tau),\;\mbox{ as $n\to\infty$.}
\end{equation}
\end{definition}

%\textbf{HA: (13) it's been used like this so far, in other papers too. however this is a deep question. in fact, 'extreme value' is for $X_i$' s, but EVL is defined to be for $M_n$' s, $M_n$' s constitute a new sequence of random variables, and the distributional law to be defined is for them. In conclusion, 'law' is for $M_n$' s but it is the 'Extreme Value' for $X_i$' s.  so EV is like a label referring to $X_i$' s in my opinion. }

\begin{remark}
\label{rem:advantage-EVL}
We remark that one of the advantages of the EVL approach for the study of rare events for random dynamics is that its definition follows straightforwardly from the deterministic case. In fact, the only difference is that for random dynamical systems, the r.v.\ $M_n$'s are defined on $\mathcal M\times \Omega$ where we use the measure $\p=\mu_\varepsilon\times\theta_\varepsilon^\N$ as opposed to the deterministic case where the ambient space is $\mathcal M$ and $\p=\mu$.
\end{remark}
The motivation for using a normalising sequence $u_n$ satisfying \eqref{eq:un} comes from the case when $X_0, X_1,\ldots$ are independent and identically distributed. In this i.i.d.\ setting, it is clear that $\p(M_n\leq u)= (F(u))^n$. Hence, condition \eqref{eq:un} implies that
\[
\p(M_n\leq u_n)= (1-\p(X_0>u_n))^n\sim\left(1-\frac\tau n\right)^n\to\e^{-\tau},
\]
as $n\to\infty$. Moreover, the reciprocal is also true. Note that in this case $H(\tau)=1-\e^{-\tau}$ is the standard exponential d.f.

%We assume that $\varphi$ achieves a global maximum at $\zeta\in\mathcal M$, for every $u<\varphi(\zeta)$ but sufficiently close to $\varphi(\zeta)$, the event $\{x\in\mathcal M:\; \varphi(x)>u\}=\{X_0>u\}$ corresponds to a topological ball ``centred'' at $\zeta$ and, 
For every sequence $(u_n)_{n\in\N}$ satisfying \eqref{eq:un} we define:
\begin{equation}
\label{def:Un}
U_n:=\{X_0>u_n\}.
\end{equation} 
%is a nested sequence of sets such that
%\begin{equation}
%\label{def:zeta}
%\bigcap_{n\in\N} U_n=\{\zeta\}.
%\end{equation}

When $X_0,X_1,X_2,\ldots$ are not independent, the standard exponential law still applies under some conditions on the dependence structure. These conditions are the following:
\begin{condition}[$D_2(u_n)$]\label{cond:D2} We say that $D_2(u_n)$ holds for the sequence $X_0,X_1,\ldots$ if for all $\ell,t$
and $n$
\begin{align*}
|\p\left(X_0>u_n\cap
  \max\{X_{t},\ldots,X_{t+\ell-1}\leq u_n\}\right)-\p(X_0>u_n)
  \p(M_{\ell}\leq u_n)|\leq \gamma(n,t),
\end{align*}
where $\gamma(n,t)$ is decreasing in $t$ for each $n$ and
$n\gamma(n,t_n)\to0$ when $n\rightarrow\infty$ for some sequence
$t_n=o(n)$.
\end{condition}
Now, let $(k_n)_{n\in\N}$ be a sequence of integers such that
\begin{equation}
\label{eq:kn-sequence-1}
k_n\to\infty\quad \mbox{and}\quad  k_n t_n = o(n).
\end{equation}
\begin{condition}[$D'(u_n)$]\label{cond:D'} We say that $D'(u_n)$
holds for the sequence $X_0, X_1, X_2, \ldots$ if there exists a sequence $(k_n)_{n\in\N}$ satisfying \eqref{eq:kn-sequence-1} and such that
\begin{equation}
\label{eq:D'un}
\lim_{n\rightarrow\infty}\,n\sum_{j=1}^{\lfloor n/k_n \rfloor}\p( X_0>u_n,X_j>u_n)=0.
\end{equation}
\end{condition}
By \cite[Theorem~1]{FF08a}, if conditions $D_2(u_n)$ and $D'(u_n)$ hold for $X_0, X_1,\ldots$ then there exists an EVL for $M_n$ and $H(\tau)=1-\e^{-\tau}$. Besides, as it can be seen in  \cite[Section~2]{FF08a} condition  $D_2(u_n)$ follows immediately if $X_0, X_1,\ldots$ is given by \eqref{eq:def-stat-stoch-proc-DS} and the system has sufficiently fast decay of correlations.

Now,  we turn to the other approach which regards the existence of HTS and RTS. In the deterministic case,
consider a set $A\in\B$. We define a function that we refer to as \emph{first hitting time function} to $A$ and denote by $r_A:\X\to\N\cup\{+\infty\}$ where
\begin{equation*}
r_A(x)=\min\left\{j\in\N\cup\{+\infty\}:\; f^j(x)\in A\right\}.
\end{equation*}
The restriction of $r_A$ to $A$ is called the \emph{first return time function} to $A$. We define the \emph{first return time} to $A$, which we denote by $R(A)$, as the minimum of the return time function to $A$, \ie
\[
R(A)=\min_{x\in A} r_A(x).
\]

In the random case, we have to a make choice regarding the type of definition we want to play the roles of the first hitting/return times (functions). Essentially, there are two possibilities. The \emph{quenched} perspective which consists of fixing a realisation $\o\in\Omega$ and define the objects in the same way as in the deterministic case. The \emph{annealed} perspective consists of defining the same objects by averaging over all possible realisations $\o$. Here, we will use the quenched perspective to define hitting/return times because it will facilitate the connection between EVLs and HTS/RTS in the random setting. (We refer to \cite{MR11} for more details on both perspectives.)

For some $\o\in\Omega$ fixed, some $x\in \M$ and $A\subset \M$ measurable, we may define the \emph{first random hitting time}
\begin{equation*}
\label{eq:rht}
r_A^{\o}(x):=\min\{j\in\N:\; f_{\o}^j(x)\in A\}
\end{equation*}
and the \emph{first random return} from A to A as
\begin{equation*}
\label{eq:rrt}
R^{\o}(A)=\min\{r_A^{\o}(x):\; x\in A\}.
\end{equation*}
\begin{definition}
\label{def:HTS/RTS}
Given a sequence of subsets of $\X$, $(V_n)_{n\in \N}$, so that
$\p(V_n)\to 0$, the system has (random) \emph{HTS}
$G$ for $(V_n)_{n\in \N}$ if for all $t\ge 0$
\begin{equation}\label{eq:def-HTS-law}
 \p\left(r_{V_n}\leq\frac t{\p(V_n)}\right)\to G(t) \;\mbox{ as $n\to\infty$,}
\end{equation}
and the system has (random) \emph{RTS}
$\tilde G$ for $(V_n)_{n\in \N}$ if for all $t\ge 0$
\begin{equation}\label{eq:def-RTS-law}
\p_{V_n}\left(r_{V_n}\leq\frac t{\p(V_n)}\right)\to\tilde G (t)\;\mbox{ as $n\to\infty$}.
\end{equation}
In the deterministic case, $\X=\mathcal M$, $\p=\mu$ and $T=f$. In the random case, $\X=\mathcal M\times\Omega$, $\p=\mu_\varepsilon\times\theta_\varepsilon^\N$, $T=S$ defined in \eqref{def:skew-product}, $V_n=V_n^*\times\Omega$, where $V_n^*\subset \mathcal M$ and $\mu_\varepsilon(V_n^*)\to 0$, as $n\to \infty$.
\end{definition}
Note that
\[
 \p\left(r_{V_n}\leq\frac t{\p(V_n)}\right)=\mu_\varepsilon\times\theta_\varepsilon^\N\left(r^{\o}_{V_n^*}\leq\frac t{\mu_{\varepsilon}(V_n^*)}\right).
\]

The normalising sequences to obtain HTS/RTS, are motivated by Kac's Lemma, which states that the expected value of $r_A$ w.r.t.\ $\mu_A$ is  $\int_A r_A~d\mu_A =1/\mu(A)$.  So in studying the fluctuations of $r_A$ on $A$, the relevant normalising factor should be $1/\mu(A)$.

The existence of exponential HTS is equivalent to the existence of exponential RTS. In fact, according to the Main Theorem in \cite{HLV05}, a system has HTS $G$ if and only if it has RTS $\tilde G$ and
\begin{equation}
\label{eq:HTS-RTS}
G(t)=\int_0^t(1-\tilde G(s))\,\dif s.
\end{equation}

In \cite{FFT10}, the link between HTS/RTS (for balls) and  EVLs of stochastic processes given by \eqref{eq:def-stat-stoch-proc-DS} was established for invariant measures $\mu$ absolutely continuous w.r.t.\ $\l$. Essentially, it was proved that if such time series have an EVL $H$ then the system has HTS $H$ for balls ``centred'' at $\zeta$ and vice versa. %\textbf{HA: (17) I suggest to delete the following sentence:} %\textbf{JMF: I prefer putting it in brackets}
(Recall that having HTS $H$ is equivalent to saying that the system has RTS $\tilde H$, where $H$ and $\tilde H$ are related by \eqref{eq:HTS-RTS}). This was based on the elementary observation that for stochastic processes given by \eqref{eq:def-stat-stoch-proc-DS} we have:
\begin{equation}
\label{eq:rel-Mn-r}
\{M_n\leq u\}=\{r_{\{X_0>u\}}>n\}.
\end{equation}
This connection was exploited to prove EVLs using tools from HTS/RTS and the other way around. In  \cite{FFT11}, we carried the connection further to include more general measures, which, in particular, allows us to obtain the connection in the random setting. To check that we just need to use the skew product map to look at the random setting as a deterministic system and take the observable $\bar\varphi:\mathcal M\times\Omega\to\R\cup\{+\infty\}$ defined as in \eqref{def:random-observable} with $\varphi:\mathcal M\to\R\cup\{+\infty\}$ as in \cite[equation (4.1)]{FFT11}. Then Theorems~1 and 2 from \cite{FFT11} guarantee that if we have an EVL, in the sense that \eqref{eq:EVL-law} holds for some d.f.\ $H$, then we have HTS for sequences $\{V_n\}_{n\in\N}$, where $V_n=B_{\delta_n}\times \Omega$ and $\delta_n\to0$ as $n\to\infty$, with $G=H$ and vice-versa.

If we consider multiple exceedances we are lead to point processes of rare events counting the number of exceedances in a certain time frame. For every $A\subset\R$ we define 
\[
\nn_u(A):=\sum_{i\in A\cap\N_0}\I_{X_i>u}.
\]
In the particular case where $A=I=[a,b)$ we simply write 
$\nn_{u,a}^b:=\nn_u([a,b)).$ 
Observe that $\nn_{u,0}^n$ counts the number of exceedances amongst the first $n$ observations of the process $X_0,X_1,\ldots,X_n$ or, in other words, the number of entrances in $U(u)$ up to time $n$. Also, note that
\begin{equation}
\label{eq:rel-HTS-EVL-pp}
\{\nn_{u,0}^n=0\}=\{M_n\leq u\}%=\{r_{U(u)}>n\}
\end{equation}

In order to define a point process that captures the essence of an EVL and HTS through \eqref{eq:rel-HTS-EVL-pp}, we need to re-scale time using the factor $v:=1/\p(X>u)$ given by Kac's Theorem. However, before we give the definition, we need some formalism. Let $\S$ denote the semi-ring of subsets of  $\R_0^+$ whose elements
are intervals of the type $[a,b)$, for $a,b\in\R_0^+$. Let $\RR$
denote the ring generated by $\S$. Recall that for every $J\in\RR$
there are $k\in\N$ and $k$ intervals $I_1,\ldots,I_k\in\S$ such that
$J=\cup_{i=1}^k I_j$. In order to fix notation, let
$a_j,b_j\in\R_0^+$ be such that $I_j=[a_j,b_j)\in\S$. For
$I=[a,b)\in\S$ and $\alpha\in \R$, we denote $\alpha I:=[\alpha
a,\alpha b)$ and $I+\alpha:=[a+\alpha,b+\alpha)$. Similarly, for
$J\in\RR$ define $\alpha J:=\alpha I_1\cup\cdots\cup \alpha I_k$ and
$J+\alpha:=(I_1+\alpha)\cup\cdots\cup (I_k+\alpha)$.

\begin{definition}
We define the \emph{rare event point process} (REPP) by
counting the number of exceedances (or hits to $U(u_n)$) during the (re-scaled) time period $v_nJ\in\RR$, where $J\in\RR$. To be more precise, for every $J\in\RR$, set
\begin{equation}
\label{eq:def-REPP} N_n(J):=\nn_{u_n}(v_nJ)=\sum_{j\in v_nJ\cap\N_0}\I_{X_j>u_n}.
\end{equation}
\end{definition}

Under similar dependence conditions to the ones just seen above, the REPP just defined converges in distribution to a standard Poisson process, when no clustering is involved and to a compound Poisson process with intensity $\theta$ and a geometric multiplicity d.f., otherwise. For completeness, we define here what we mean by a Poisson and a compound Poisson process. (See \cite{K86} for more details.)

\begin{definition}
\label{def:compound-poisson-process}
Let $T_1, T_2,\ldots$ be  an i.i.d.\ sequence of random variables with common exponential distribution of mean $1/\theta$. Let  $D_1, D_2, \ldots$ be another i.i.d.\ sequence of random variables, independent of the previous one, and with d.f.\ $\pi$. Given these sequences, for $J\in\RR$, set
$$
N(J)=\int \I_J\;d\left(\sum_{i=1}^\infty D_i \delta_{T_1+\ldots+T_i}\right),
$$ 
where $\delta_t$ denotes the Dirac measure at $t>0$.  Whenever we are in this setting, we say that $N$ is a compound Poisson process of intensity $\theta$ and multiplicity d.f.\ $\pi$.
\end{definition}
\begin{remark}
\label{rem:poisson-process}
In this paper, the multiplicity will always be integer valued which means that $\pi$ is completely defined by the values $\pi_k=\p(D_1=k)$, for every $k\in\N_0$. Note that, if $\pi_1=1$ and $\theta=1$, then $N$ is the standard Poisson process and, for every $t>0$, the random variable $N([0,t))$ has a Poisson distribution of mean $t$. 
\end{remark}

\begin{remark}
\label{rem:compound-poisson}
When clustering is involved, we will see that $\pi$ is actually a geometric distribution of parameter $\theta\in (0,1]$,  \ie $\pi_k=\theta(1-\theta)^k$, for every $k\in\N_0$. This means that, as in \cite{HV09}, here, the random variable $N([0,t))$ follows a P\'olya-Aeppli distribution, \emph{i.e.}:
$$
\p(N([0,t))=k)=\e^{-\theta t}\sum_{j=1}^k \theta^j(1-\theta)^{k-j}\frac{(\theta t)^j}{j!}\binom{k-1}{j-1},
$$
for all $k\in\N$ and $\p(N([0,t))=0)=\e^{-\theta t}$. 
\end{remark}

When $D'(u_n)$ holds, since there is no clustering, then, due to a criterion proposed by Kallenberg \cite[Theorem~4.7]{K86}, which applies only to simple point processes, without multiple events, we can simply adjust condition $D_2(u_n)$ to this scenario of multiple exceedances in order to prove that the REPP converges in distribution to a standard Poisson process. We denote this adapted condition by:
\begin{condition}[$D_3(u_n)$]\label{cond:D^*} Let $A\in\RR$ and $t\in\N$.
We say that $D_3(u_n)$ holds for the sequence
$X_0,X_1,\ldots$ if
\[ \left|\p\left(\{X_0>u_n\}\cap
  \{\nn(A+t)=0\}\right)-\p(\{X_0>u_n\})
  \p(\nn(A)=0)\right|\leq \gamma(n,t),
\]
where $\gamma(n,t)$ is nonincreasing in $t$ for each $n$ and
$n\gamma(n,t_n)\to0$ as $n\rightarrow\infty$ for some sequence
$t_n=o(n)$, which means that $t_n/n\to0$ as $n\to \infty$.
\end{condition}
Condition $D_3(u_n)$ follows, as easily as $D_2(u_n)$, from sufficiently fast decay of correlations. 

In \cite[Theorem~5]{FFT10}  a strengthening of \cite[Theorem~1]{FF08a} is proved, which essentially says that, under $D_3(u_n)$ and $D'(u_n)$, the REPP  $N_n$ defined in \eqref{eq:def-REPP} converges in distribution to a standard Poisson process.
%\begin{theorem}[{\cite[Theorem~5]{FFT10}}]
%\label{thm:D3+D'=>Poisson}
%  Let $X_1, X_2,\ldots$ be a stationary stochastic process
%  for which conditions $D_3(u_n)$ and $D'(u_n)$ hold for a
%  sequence of levels $u_n$ such that \eqref{eq:un} holds.
% Then the REPP $N_n$ defined
%  in \eqref{eq:def-REPP} is such that
%  $N_n\xrightarrow[]{d}N$, as $n\rightarrow\infty$, where $N$
%  denotes a Poisson Process with intensity $1$. %$\tau$
%\end{theorem}

Next, we give an abstract result, in the deterministic setting, that allows to check conditions $D_2(u_n)$ and $D'(u_n)$ for any stochastic process $X_0,X_1,\dots$ arising from a system which has decay of correlations against $L^1$ observables. As a consequence of this result in Section~\ref{sec:dichotomy}, more precisely in Propositions \ref{prop:Rychlik}, \ref{prop:MPE} and \ref{prop:Rychlik.discontinuous}, we will obtain the announced dichotomy for the EI based on the periodicity of the point $\zeta$.

\begin{theorem}
\label{thm:DC+R=>EVL}
Consider a dynamical system $(\mathcal M,\B,\mu,f)$ for which there exists 
a Banach space $\mathcal C$ of real valued functions such that 
for all $\phi\in\mathcal C$ and $\psi\in L^1(\mu)$,
\begin{equation}
\label{DC:L1}
\cv_\mu(\phi,\,\psi,n)\leq  Cn^{-2},
\end{equation} where $C>0$ is a constant independent of both $\phi, \psi$. Let $X_0, X_1, \ldots$ be given by \eqref{eq:def-stat-stoch-proc-DS}, where $\varphi$ achieves a global maximum at some point $\zeta$ for which condition \ref{item:U-ball} 
 holds.  Let $u_n$ be such that \eqref{eq:un} holds,  $U_n$ be defined as in \eqref{def:Un} and set $R_n:=R(U_n)$.
%
%
%
%For any point $\zeta$, consider that $X_0, X_1,\ldots$ is defined as in \eqref{eq:def-stat-stoch-proc-DS}, let $u_n$ be such that \eqref{eq:un} holds and assume that $U_n$, defined in \eqref{def:Un}, is such that \eqref{def:zeta} holds. Set $R_n:=R(U_n)$.

 If there exists $C'>0$ such that for all $n$ we have $\I_{U_n}\in \mathcal C$, $\|\I_{U_n}\|_{\mathcal C}\leq C'$ and $R_n\to\infty$, as $n\to\infty$, then conditions $D_2(u_n)$ and $D'(u_n)$ hold for $X_0,X_1,\ldots$. This implies that there is an EVL for $M_n$ defined in \eqref{eq:Mn-definition}   and  $H(\tau)=1-\e^{-\tau}$.
 %since $D_2(u_n)$ and $D'(u_n)$ may not be necessary conditions here.
 \end{theorem}
 In light of the connection between EVLs and HTS/RTS it follows immediately:
 \begin{maincorollary}
 \label{for:DC+R=>HTS} Under the same hypothesis of Theorem~\ref{thm:DC+R=>EVL} we have HTS/RTS for balls around $\zeta$ with $G(t)=\tilde G(t)=1-\e^{\t}$.
 \end{maincorollary}

Since, under the same assumptions of Theorem~\ref{thm:DC+R=>EVL}, condition $D_3(u_n)$ holds trivially then applying \cite[Theorem~5]{FFT10} we obtain:
 \begin{maincorollary}
 \label{cor:DC+R=>Poisson} Under the same hypothesis of Theorem~\ref{thm:DC+R=>EVL}, the REPP $N_n$ defined
  in \eqref{eq:def-REPP} is such that
  $N_n\xrightarrow[]{d}N$, as $n\rightarrow\infty$, where $N$
  denotes a Poisson Process with intensity $1$.
 \end{maincorollary}

\begin{remark}
\label{rem:cont}
Note that condition $R_n\to\infty$, as $n\to\infty$, is easily verified if the map is continuous at every point of the orbit of $\zeta$. We will state this formally in Lemma~\ref{prop:Rn-continuous}.
\end{remark}

\begin{remark}
\label{rem:decay-L1}
Observe that decay of correlations as in \eqref{DC:L1} against $L^1(\mu)$ observables is a very strong property. In fact, regardless of the rate (in this case $n^{-2}$), as long as it is summable, one can  actually  show that the system has exponential decay of correlations of H\"older observables against $L^\infty(\mu)$. (See \cite[Theorem~B]{AFL11}.)
\end{remark}

Now, we give an abstract result in the random setting which concludes by stating that by adding random noise, regardless of the   point $\zeta$ chosen, we always get an EI equal to 1.
\begin{theorem}
\label{thm:random-EVL}
Consider a dynamical system $(\mathcal M\times \Omega, \mathcal B, \mu_\varepsilon\times\theta_\varepsilon^\N, S)$, where $\mathcal M=\mathbb T^d$, for some $d\in\N$, $f:\mathcal M\to \mathcal M$ is a deterministic system which is randomly perturbed as in \eqref{eq:add-perturbation} with noise distribution given by \eqref{eq:noise-distribution} and $S$ is the skew product map defined in \eqref{def:skew-product}. Assume that there exists $\eta>0$ such that $\dist(f(x),f(y))\leq \eta\dist(x,y)$, for all $x,y\in\mathcal M$. Assume also that the stationary measure $\mu_\varepsilon$ is such that $\mu_\varepsilon=h_\varepsilon \l$, with $0<\underline h _\varepsilon\leq h_\varepsilon\leq \overline h _\varepsilon<\infty$. Suppose that there exists a Banach space $\mathcal C$ of real valued functions defined on $\mathcal M$ such that
for all $\phi\in\mathcal C$ and $\psi\in L^1(\mu_\varepsilon)$,
\begin{equation}
\label{RDC:L1}
\cv_{\mu_\varepsilon\times\theta_\varepsilon^\N}(\phi,\,\psi,n)\leq  Cn^{-2},
\end{equation} where $\cv_{\mu_\varepsilon\times\theta_\varepsilon^\N}(\cdot)$ is defined as in \eqref{eq:cor-random} and $C>0$ is a constant independent of both $\phi, \psi$. 

For any point $\zeta\in\mathcal M$, consider that $X_0, X_1,\ldots$ is defined as in \eqref{eq:def-rand-stat-stoch-proc-RDS2}, let $u_n$ be such that \eqref{eq:un} holds and assume that $U_n$ is defined as in \eqref{def:Un}.%, is such that \eqref{def:zeta} holds.
If there exists $C'>0$ such that for all $n$ we have $\I_{U_n}\in \mathcal C$ and $\|\I_{U_n}\|_{\mathcal C}\leq C'$, then the stochastic process $X_0, X_1,\ldots$ satisfies $D_2(u_n)$ and $D'(u_n)$, which implies that we have an EVL for $M_n$ such that $\bar H(\tau)=\e^{-\tau}$.

\end{theorem}

Again, using the connection between EVLs and HTS/RTS we get
 \begin{maincorollary}
 \label{cor:random-EVL=>HTS} Under the same hypothesis of Theorem~\ref{thm:random-EVL} we have exponential HTS/RTS for balls around $\zeta$, in the sense that \eqref{eq:def-HTS-law} and \eqref{eq:def-RTS-law} hold with $G(t)=\tilde G(t)=1-\e^{\t}$ and $V_n=B_{\delta_n}(\zeta)\times\Omega$, where $\delta_n\to0$, as $n\to\infty$.
 \end{maincorollary}
 
 Moreover, appealing to \cite[Theorem~5]{FFT10} once again, we have
 \begin{maincorollary}
 \label{cor:random-EVL=>Poisson} Under the same hypothesis of Theorem~\ref{thm:random-EVL}, the stochastic process $X_0, X_1,\ldots$ satisfies $D_3(u_n)$ and $D'(u_n)$, which implies that  the REPP $N_n$ defined
  in \eqref{eq:def-REPP} is such that
  $N_n\xrightarrow[]{d}N$, as $n\rightarrow\infty$, where $N$
  denotes a Poisson Process with intensity $1$.
 \end{maincorollary}

\begin{remark}
\label{rem:interval-maps}
We remark that we do not need to consider that $\mathcal M$ is a $d$ dimensional torus in order to apply the theory. Basically, we only need that $f_\omega(\mathcal M)\subset \mathcal M$, for all $\omega\in B_\eps(0)$.
As we will see in more details in Section~\ref{sec:random}, for example, piecewise expanding maps of the interval, with finitely many branches, 
satisfy all the conclusions of Theorem~\ref{thm:random-EVL} .
\end{remark}

\section{Extremal Index dichotomy for deterministic systems}
\label{sec:dichotomy}
In this section we will start by proving Theorem~\ref{thm:DC+R=>EVL}, Corollary~\ref{cor:DC+R=>Poisson} and a simple lemma asserting that continuity is enough to guarantee that $R_n\to\infty$, as $n\to\infty$.

Next, we give examples of systems to which we can apply Theorem~\ref{thm:DC+R=>EVL} in order to prove a dichotomy regarding the existence of an EI equal to 1 or less than 1, depending on whether $\zeta$ is non-periodic or periodic, respectively. This will be done for uniformly expanding and piecewise expanding maps, when all points in the orbit of $\zeta$ are continuity points of the map.

In the third subsection, we will consider Rychlik maps, which are piecewise expanding maps of the interval, and will analyse the EI also in the cases when the orbit of $\zeta$ hits a discontinuity point of the map. 

%\textbf{HA: (22) introduction to the section }

\subsection{Decay against $L^1$ implies exponential EVL at non-periodic points}
\label{subsec:proof-dichtomy}
\begin{proof}[Proofs of Theorem~\ref{thm:DC+R=>EVL} and Corollary~\ref{cor:DC+R=>Poisson}]

As explained in \cite[Section~5.1]{F12}, conditions $D_2(u_n)$ and $D_3(u_n)$ are designed to follow easily from decay of correlations. In fact, if we choose $\phi=\I_{U_n}$ and $\psi=\I_{\{M_\ell\leq u_n\}}$, in the case of $D_2(u_n)$, and $\psi=\I_{\nn(A)=0}$, for some $A\in\mathcal R$, in the case of $D_3(u_n)$, we have that we can take $\gamma(n,t)=C^* t^{-2}$, where $C^*=CC'$. Hence, conditions $D_2(u_n)$ and $D_3(u_n)$ are trivially satisfied for the sequence $(t_n)_n$ given by $t_n=n^{2/3}$, for example.

Now, we turn to condition $D'(u_n)$.
Taking $\psi=\phi=\I_{U_n}$ in \eqref{DC:L1} and since $\|\I_{U_n}\|_{\mathcal C}\leq C'$ we easily get
\begin{align}
\label{eq:estimate1}
\mu\left(U_n\cap f^{-j}(U_n)\right) \le
 (\mu(U_n))^2+C \left\| \I_{U_n}\right\|_{\mathcal C} \left\| \I_{U_n}\right\|_{L^1(\mu)} j^{-2}\leq  (\mu(U_n))^2+C^*\mu(U_n)j^{-2},
\end{align}
where $C^*=CC'>0$.
By definition of $R_n$, estimate \eqref{eq:estimate1} and since $n\mu(U_n)\to \tau$ as $n \to \infty$ it follows that there exists some constant $D>0$ such that
\begin{align*}
n\sum_{j=1}^{\lfloor n/k_n \rfloor}& \mu(U_n\cap f^{-j}(U_n)) = n\sum_{j=R_n}^{\lfloor n/k_n \rfloor} \mu(U_n\cap f^{-j}(U_n))\le n\big\lfloor\tfrac {n}{k_n}\big\rfloor\mu(U_n)^2 +n\,C^*\mu(U_n) \sum_{j=R_n}^{\lfloor n/k_n \rfloor}j^{-2}\\
&\le \frac{(n\mu(U_n))^2}{k_n} +n\,C^*\mu(U_n) \sum_{j=R_n}^{\infty}j^{-2}\leq D \left(\frac{\tau^2}{k_n}+\tau \sum_{j=R_n}^{\infty}j^{-2} \right)\xrightarrow[n\to\infty]{}0.
\end{align*}
\end{proof}

\begin{remark}
\label{rem:L1stuff}
In the above demonstration it is important to use $L^1$-norm to obtain the factor $\mu(U_n)$ in the second summand of the last term in \eqref{eq:estimate1}, which is crucial to kill off the $n$ factor coming from the definition of $D'(u_n)$. However, note that we actually do not need decay of correlations against $L^1$ in its full strength, which means that it holds for all $L^1$ functions. In fact, in order to prove $D'(u_n)$ we only need it to hold for the functions $\I_{U_n}$.

Also, note that we do not need such a strong statement regarding the decay of correlations of the system in order to prove $D_2(u_n)$ or $D_3(u_n)$. In particular, even if $\I_{U_n}\notin \mathcal C$ (as when $\mathcal C$ is the space of H\"older continuous functions), we can still verify these conditions by using a suitable H\"older approximation. (See \cite[Proposition~5.2]{F12}.)
%
%
%proceed as in \cite[Lemma~3.3]{C01} or \cite[Lemma~6.1]{FFT10}, where condition $D_2(u_n)$ is still verified by using a suitable H\"older approximation.
\end{remark}

According to Theorem~\ref{thm:DC+R=>EVL}, in general terms, if the system has decay of correlations against $L^1$ observables, then to prove $D'(u_n)$ one has basically to show that $R_n\to\infty$, as $n\to\infty$. Next lemma gives us a sufficient condition for that to happen.

\begin{lemma}
\label{prop:Rn-continuous}
Assume that $\zeta$ is not a periodic point and that $f$ is continuous at every point of the orbit of $\zeta$, namely $\zeta, f(\zeta), f^2(\zeta),\ldots$, then $\lim_{n\to\infty}R_n=\infty$, where $R_n$ is as in Theorem~\ref{thm:DC+R=>EVL}.
\end{lemma}

\begin{proof}
Let $j\in\N$. We will show that if $n\in\N$ is sufficiently large then $R_n>j$. Let $\epsilon=\min_{i=1,\ldots,j}\dist(f^i(\zeta),\zeta)$. Our assumptions assure that each $f^i$, for $i=1,\ldots,j$, is continuous at $\zeta$. Hence, for every $i=1,\ldots,j$, there exists $\delta_i>0$ such that $f^i(B_{\delta_i}(\zeta))\subset B_{\epsilon/2}(f^i(\zeta))$. Let $U:=\bigcap_{i=1}^j B_{\delta_i}(\zeta)$. If we choose $N$ sufficiently large that $U_n\subset U$ for all $n\geq N$, then using the definition of $\epsilon$ it is clear that $f^i(U_n)\cap U_n=\emptyset$, for all $i=1,\ldots,j$, which implies that $R_n>j$.
\end{proof}

%\begin{remark}
%\label{rem:periodic-points}
%Note that if $\zeta$ is a periodic point of prime period $p$ then clearly, regardless of how large $n$ is, $R_n=p$. However, for periodic points, as explained in \cite{FFT12}, if they are repelling then we can still get a limit law with $\bar H(\tau)=\e^{-\vartheta\tau}$, for an EI $0<\vartheta<1$, if some conditions on the long and short term dependence structure of the process hold. In general terms, the analysis is done as before except for the fact that the role of the balls $U_n$, $n=1,2,\ldots$, is replaced by that of the annuli $A_n=U_n\setminus f^{-p}(U_n)$. In particular, the dependence conditions to be verified, called $D^p(u_n)$ and $D'_p(u_n)$, are just obtained from conditions $D_2(u_n)$ and $D'(u_n)$ by replacing $U_n$ by $A_n$. Hence, if we replace the conditions of Theorem~\ref{thm:DC+R=>EVL} on $U_n$ by similar ones on $A_n$, instead, then the conclusion will be that conditions $D^p(u_n)$ and $D'_p(u_n)$ hold. This is proved in \cite[Lemma~4.1 and Proposition 2]{FFT12}. Besides, in \cite[Proposition 2]{FFT12} it is shown that for repelling periodic points $\lim_{n\to\infty}R(A_n)=\infty$.
%\end{remark}

\subsection{The dichotomy for specific systems}\label{subsec:examples-dichotomy}
One of the results in \cite{FFT12} is that for uniformly expanding systems like the doubling map, there is a dichotomy in terms of the type of laws of rare events that one gets at every possible centre $\zeta$. Namely, it was shown that either $\zeta$ is non-periodic, in which case, you always get a standard exponential EVL/HTS, or $\zeta$ is a periodic (repelling) point, in which case you obtain an exponential law with an EI $0<\vartheta<1$ given by the expansion rate at $\zeta$ (see \cite[Section~6]{FFT12}). This was proved for cylinders rather than balls, meaning that the sets $U_n$ are dynamically defined cylinders (see \cite[Section~5]{FFT12} or \cite[Section~5]{FFT11}, for details). Results for cylinders are weaker than the ones for balls, since, in rough terms, it means that the limit is only obtained for certain subsequences of $n\in\N$ rather than for the whole sequence.

In \cite{FFT12}, it was conjectured that this dichotomy should hold in greater generality, namely for balls rather than cylinders and more general systems. As a consequence of Theorem~\ref{thm:DC+R=>EVL} we will be able to show that the dichotomy indeed holds for balls and more general systems. We remark that from the results in \cite{FP12}, one can also derive the dichotomy for conformal repellers and, in \cite{K12}, the dichotomy is also obtained for maps with a spectral gap for their Perron-Frobenius operator. In both these papers, the results were obtained by studying the spectral properties of the Perron-Frobenius operator.

\subsubsection{Rychlik maps}

We will introduce a class of dynamical systems considered by Rychlik in \cite{R83}.  This class includes, for example, piecewise $C^2$ uniformly expanding maps of the unit interval with the relevant physical measures.  We first need some definitions.

\begin{definition}
\label{def:variation}
Given a potential $\psi:Y\to \R$ on an interval $Y$, the \emph{variation} of $\psi$ is defined as
$${\rm Var}(\psi):=\sup\left\{\sum_{i=0}^{n-1} |\psi(x_{i+1})-\psi(x_i)|\right\},$$
where the supremum is taken over all finite ordered sequences $(x_i)_{i=0}^n\subset Y$.
\end{definition}

We use the norm $\|\psi\|_{BV}= \sup|\psi|+{\rm Var}(\psi)$, which makes $BV:=\left\{\psi:Y\to \R:\|\psi\|_{BV}<\infty\right\}$ into a Banach space. We also define
\[
S_n\psi(x):=\psi(x)+\cdots +\psi\circ f^{n-1}(x).
\]
\begin{definition}\label{def:equilibrium-state}
For a measurable potential $\psi:\X\to \R$, we define the \emph{pressure} of $(\X,f,\phi)$ to be
$$P(\phi):=\sup_{\p\in \M_f}\left\{h(\p)+\int\phi~d\p:-\int\phi~d\p<\infty\right\},$$
where $\M_f$ is the set of $f$-invariant measures and $h(\p)$ denotes the metric entropy of the measure $\p$, see \cite{W82} for details.
If $\p$ is an invariant probability measure such that $h(\p_\phi)+\int\phi~d\p=P(\phi)$, then we say that $\p$ is an \emph{equilibrium state}.
\end{definition}

\begin{definition}\label{def:conformal-measure}
 A measure $m$ is called a \emph{$\phi$-conformal} measure if $m(\M)=1$ and if whenever $f:A\to f(A)$ is a bijection, for a Borel set $A$, then
$m(f(A))=\int_A e^{-\phi}~dm$.  Therefore, 
if $f^n:A\to f^n(A)$ is a bijection then
$m(f^n(A))=\int_A e^{-S_n\phi}~dm.$
\end{definition}

%\textbf{HA: (26) definition of 'conformality' and 'equilibrium state' is missing}

\begin{definition}[Rychlik system]\label{def:Rychlik}
$(Y,f, \psi)$ is a \emph{Rychlik system} if $Y$ is an interval, $\{Y_i\}_i$ is an at most countable collection of open intervals such that $\bigcup_i \overline{Y}_i\supset Y$ (where $\overline{Y}_i$ is the closure of $Y_i$),  $f:\bigcup_{i} Y_i \to Y$  is a function continuous on each $Y_i$, which admits a continuous extension to the closure of $Y_i$ that we denote by $f_i:\overline{Y}_i\to Y$ and $\psi:Y\to [-\infty, \infty)$ is a potential such that
\begin{enumerate}
\item
$f_i:\overline{Y}_i\to f(\overline{Y}_i)$ is a diffeomorphism;
\item
${\rm Var}\ \e^\psi<+\infty$, $\psi=-\infty$ on $Y \setminus \bigcup_i Y_i$ and $P(\psi)=0$;
\item there is a $\psi$-conformal measure $m_\psi$ on $Y$;
\item
$(f,\psi)$ is expanding: $\displaystyle \sup_{x\in Y} \psi(x) < 0$.
\end{enumerate}
\end{definition}

Rychlik \cite{R83}  proved that these maps have exponential decay of correlations against $L^1$ observables. To be more precise, if
$(Y,f,\psi)$ is a topologically mixing Rychlik system, then there exists an equilibrium state $\mu_\psi=h m_\psi$ where $h\in BV$%\textbf{HA: here referee has a comment: in the paper \cite{R83}, in part 3, close to the end of the first paragraph, it says $h\geq 0$. In the book of Boyarsky and Gora, in Theorem 8.2.3. it has a positive lower bound in case being lower semi continuous. So maybe we should ask for it for the density??} 
and $m_\psi$ and $\mu_\psi$ are non-atomic and $(Y, f, \mu_\psi)$ has exponential decay of correlations, i.e., there exists $C>0$ and $\gamma\in (0,1)$ such that
\begin{equation}
\label{eq:Rychlik-DC}
\left|\int\varsigma\circ f^n\cdot \phi~\dif\mu_\psi- \int\varsigma~\dif\mu_\psi \int\phi~\dif\mu_\psi\right| \le C\|\varsigma\|_{L^1(\mu_\psi)}\|\phi\|_{BV} \gamma^n,
\end{equation}
for any $\varsigma\in L^1(\mu_\psi)$ and $\phi\in BV$. 
Note that, in the original statement, instead of the $L^1(\mu_\psi)$-norm, the $L^1(m_\psi)$-norm appeared. However, we will assume that $h>c$, for some $c>0$, which means that we can write \eqref{eq:Rychlik-DC} as it is. We remark that $h$ being bounded below by a positive constant is not very restrictive. That is the case if, for example, $h$ is lower semi-continuous (see \cite[Theorem~8.2.3]{BG97}) or if the system has summable variations as uniformly expanding systems with H\"older continuous potentials do.

Let $\mathbb S=Y\setminus \bigcup_i Y_i$ and define $ \Lambda:=\{x\in Y: \, f^n(x)\notin \mathbb S,\;\text{for all}\; n\in\N_0 \}$. As a consequence of Theorem~\ref{thm:DC+R=>EVL}, Corollary~\ref{cor:DC+R=>Poisson} and Lemma~\ref{prop:Rn-continuous} it follows immediately:
\begin{proposition}
\label{prop:Rychlik}
Suppose that $(Y,f,\psi)$ is a topologically mixing Rychlik system, $\psi$ is H\"older continuous on each $\overline Y_i$, and $\mu=\mu_\psi$ is the corresponding equilibrium state such that $\frac{d\mu_\psi}{d m_\psi}>c$, for some $c>0$.  Let $X_0, X_1, \ldots$ be given by \eqref{eq:def-stat-stoch-proc-DS}, where $\varphi$ achieves a global maximum at some point $\zeta$. Then we have an EVL for $M_n$ and
\begin{enumerate}

\item if $\zeta\in\Lambda$ is not a periodic point then the EVL is such that  $\bar H(\tau)=\e^{-\tau}$ and the REPP $N_n$ converges in distribution to a standard Poisson process $N$ of intensity 1.

\item if $\zeta\in\Lambda$ is a (repelling) periodic point of prime period $p$ then the EVL is such that $\bar H(\tau)=\e^{-\vartheta\tau}$ where the EI is given by $\vartheta=1-\e^{S_p\psi(\zeta)}$ and the REPP $N_n$ converges in distribution to a compound Poisson process $N$ with intensity $\vartheta$ and multiplicity d.f. $\pi$ given by
$
\pi(\kappa)=\vartheta(1-\vartheta)^\kappa,
$
for every $\kappa\in\N_0$.
\end{enumerate}

\end{proposition}
\begin{proof}
We start by noting that statement (2) has already been proved in \cite[Proposition~2]{FFT12} and \cite[Corollary~3]{FFT12a}.

Regarding statement (1),
first note that for Rychlik maps, \eqref{eq:Rychlik-DC} clearly implies that condition \eqref{DC:L1} is satisfied. Besides since $U_n$ must be an interval then $\I_{U_n}\in BV$ and $\|\I_{U_n}\|_{BV}\leq 2$. Moreover, by definition of $\Lambda$, we can apply Lemma~\ref{prop:Rn-continuous} and consequently obtain that $\lim_{n\to\infty}R_n=\infty$. Hence, we are now in condition of applying Theorem~\ref{thm:DC+R=>EVL}  and   Corollary~\ref{cor:DC+R=>Poisson} to obtain the result.
\end{proof}
%\begin{proposition}
%\label{prop:Rychlik}
%Suppose that $(Y,f,\psi)$ is a topologically mixing Rychlik system, $\psi$ is H\"older continuous on each $\overline Y_i$, and $\mu=\mu_\psi$ is the corresponding equilibrium state such that $\frac{d\mu_\psi}{d m_\psi}>c$, for some $c>0$.  Let $X_0, X_1, \ldots$ be given by \eqref{eq:def-stat-stoch-proc-DS}, where $\varphi$ achieves a global maximum at some point $\zeta$. Then we have an EVL for $M_n$ and
%\begin{enumerate}
%
%\item if $\zeta\in\Lambda$ is not a periodic point then the EVL is such that  $\bar H(\tau)=\e^{-\tau}$;
%
%\item if $\zeta\in\Lambda$ is a (repelling) periodic point of prime period $p$ then the EVL is such that $\bar H(\tau)=\e^{-\vartheta\tau}$ where the EI is given by $\vartheta=1-\e^{S_p\psi(\zeta)}$.
%
%\end{enumerate}
%
%
%
%\end{proposition}
%\begin{proof}
%We start by noting that statement (2) does not need to be proved since it is precisely the content of  \cite[Proposition~2]{FFT12}.
%
%Regarding statement (1),
%first note that for Rychlik maps, \eqref{eq:Rychlik-DC} clearly implies that condition \eqref{DC:L1} is satisfied. Besides since $U_n$ must be an interval then $\I_{U_n}\in BV$ and $\|\I_{U_n}\|_{BV}\leq 2$. Moreover, by definition of $\Lambda$, we can apply Lemma~\ref{prop:Rn-continuous} and consequently obtain that $\lim_{n\to\infty}R_n=\infty$. Hence, we are now in condition of applying Theorem~\ref{thm:DC+R=>EVL} to obtain the result.
%\end{proof}
\subsubsection{Piecewise expanding maps in higher dimensions}\label{subset:example-Saussol} As a second example, we will consider multidimensional piecewise uniformly expanding maps for which we follow the definition given by Saussol \cite{S00}. As pointed out in \cite{AFL11}, these maps generalize Markov maps which also contains one-dimensional piecewise uniformly expanding maps. 

We need some notation: $\dist(\cdot,\cdot)$ being the usual metric in $\R^N$, given any $\eps>0$, we introduce $B_{\eps}(x)=\{y\in\R^N: \dist(x,y)<\eps\}$. Moreover, $Z$ being a compact subset of $\R^N$, for any $A\subset Z$ and given a real number $c>0$, we write $B_c(A)=\{x\in \R^N: \dist(x,A)\leq c\}$; $Z^\circ$ stands for the interior of $Z$, and $\overline{Z}$ is the closure.
\begin{definition}[Multidimensional piecewise expanding system]
\label{def:MPE}
$(Z,f,\mu)$ is a \emph{multidimensional piecewise expanding system} if $Z$ is a compact  subset of $\R^N$ with $\overline{Z^\circ}=Z$, $f:Z\to Z$ and $\{Z_i\}$ is a family of at most countably many disjoint open sets such that $\l(Z\setminus\bigcup_{i} Z_i)=0$ and there exist open sets $\widetilde{Z_i}\supset\overline{Z_i}$ and $C^{1+\alpha}$ maps $f_i: \widetilde{Z_i}\to\R^N$, for some real number $0<\alpha\leq 1$ and some sufficiently small real number $\eps_1>0$  such that for all $i$,
\begin{enumerate}
\item $f_i(\widetilde{Z_i})\supset B_{\eps_1}(f(Z_i))$;\\
\item for $x,y\in f(Z_i)$ with $\dist(x,y)\leq\eps_1$,
$$|\det Df_i^{-1}(x)-\det Df_i^{-1}(y)|\leq c|\det Df_i^{-1}(x)|\dist(x,y)^\alpha;$$
\item there exists $s=s(f)<1$ such that $\forall x,y\in f(\widetilde{Z_i}) \textrm{ with } \dist(x,y)\leq\eps_1$, we have
$$\dist(f_i^{-1}x,f_i^{-1}y)\leq s\, \dist(x,y).$$
\item Let us put $G(\eps,\eps_1):=\sup_x G(x,\eps,\eps_1)$ where
\begin{equation}\label{GG}
G(x,\eps,\eps_1):=\sum_{i}\frac{\l(f_i^{-1}B_{\eps}(\partial f Z_i)\cup B_{(1-s)\eps_1}(x))}{\l(B_{(1-s)\eps_1}(x))}
\end{equation}
and assume that $\sup\limits_{\delta\leq\eps_1}\big(s^\alpha+2\sup\limits_{\eps\leq\delta}\frac{G(\eps)}{\eps^\alpha}\delta^\alpha\big)<1$.
\end{enumerate}
\end{definition}

Now, let us introduce the space of quasi-H\"older functions in which the spectrum of corresponding Perron-Frobenius operator is investigated.
Given a Borel set $\Gamma\subset Z$, we define the oscillation of $\varphi\in L^1(\l)$ over $\Gamma$ as
$$\mathrm{osc}(\varphi,\Gamma):=\esssup\limits_{\Gamma}\varphi-\essinf\limits_{\Gamma}\varphi.$$

It is easy to verify that $x\mapsto \mathrm{osc}(\varphi,B_{\eps}(x))$ defines a measurable function (see \cite[Proposition 3.1]{S00}). Given real numbers $0<\alpha\leq1$ and $\eps_0>0$, we define $\alpha$-seminorm of $\varphi$ as
$$|\varphi|_{\alpha}=\displaystyle\sup_{0<\eps\leq\eps_0}\eps^{-\alpha}\int_{\R^\N}\mathrm{osc}(\varphi,B_{\eps}(x))\,\dif\l(x).$$
Let us consider the space of functions with bounded $\alpha$-seminorm
$$V_\alpha=\{\varphi\in L^1(\l): |\varphi|_\alpha<\infty\},$$
and endow $V_\alpha$ with the norm
$$\|\cdot\|_\alpha=\|\cdot\|_{L^1(\l)}+|\cdot|_\alpha$$
which makes it into a Banach space. We note that $V_\alpha$ is independent of the choice of $\eps_0$.
According to \cite[Theorem 5.1]{S00}, there exists an absolutely continuous invariant probability measure (a.c.i.p.) $\mu$. Also in \cite[Theorem 6.1]{S00}, it is shown that on the mixing components $\mu$ enjoys exponential decay of correlations against $L^1$ observables on $V_\alpha$, more precisely, if the map $f$ is as defined above and if $\mu$ is the mixing a.c.i.p., then there exist constants $C<\infty$ and $\gamma<1$ such that
\begin{equation}\label{eq:doc-pen}
\Big|\int_{Z}\psi\circ f^n\, h\, \dif\mu \Big|\leq C\|\psi\|_{L^1}\|h\|_\alpha\gamma^n,\, \forall\psi\in L^1,\mbox{ where }\int\psi\,\dif\mu=0 \textrm{ and } \forall h\in V_\alpha.
\end{equation}
We refer the reader to \cite{S00} for the exact values of the above constants.\\

Let $\mathbb S=Z\setminus \bigcup_{i} Z_i$ and define $ \Lambda:=\{x\in Z: \, f^n(x)\notin \mathbb S,\;\text{for all}\; n\in\N_0 \}$. As a consequence of Theorem~\ref{thm:DC+R=>EVL}, Corollary~\ref{cor:DC+R=>Poisson} and Lemma~\ref{prop:Rn-continuous} it follows immediately:
%\begin{proposition}
%\label{prop:MPE}
%Suppose that $(Z,f,\mu)$ is a topologically mixing multidimensional piecewise expanding system and $\mu$ is its a.c.i.p. Let $X_0, X_1, \ldots$ be given by \eqref{eq:def-stat-stoch-proc-DS}, where $\varphi$ achieves a global maximum at some point $\zeta$.  Then we have an EVL for $M_n$ and
%\begin{enumerate}
%
%\item if $\zeta\in\Lambda$ is not a periodic point then the EVL is such that  $\bar H(\tau)=\e^{-\tau}$;
%
%\item if $\zeta\in\Lambda$ is a (repelling) periodic point of prime period $p$ then the EVL is such that $\bar H(\tau)=\e^{-\vartheta\tau}$ where the EI is given by $\vartheta=1-\e^{S_p\psi(\zeta)}$.
%
%\end{enumerate}
%\end{proposition}
%
%\begin{proof}
%For proving (1), we can start by remarking that the condition \eqref{DC:L1} is satisfied since we have \eqref{eq:doc-pen}. Since $U_n$ corresponds to a ball, by definition of $|\cdot|_\alpha$, it follows easily that $\I_{U_n}\in V_\alpha$ and $\|\I_{U_n}\|_{\alpha}$ is uniformly bounded by above.
%Now, considering the definition of $\Lambda$, we can apply Lemma~\ref{prop:Rn-continuous} and consequently obtain that $\lim_{n\to\infty}R_n=\infty$. Then the result follows by applying Theorem~\ref{thm:DC+R=>EVL}.
%
%Regarding the proof of the statement (2), as it is pointed out in \cite[Remark~5]{FFT12}, one can apply the same argument as in \cite[Proposition~2]{FFT12} and get the result since we have decay of correlations with an $L^1$-norm estimate.
%\end{proof}
\begin{proposition}
\label{prop:MPE}
Suppose that $(Z,f,\mu)$ is a topologically mixing multidimensional piecewise expanding system and $\mu$ is its a.c.i.p. Let $X_0, X_1, \ldots$ be given by \eqref{eq:def-stat-stoch-proc-DS}, where $\varphi$ achieves a global maximum at some point $\zeta$.  Then we have an EVL for $M_n$ and
\begin{enumerate}

\item if $\zeta\in\Lambda$ is not a periodic point then the EVL is such that  $\bar H(\tau)=\e^{-\tau}$  and the REPP $N_n$ converges in distribution to a standard Poisson process $N$ of intensity 1;

\item if $\zeta\in\Lambda$ is a (repelling) periodic point of prime period $p$ then the EVL is such that $\bar H(\tau)=\e^{-\vartheta\tau}$ where the EI is given by $\vartheta=1-|\det D(f^{-p})(\zeta)|$ and the REPP $N_n$ converges in distribution to a compound Poisson process $N$ with intensity $\vartheta$ and multiplicity d.f. $\pi$ given by
$
\pi(\kappa)=\vartheta(1-\vartheta)^\kappa,
$
for every $\kappa\in\N_0$.

\end{enumerate}
\end{proposition}

\begin{proof}
Statement (2) has already been proved in \cite[Corollary~4]{FFT12a}.

For proving (1), we can start by remarking that the condition \eqref{DC:L1} is satisfied since we have \eqref{eq:doc-pen}. Since $U_n$ corresponds to a ball, by definition of $|\cdot|_\alpha$, it follows easily that $\I_{U_n}\in V_\alpha$ and $\|\I_{U_n}\|_{\alpha}$ is uniformly bounded by above.
Now, considering the definition of $\Lambda$, we can apply Lemma~\ref{prop:Rn-continuous} and consequently obtain that $\lim_{n\to\infty}R_n=\infty$. The result then follows by applying Theorem~\ref{thm:DC+R=>EVL} and Corollary~\ref{cor:DC+R=>Poisson}.%Regarding the proof of the statement (2), as it is pointed out in \cite[Remark~5]{FFT12}, one can apply the same argument as in \cite[Proposition~2]{FFT12} and get the result since we have decay of correlations with an $L^1$-norm estimate.
\end{proof}

\subsection{The extremal behaviour at discontinuity points}\label{subsec:discontinuity}

 In this section, we go back to Rychlik maps introduced in Section 3.2.1, but with finitely many branches, and study the extremal behaviour of the system when the orbit of $\zeta$ hits a discontinuity point of the map.

Consider a point $\zeta\in Y$. Note that here we consider at most finitely many collection of open intervals such that $\bigcup_i \overline{Y}_i\supset Y$. If $\zeta\in\Lambda$ then we say that $\zeta$ is a \emph{simple point}. If $\zeta$ is a \emph{non-simple point}, which means that $r_{\mathbb S}(\zeta)$ is finite, then let $\ell=r_{\mathbb S}(\zeta)$ and $z=f^\ell(\zeta)$. We will always assume that $z\in \mathbb S$ is such that: there exist $i^+,i^-\in\N$ so that $z$ is the right end point of $Y_{i^-}$ and the left end point of $Y_{i^+}$. We  consider that the point $z$ is doubled and has two versions: $z^+\in Y_{i^+}$ and $z^-\in Y_{i^-}$, so that $f(z^+):=f_{i^+}(z)=\lim_{x\to z,\, x\in Y_{i^+}}f(x)$ and $f(z^-):=f_{i^-}(z)=\lim_{x\to z,\, x\in Y_{i^-}}f(x)$. %\textbf{HA: (32-c) arbitrarily close?? should we make it precise?}. 
%Note that, up to some fixed time, we can always choose these points so that their history coincides with that of $z$ and they never hit $\mathbb S$. 
When $\zeta$ is a non-simple point we consider that its orbit bifurcates when it hits $\mathbb S$ and consider its two possible evolutions. We express this fact by saying that when $\zeta$ is non-simple we consider the ``orbits'' of $\zeta^+$ and $\zeta^-$ which are defined in the following way:
\begin{itemize}
\item for $j=1,\ldots, \ell$ we let $f^j(\zeta^\pm):=f^j(\zeta)$;

\item for $j=\ell+1$, we define $f^j(\zeta^\pm):=f_{i^\pm}(f^{j-1}(\zeta^\pm))$

\item for $j>\ell+1$ we consider two possibilities:

\begin{itemize}

\item if $j-1$ is such that $f^{j-1}(\zeta^\pm)\notin\mathbb S$, then we set $f^j(\zeta^\pm):=f(f^{j-1}(\zeta^\pm))$

\item otherwise we set $f^j(\zeta^\pm):=f_{i}(f^{j-1}(\zeta^\pm))$, where $i$ is such that $f^{j-\ell}(z^\pm)\in Y_i$

\end{itemize}

\end{itemize}

\begin{remark}\label{rem:2-evolutions}
Note that for the ``orbits'' of $\zeta^\pm$ just defined above, there is a sequence $(i_j^\pm)_{j\in\N}$ such that, for all $n\in\N$, we have $f^n(\zeta^\pm)\in \overline Y_{i_n^\pm}$ and $f^n(\zeta^\pm)=f_{i_n^\pm}\circ\cdots\circ f_{i_1^\pm}(\zeta)$. Also observe that, in the notation above, $i^\pm_\ell=i^\pm$.
\end{remark}

A non-simple point $\zeta$ is \emph{aperiodic} if for all $j\in\N$ we have $f^j(\zeta^+)\neq\zeta\neq f^j(\zeta^-)$.

If there exists $p^\pm$ such that $f^{p^\pm}(\zeta^\pm)=\zeta$ and for $j=1,\ldots,p^\pm-1$ we have $f^j(\zeta^\pm)\neq\zeta$, but,  for all $j \in\N$, we have $f^j(\zeta^\mp)\neq\zeta$, then we say that $\zeta$ is \emph{singly returning}. If $\zeta$ is singly returning and $f^\pm(\zeta^\pm)=\zeta^\pm$, which means that $f^{p^\pm}(z^\pm)\in Y_{i^\pm}$, then we say that $\zeta$ is a \emph{singly periodic}  point of period $p^\pm$. If $\zeta$ is singly returning and $f^{p^\pm}(\zeta^\pm)=\zeta^\mp$, which means that $f^{p^\pm}(z^\pm)\in Y_{i^\mp}$, then we say that $\zeta$ is an \emph{eventually aperiodic}  point.

If there exist $p^+$ and $p^-$ such that $f^{p^+}(\zeta^+)=\zeta=f^{p^-}(\zeta^-)$ and for $j=1,\ldots,p^+-1$ and $k=1,\ldots, p^--1$ we have $f^j(\zeta^+)\neq\zeta\neq f^k(\zeta^-)$, then we say that $\zeta$ is \emph{doubly returning}. In the case, $\zeta$ is a doubly returning point and both $f^{p^+}(\zeta^+)=\zeta^+$ and $f^{p^-}(\zeta^-)=\zeta^-$, then we say that $\zeta$ is \emph{doubly periodic} with periods $p^+$ and $p^-$, respectively. If $\zeta$ is doubly returning, $f^{p^\pm}(\zeta^\pm)=\zeta^\pm$ and $f^{p^\mp}(\zeta^\mp)=\zeta^\pm$ then we say that $\zeta$ is \emph{doubly returning with one switch}.  If $\zeta$ is doubly returning, $f^{p^\pm}(\zeta^\pm)=\zeta^\mp$ and $f^{p^\mp}(\zeta^\mp)=\zeta^\pm$ then we say that $\zeta$ is \emph{doubly returning with two switches}.

In what follows consider that
\[
U_n^\pm=U_n\cap f^{-\ell}(Y_{i^\pm}).
\]
The main goals of this section are to compute the EI and also the limit for the REPP at non-simple points as defined above. In the case of aperiodic non-simple points, the analysis is very similar to the one held for non-periodic points, in the previous sections, and we get an EI equal to 1 and the convergence of the REPP to the standard Poisson process. In the case of singly returning and doubly returning points, we have periodicity and consequently clustering. This means that the analysis should follow the footsteps of \cite{FFT12,FFT12a} with the necessary adjustments. For completeness we include a brief review of the results needed in Appendix~\ref{sec:periodicity}  and in particular the formulas \eqref{eq:EI-formula} and \eqref{eq:multiplicity} in Appendix~\ref{sec:EI-multiplicity-formulas}, that we will use to compute the EI and the multiplicity distribution of the limiting compound Poisson process.

\begin{proposition}\label{prop:Rychlik.discontinuous}
Suppose that $(Y,f,\psi)$ is a topologically mixing Rychlik system with finitely many branches, $\psi$ is H\"older continuous on each $\overline Y_i$, and $\mu=\mu_\psi$ is the corresponding equilibrium state.  Let $X_0, X_1, \ldots$ be given by \eqref{eq:def-stat-stoch-proc-DS}, where $\varphi$ achieves a global maximum at some point $\zeta\in Y\setminus \Lambda$.  Let $u_n$ be such that \eqref{eq:un} holds and $U_n$ be defined as in \eqref{def:Un}. % Let $\zeta\in Y\setminus \Lambda$, consider that $X_0, X_1,\ldots$ is defined as in \eqref{eq:def-stat-stoch-proc-DS}, let $u_n$ be such that \eqref{eq:un} holds and assume that $U_n$, defined in \eqref{def:Un}, is such that \eqref{def:zeta} holds.  
We assume that $\mu(U_n^\pm)\sim\alpha^\pm\mu(U_n)$, where $0<\alpha^-,\alpha^+<1$ and $\alpha^-+\alpha^+=1$. Then we have an EVL for $M_n$ and\begin{enumerate}

\item if $\zeta$ is an aperiodic non-simple point then the EVL is such that  $\bar H(\tau)=\e^{-\tau}$;

\item if $\zeta$ is a non-simple, repelling singly returning point then the EVL is such that $\bar H(\tau)=\e^{-\vartheta\tau}$ where the EI is given by $\vartheta=1-\alpha^\pm\e^{S_{p^\pm}\psi(\zeta^\pm)}$;

\item if $\zeta$ is a non-simple, repelling doubly returning point, then the EVL is such that $\bar H(\tau)=\e^{-\vartheta\tau}$ where the EI is given by $\vartheta=1-\alpha^+\e^{S_{p^+}\psi(\zeta^+)}-\alpha^-\e^{S_{p^-}\psi(\zeta^-)}$, when $\zeta$ has no switches; $\vartheta=1-\alpha^{\pm}(\e^{S_{p^+}\psi(\zeta^+)}+\e^{S_{p^-}\psi(\zeta^-)})$, when $\zeta$ has one switch; $\vartheta=1-\alpha^-\e^{S_{p^+}\psi(\zeta^+)}-\alpha^+\e^{S_{p^-}\psi(\zeta^-)}$, when $\zeta$ has two switches.

\end{enumerate}

\end{proposition}

\begin{remark}
\label{rem:Keller-formula}
We remark that, in the particular case when $\mu_\psi$ is absolutely continuous with respect to the Lebesgue measure and the invariant density is continuous at the points $\zeta$ considered in the proposition above,  the formulas for the EI can be seen as special cases of the formula in \cite[Remark~8]{K12}. 
\end{remark}

\begin{proof}

If $\zeta$ is an aperiodic non-simple point then we just have to mimic the argument for non-periodic points in the previous sections. The proof of $D_2(u_n)$ is done exactly as before. Using decay of correlations against $L^1$, stated in \ref{eq:Rychlik-DC}, the proof that $D'(u_n)$ holds for these points follows the same footsteps except for the adjustments in order to consider the two possible evolutions corresponding to the ``orbits'' of $\zeta^+$ and $\zeta^-$. For example, to prove that $R(U_n)\to\infty$, as $n\to \infty$, in the argument of  Lemma~\ref{prop:Rn-continuous} we would define $$\epsilon=\min\left\{\min_{k=1,\ldots j}\dist(f^k(\zeta^+),\zeta), \min_{k=1,\ldots j}\dist(f^k(\zeta^-),\zeta)\right\},$$ 
and proceed as before.
%
%Now, use the fact that by Remark~\ref{rem:2-evolutions}, for all $k\in\N$, $f^k(\zeta^\pm)$ corresponds to the $k$-fold composition of continuous functions, in order to choose $n\in\N$ large enough so that $U_n$ is sufficiently small to have that by continuity, for all $k\leq j$, both $f_{i_k^+}\circ\cdots\circ f_{i_1^+}(U_n^+)$ and $f_{i_k^-}\circ\cdots\circ f_{i_1^-}(U_n^-)$ are inside an $\epsilon/2$-ball around $f^k(\zeta^+)$ and $f^k(\zeta^-)$, respectively. This provides that $R_n>j$. As $j$ is arbitrary the statement follows.

When $\zeta$ is a non-simple (singly or doubly) returning point, we just need to adjust the definition \eqref{eq:def-Qp} of $Q_p(u_n)$ to cope with the two possibly different evolutions of $\zeta^+$ and $\zeta^-$. Everything else, namely the proofs of conditions $D^p(u_n)$ and $D'_p(u_n)$ follow from decay of correlations against $L^1$, stated in \ref{eq:Rychlik-DC}, exactly  in the same lines as in the proof of \cite[Theorem~2]{FFT12a}. Hence, essentially, for each different case we have to define coherently $Q_p(u_n)$ and compute the EI using formula \eqref{eq:EI-formula}.

%Let's consider now statement (2). 
Assume first that $\zeta$ is a singly returning (eventually aperiodic or not) non-simple point. Without loss of generality (w.l.o.g.), we also assume that there exists $p$ such that $f^p(\zeta^+)=\zeta^+$. In this case, we should define $Q_p(u_n)=U_n^-\cup (U_n^+\setminus f^{-p}(U_n))$. %In this case we would have to adapt the argument in \cite[Proposition~2]{FFT12} as follows. 
%Consider that $$Q_p(u_n)=\{X_0>u_n\cap X_p\leq u_n\}=U_n^-\cup(U_n^+\setminus f^{-p}(U_n^+))$$. As in \cite[Theorem~5]{FFT12} it is easy to check that condition $S\!P_{p,\vartheta}(u_n)$ holds with 
We can now compute the EI:
%$\vartheta=\lim_{n\to\infty}\frac{\mu(Q_p(u_n))}{\mu(U_n)}=1-\alpha^+\e^{S_{p}\psi(\zeta^+)}$ since
\begin{align*}
%\label{eq:Qp-estimate}
\vartheta&=\lim_{n\to\infty}\frac{\mu(Q_p(u_n))}{\mu(U_n)}=\lim_{n\to\infty}\frac{\mu(U_n^-)+(1-\e^{S_{p}\psi(\zeta^+)})\mu(U_n^+)}{\mu(U_n)}\\
&=\lim_{n\to\infty}\frac{\alpha^-\mu(U_n)+\alpha^+(1-\e^{S_{p}\psi(\zeta^+)})\mu(U_n)}{\mu(U_n)}=1-\alpha^+\e^{S_{p}\psi(\zeta^+)}.
%\mu(Q_p(u_n))&=\mu(U_n^-)+\mu(U_n^+\setminus f^{-p}(U_n^+))\sim \mu(U_n^-)+(1-\e^{S_{p}\psi(\zeta^+)})\mu(U_n^+)\\
%&\sim \alpha^-\mu(U_n)+\alpha^+(1-\e^{S_{p}\psi(\zeta^+)})\mu(U_n)\nonumber.
\end{align*}

Let $\zeta$ be a non-simple, repelling doubly returning point  and $p^-,p^+$ be such that $f^{p^-}(\zeta^-)=\zeta$ and $f^{p^+}(\zeta^+)=\zeta$ . For definiteness, we assume w.l.o.g.\ that $p^-<p^+$.  

First we consider the case where no switching occurs. In this case, we have two different ``periods'', hence we should define $Q_{p^-,p^+}(u_n)=\left(U_n^-\setminus f^{-p^-}(U_n^-)\right)\cup\left(U_n^+\setminus f^{-p^+}(U_n^+)\right)$. It follows that
\begin{align*}
\vartheta&=\lim_{n\to\infty}\frac{\alpha^-(1-\e^{S_{p}\psi(\zeta^-)})\mu(U_n)+\alpha^+(1-\e^{S_{p}\psi(\zeta^+)})\mu(U_n)}{\mu(U_n)}=1-\alpha^-\e^{S_{p}\psi(\zeta^-)}-\alpha^+\e^{S_{p}\psi(\zeta^+)}.
%\mu(Q_p(u_n))&=\mu(U_n^-)+\mu(U_n^+\setminus f^{-p}(U_n^+))\sim \mu(U_n^-)+(1-\e^{S_{p}\psi(\zeta^+)})\mu(U_n^+)\\
%&\sim \alpha^-\mu(U_n)+\alpha^+(1-\e^{S_{p}\psi(\zeta^+)})\mu(U_n)\nonumber.
\end{align*}

Next, we consider the case with one switch. In this case, we also have two different ``periods'' and for definiteness we assume w.l.o.g.\ that $f^{p^-}(\zeta^-)=\zeta^-$ and $f^{p^+}(\zeta^+)=\zeta^-$. Then we
define $Q_{p^-,p^+}(u_n)=\left(U_n^-\setminus f^{-p^-}(U_n^-)\right)\cup\left(U_n^+\setminus f^{-p^+}(U_n^-)\right)$. It follows that
\begin{align*}
\vartheta&=\lim_{n\to\infty}\frac{\alpha^-(1-\e^{S_{p}\psi(\zeta^-)})\mu(U_n)+\alpha^-(1-\e^{S_{p}\psi(\zeta^+)})\mu(U_n)}{\mu(U_n)}=1-\alpha^-\e^{S_{p}\psi(\zeta^-)}-\alpha^-\e^{S_{p}\psi(\zeta^+)}.
%\mu(Q_p(u_n))&=\mu(U_n^-)+\mu(U_n^+\setminus f^{-p}(U_n^+))\sim \mu(U_n^-)+(1-\e^{S_{p}\psi(\zeta^+)})\mu(U_n^+)\\
%&\sim \alpha^-\mu(U_n)+\alpha^+(1-\e^{S_{p}\psi(\zeta^+)})\mu(U_n)\nonumber.
\end{align*}

Finally, we consider the case with two switches. In this case, we also have two different ``periods'' and we should define $Q_{p^-,p^+}(u_n)=\left(U_n^-\setminus f^{-p^-}(U_n^+)\right)\cup\left(U_n^+\setminus f^{-p^+}(U_n^-)\right)$. It follows that
\begin{align*}
\vartheta&=\lim_{n\to\infty}\frac{\alpha^+(1-\e^{S_{p}\psi(\zeta^-)})\mu(U_n)+\alpha^-(1-\e^{S_{p}\psi(\zeta^+)})\mu(U_n)}{\mu(U_n)}=1-\alpha^-\e^{S_{p}\psi(\zeta^+)}-\alpha^+\e^{S_{p}\psi(\zeta^-)}.
%\mu(Q_p(u_n))&=\mu(U_n^-)+\mu(U_n^+\setminus f^{-p}(U_n^+))\sim \mu(U_n^-)+(1-\e^{S_{p}\psi(\zeta^+)})\mu(U_n^+)\\
%&\sim \alpha^-\mu(U_n)+\alpha^+(1-\e^{S_{p}\psi(\zeta^+)})\mu(U_n)\nonumber.
\end{align*}

\end{proof}
Next result gives the convergence of the REPP at non-simple points. Note that, contrarily to the usual geometric distribution obtained, for example, in \cite{HV09, CCC09, FFT12a}, in here, the multiplicity distribution is quite different. In fact, for eventually aperiodic returning points, for example, we have that $\pi(\kappa)=0$ for all $\kappa\geq 3$. 

\begin{proposition}\label{prop:Rychlik.discontinuous-REPP}
Let  $a^\pm:=\e^{S_{p^\pm}\psi(\zeta^\pm)}$. Under the same assumptions of Proposition~\ref{prop:Rychlik.discontinuous}, we have:  \begin{enumerate}

\item if $\zeta$ is an aperiodic non-simple point then the REPP converges to a standard Poisson process of intensity 1;

\item if $\zeta$ is a non-simple, singly returning point 
\begin{enumerate}
\item not eventually aperiodic then the REPP converges to a compound Poisson process of intensity $\vartheta$, given in Proposition~\ref{prop:Rychlik.discontinuous}, and multiplicity distribution defined by:
%$$
%\pi(1)=\frac{(1-\alpha^\pm a^\pm)-\alpha^\pm a^\pm(1-a^\pm)}{1-\alpha^\pm a^\pm}
%$$
$$
\pi(1)=\tfrac{\vartheta-(1-\vartheta)(1-a^\pm)}{\vartheta}, \quad\pi(\kappa)=\tfrac{\alpha^\pm(1-a^\pm)^2 (a^\pm)^{-(\kappa-1)}}{\vartheta},\quad\kappa\geq2.
$$

\item eventually aperiodic then the REPP converges to a compound Poisson process of intensity  $\vartheta$, given in Proposition~\ref{prop:Rychlik.discontinuous}, and multiplicity distribution defined by:
$$
\pi(1)=\tfrac{2\vartheta-1}{\vartheta}, \quad\pi(2)=\tfrac{1-\vartheta}{\vartheta},\quad \pi(\kappa)=0,\quad\kappa\geq3.
$$
\end{enumerate}

\item if $\zeta$ is a non-simple, repelling doubly returning point
\begin{enumerate}
\item with no switches then the REPP converges to a compound Poisson process of intensity  $\vartheta$, given in Proposition~\ref{prop:Rychlik.discontinuous}, and multiplicity distribution defined by:
\begin{align*}
%&\pi(1)=\tfrac{1-\alpha^- a^- (2-a^-)-\alpha^+ a^+ (2-a^+)}{\vartheta}, \\
&\pi(1)=\tfrac{2\vartheta-1+\alpha^- (a^-)^2 +\alpha^+ (a^+)^2 }{\vartheta}\\
&\pi(\kappa)=\tfrac{\alpha^-(1-a^-)^2 (a^-)^{-(\kappa-1)}+\alpha^-(1-a^+)^2 (a^+)^{-(\kappa-1)}}{\vartheta},\;\kappa\geq2.
\end{align*}

\item with one switch then the REPP converges to a compound Poisson process of intensity  $\vartheta$, given in Proposition~\ref{prop:Rychlik.discontinuous}, and multiplicity distribution defined by:
%$$
%\pi(1)=\tfrac{1-\alpha^\pm \left(a^- (2-a^-)+ a^+ (2-a^+)\right)}{\vartheta}, \;\pi(\kappa)=\tfrac{\alpha^\pm(1-a^\pm)^2 (a^\pm)^{-(\kappa-2)}(a^\pm+a^\mp)}{\vartheta},\;\kappa\geq2.
%$$
$$
\pi(1)=\tfrac{2\vartheta-1+ a^\pm(1-\vartheta)}{\vartheta}, \;\pi(\kappa)=\tfrac{(1-\vartheta)(a^\pm)^{\kappa-2}(1-a^\pm)^2}{\vartheta},\;\kappa\geq2.
$$

\item with two switches then the REPP converges to a compound Poisson process of intensity  $\vartheta$, given in Proposition~\ref{prop:Rychlik.discontinuous}, and multiplicity distribution defined by:
%\begin{align*}
%&\pi(1)=\tfrac{1-2\alpha^+a^--2\alpha^-a^++a^-a^+}{\vartheta}, \\
%&\pi(2j)=\tfrac{(a^-a^+)^{j-1}\left(\alpha^+a^-(1+a^-a^+)+\alpha^-a^+(1+a^-a^+)-2a^-a^+\right)}{\vartheta},\;j\geq1,\\
%&\pi(2j+1)=\tfrac{(a^-a^+)^{j}\left(1-2\alpha^+a^--2\alpha^-a^++a^-a^+\right)}{\vartheta},\;j\geq1.
%\end{align*}
\begin{align*}
&\pi(1)=\tfrac{1-2(1-\vartheta)+a^-a^+}{\vartheta}, \quad
\pi(2j)=\tfrac{(a^-a^+)^{j-1}\left((1-\vartheta)(1+a^-a^+)-2a^-a^+\right)}{\vartheta},\\
&\pi(2j+1)=\tfrac{(a^-a^+)^{j}\left(1-2(1-\vartheta)+a^-a^+\right)}{\vartheta},\;j\geq1.
\end{align*}
\end{enumerate}

\end{enumerate}
\end{proposition}

\begin{proof}
When $\zeta$ is an aperiodic non-simple point then as we have seen in Proposition~\ref{prop:Rychlik.discontinuous}, condition $D'(u_n)$ holds. Clearly, $D_3(u_n)$ follows from decay of correlations and, by \cite[Theorem~5]{FFT10}, we easily conclude that the REPP converges to the standard Poisson process of intensity 1.

When $\zeta$ is a non-simple (singly or doubly) returning point, we just need to adjust the definition of the sets $U^{(\kappa)}$ given in \eqref{eq:Uk-definition}, which ultimately affects the sets $Q_p^\kappa(u)$, given in \eqref{eq:Q-definition},  in order to cope with the two possibly different evolutions of $\zeta^+$ and $\zeta^-$. Everything else, namely the proofs of conditions $D^p(u_n)^*$ and $D'_p(u_n)^*$ follow from decay of correlations against $L^1$, stated in \ref{eq:Rychlik-DC}, exactly  in the same lines as in the proof of \cite[Theorem~2]{FFT12a}. Hence, essentially, for each different case we have to define coherently the sets $U^{(\kappa)}$ and compute the multiplicity distribution  using formula \eqref{eq:multiplicity}.

In all cases, $U^{(0)}=U_n=U_n^-\cup U_n^+$.

Let $\zeta$ be a singly returning non-simple point which is not eventually aperiodic. We assume w.l.o.g.\ that there exist $p$ such that $f^{p}(\zeta^+)=\zeta^+$ and $f^{j}(\zeta^-)\neq\zeta$, for all $j\in\N$. For every $\kappa\in\N$, we define
$$
U^{(\kappa)}:=\left(\bigcap_{i=0}^{\kappa}f^{-ip}(U_n^+)\right)
$$
Using  \eqref{eq:Q-definition}, we can now easily define $Q^\kappa:=U^{(\kappa)}\setminus U^{(\kappa+1)}$, for all $\kappa\geq0$. We have $\p(Q^0)\sim\p(U_n)-a^+\p(U_n^+)\sim\p(U_n)(1-\alpha^+a^+)$. The same computations would lead us to $\p(Q^\kappa)\sim \p(U_n)(\alpha^+(1-a^+)(a^+)^\kappa)$. Using formula \eqref{eq:multiplicity}, it follows:
\begin{align*}
\pi(1)&=\lim_{n\to\infty}\tfrac{\p\left(Q^0\right)-\p\left(Q^1\right)}{\p\left(Q^0\right)}=\tfrac{(1-\alpha^+a^+)-(\alpha^+(1-a^+)a^+)}{(1-\alpha^+a^+)}=\tfrac{\vartheta-(1-\vartheta)(1-a^\pm)}{\vartheta}\\
\pi(\kappa)&=\lim_{n\to\infty}\tfrac{\p\left(Q^{\kappa-1}\right)-\p\left(Q^{\kappa}\right)}{\p\left(Q^0\right)}=\tfrac{(\alpha^+(1-a^+)(a^+)^{\kappa-1})-(\alpha^+(1-a^+)(a^+)^\kappa)}{(1-\alpha^+a^+)}=\tfrac{\alpha^+(1-a^+)^2(a^+)^{\kappa-1}}{\vartheta}.\end{align*}

Let $\zeta$ be a singly returning non-simple point which is eventually aperiodic. We assume w.l.o.g. that there exist $p$ such that $f^{p}(\zeta^+)=\zeta^-$ and $f^{j}(\zeta^-)\neq\zeta$, for all $j\in\N$. For every $\kappa\in\N$, we define
$$
U^{(1)}:=\left(U_n^+\cap f^{-ip}(U_n^-)\right),\quad U^{(\kappa)}:=\emptyset,\quad \kappa\geq 2
$$
Note that $Q^0=U_n\setminus U^{(1)}$, $Q^1=U^{(1)}$ and $Q^\kappa=\emptyset$, for all $\kappa\geq2$. We have $\p(Q^0)\sim\p(U_n)-a^+\p(U_n^-)\sim\p(U_n)(1-\alpha^-a^+)$, $\p(Q^1)\sim \p(U_n)\alpha^-a^+$, $\p(Q^\kappa)=0$, for all $\kappa\geq 2$. Using formula \eqref{eq:multiplicity}, it follows:
\begin{align*}
\pi(1)&=\lim_{n\to\infty}\tfrac{\p\left(Q^0\right)-\p\left(Q^1\right)}{\p\left(Q^0\right)}=\tfrac{(1-\alpha^-a^+)-(\alpha^-a^+)}{(1-\alpha^-a^+)}=\tfrac{2\vartheta-1}{\vartheta}\\
\pi(2)&=\lim_{n\to\infty}\tfrac{\p\left(Q^{1}\right)-\p\left(Q^{2}\right)}{\p\left(Q^0\right)}=\tfrac{\alpha^-a^+}{(1-\alpha^-a^+)}=\tfrac{1-\vartheta}{\vartheta}, \quad \pi(\kappa)=0,\quad k\geq 3.\end{align*}

Let $\zeta$ be a doubly returning non-simple point with no switches. Let $p^-, p^+$ be such that $f^{p^-}(\zeta^-)=\zeta^-$ and $f^{p^+}(\zeta^+)=\zeta^+$. For every $\kappa\in\N$, we define
$$
U^{(\kappa)}:=\left(\bigcap_{i=0}^{\kappa}f^{-ip^-}(U_n^-)\right)\bigcup\left(\bigcap_{i=0}^{\kappa}f^{-ip^+}(U_n^+)\right).
$$
Note that using  \eqref{eq:Q-definition}, we can now easily define $Q^\kappa:=U^{(\kappa)}\setminus U^{(\kappa+1)}$, for all $\kappa\geq0$. We have $\p(Q^0)\sim\p(U_n)-a^-\p(U_n^-)-a^+\p(U_n^+)\sim\p(U_n)(1-\alpha^-a^--\alpha^+a^+)$. The same computations would lead us to $\p(Q^\kappa)\sim \p(U_n)(\alpha^-(1-a^-)(a^-)^\kappa+\alpha^+(1-a^+)(a^+)^\kappa)$. Using formula \eqref{eq:multiplicity}, it follows:
\begin{align*}
\pi(1)&=\lim_{n\to\infty}\tfrac{\p\left(Q^0\right)-\p\left(Q^1\right)}{\p\left(Q^0\right)}=\tfrac{(1-\alpha^-a^--\alpha^+a^+)-(\alpha^-(1-a^-)a^-+\alpha^+(1-a^+)a^+)}{(1-\alpha^-a^--\alpha^+a^+)}=\tfrac{2\vartheta-1+\alpha^- (a^-)^2 +\alpha^+ (a^+)^2 }{\vartheta}\\
\pi(\kappa)&=\lim_{n\to\infty}\tfrac{\p\left(Q^{\kappa-1}\right)-\p\left(Q^{\kappa}\right)}{\p\left(Q^0\right)}=\tfrac{\alpha^-(1-a^-)^2(a^-)^{\kappa-1}+\alpha^+(1-a^+)^2(a^+)^{\kappa-1}}{\vartheta}.\end{align*}

Let $\zeta$ be a doubly returning non-simple point with one switch. We assume w.l.o.g.\ that there exist $p^-, p^+$ such that $f^{p^-}(\zeta^-)=\zeta^-$ and $f^{p^+}(\zeta^+)=\zeta^-$. For every $\kappa\in\N$, we define
$$
U^{(\kappa)}:=\left(\bigcap_{i=0}^{\kappa}f^{-ip^-}(U_n^-)\right)\bigcup\left(U_n^+\cap f^{-p^+}(U_n^-)\cap\bigcap_{i=0}^{\kappa}f^{-p^+-ip^-}(U_n^-)\right).
$$
Using  \eqref{eq:Q-definition}, we can now easily define $Q^\kappa:=U^{(\kappa)}\setminus U^{(\kappa+1)}$, for all $\kappa\geq0$. We have $\p(Q^0)\sim\p(U_n)-a^-\p(U_n^-)-a^+\p(U_n^-)\sim\p(U_n)(1-\alpha^-a^--\alpha^-a^+)$. The same computations would lead us to $\p(Q^\kappa)\sim \p(U_n)(\alpha^-(1-a^-)(a^-)^\kappa+\alpha^-(1-a^-)a^+(a^-)^{\kappa-1})$. Using formula \eqref{eq:multiplicity}, it follows:
\begin{align*}
\pi(1)&=\lim_{n\to\infty}\tfrac{\p\left(Q^0\right)-\p\left(Q^1\right)}{\p\left(Q^0\right)}=\tfrac{(1-\alpha^-(a^-+a^+))-(\alpha^-(1-a^-)a^-+\alpha^-(1-a^-)a^+)}{(1-\alpha^-a^--\alpha^-a^+)}=\tfrac{2\vartheta-1+ a^-(1-\vartheta)}{\vartheta}\\
\pi(\kappa)&=\lim_{n\to\infty}\tfrac{\p\left(Q^{\kappa-1}\right)-\p\left(Q^{\kappa}\right)}{\p\left(Q^0\right)}=\tfrac{\alpha^-(1-a^-)^2(a^-)^{\kappa-1}+\alpha^-(1-a^-)^2a^+(a^-)^{\kappa-2}}{\vartheta}=\tfrac{(1-\vartheta)(a^-)^{\kappa-2}(1-a^-)^2}{\vartheta}.\end{align*}

Let $\zeta$ be a doubly returning non-simple point with two switches. We assume that there exist $p^-, p^+$ such that $f^{p^-}(\zeta^-)=\zeta^+$ and $f^{p^+}(\zeta^+)=\zeta^-$. For every $j\in\N_0$, we define
\begin{align*}
U^{(2j+1)}:=\left(U_n^-\cap\bigcap_{i=1}^{j+1}f^{-ip^--(i-1)p^+}(U_n^+)\cap\bigcap_{i=1}^{j}f^{-ip^--ip^+}(U_n^-)\right)\\\bigcup\left(U_n^+\cap\bigcap_{i=1}^{j+1}f^{-ip^+-(i-1)p^-}(U_n^-)\cap\bigcap_{i=1}^{j}f^{-ip^+-ip^-}(U_n^+)\right)\\
U^{(2j)}:=\left(U_n^-\cap\bigcap_{i=1}^{j}f^{-ip^--(i-1)p^+}(U_n^+)\cap\bigcap_{i=1}^{j}f^{-ip^--ip^+}(U_n^-)\right)\\\bigcup\left(U_n^+\cap\bigcap_{i=1}^{j}f^{-ip^+-(i-1)p^-}(U_n^-)\cap\bigcap_{i=1}^{j}f^{-ip^+-ip^-}(U_n^+)\right)\\
\end{align*}
Using  \eqref{eq:Q-definition}, we can now easily define $Q^\kappa:=U^{(\kappa)}\setminus U^{(\kappa+1)}$, for all $\kappa\geq0$. We have $\p(Q^0)\sim\p(U_n)-a^-\p(U_n^+)-a^+\p(U_n^-)\sim\p(U_n)(1-\alpha^+a^--\alpha^-a^+)$. The same computations would lead us to $\p(Q^{2j})\sim \p(U_n)(1-\alpha^+a^--\alpha^-a^+)(a^-a^+)^j$ and $\p(Q^{2j})\sim \p(U_n)(\alpha^+a^-+\alpha^-a^+-a^-a^+)(a^-a^+)^j$. Using formula \eqref{eq:multiplicity}, it follows that, for every $j\in\N$:
\begin{align*}
&\pi(1)=\lim_{n\to\infty}\tfrac{\p\left(Q^0\right)-\p\left(Q^1\right)}{\p\left(Q^0\right)}=\tfrac{(1-\alpha^+a^--\alpha^-a^+)-(\alpha^+a^-+\alpha^-a^+-a^-a^+)}{(1-\alpha^+a^--\alpha^-a^+)}=\tfrac{1-2(1-\vartheta)+a^-a^+}{\vartheta}\\
&\pi(2j)=\lim_{n\to\infty}\tfrac{\p\left(Q^{2j-1}\right)-\p\left(Q^{2j}\right)}{\p\left(Q^0\right)}=\tfrac{(a^-a^+)^{j-1}(\alpha^+a^-(1+a^-a^+)+\alpha^-a^+(1+a^-a^+)-2a^-a^+)}{\vartheta}\\
&\pi(2j+1)=\lim_{n\to\infty}\tfrac{\p\left(Q^{2j}\right)-\p\left(Q^{2j+1}\right)}{\p\left(Q^0\right)}=\tfrac{(a^-a^+)^{j}(1-2\alpha^+a^--2\alpha^-a^++a^-a^+)}{\vartheta}.
\end{align*}

\end{proof}

\section{Extremes for random dynamics}
\label{sec:random}
In this section we will start with the proof of Theorem~\ref{thm:random-EVL} which states that the dichotomy observed in Section~\ref{sec:dichotomy} vanishes when we add absolutely continuous noise (w.r.t.\ Lebesgue) and for every chosen point in the phase space we have a standard exponential distribution for the EVL and HTS/RTS weak limits. We will also certify that the REPP converges to a Poisson Process with intensity $1$. Next, we will give some examples of random dynamical systems for which we can prove the existence of EVLs and HTS/RTS as well as the convergence of REPP.

In what follows, we denote the diameter of set $A\subset\mathcal M$ by $|A|:=\sup\{\dist(x,y):\;x,y\in A\}$ and for any $x\in \mathcal M$ we define the translation of $A$ by $x$ as the set: $A+x:=\{a+x:\;a\in A\}$.

\subsection{Laws of rare events for random dynamics}\label{subsec:random-EVL}

\begin{proof}[Proof of Theorem~\ref{thm:random-EVL}]

Firstly, we want to show that, as in the deterministic case, the condition $D_2(u_n)$ can be deduced from the decay of correlations.

From our assumption the random dynamical system has (annealed) decay of correlations, \ie  there exists a Banach space $\mathcal C$ of real valued functions such that 
for all $\phi\in\mathcal C$ and $\psi\in L^1(\mu_\varepsilon)$,
\begin{equation}
\label{RDC:L12} %\label{eq:adc}
\Big|\int\phi ({\mathcal
U}^t_{\varepsilon} \psi)(x)\, \dif\mu_{\varepsilon}-\int\phi\,
\dif\mu_\varepsilon\int\psi\, \dif\mu_\varepsilon\Big|\le C
\|\phi\|_{\mathcal C}\|\psi\|_{L^1(\mu_\varepsilon)} t^{-2}
\end{equation}
where $C>0$ is a constant independent of both $\varphi$ and $\psi$.

In proving $D_2(U_n)$, the main point is to choose the right observable. We take
\begin{equation*}
\phi(x)=\I_{\{ X_0>u_n \}}=\I_{\{\varphi(x)>u_n\}} \label{def:obs_phi}, \;
\psi(x)=\displaystyle\int \I_{\{ \varphi(x),\, \varphi\circ f_{{\tilde\omega}_1}(x), \,\ldots\,,\, \varphi\circ f^{\ell-1}_{\tilde\o}(x)\leq u_n\}}\, \dif\theta_\varepsilon^{\ell-1} (\tilde\o).
\end{equation*}

Substituting $\psi$ in the random evolution operator, we get
\begin{align*}
({\mathcal U}^t_{\varepsilon}
\psi)(x)=\iint\I_{\{\varphi\circ f^{t}_{\o}(x),\, \ldots\, ,\,
\varphi\circ f^{\ell-1}_{\tilde\o}\circ f^{t}_{\o}(x) \leq
u_n\}}\,\dif\theta^{\ell-1}_\varepsilon (\tilde\o)
\dif\theta^{t}_\varepsilon (\o).
\end{align*}

Since all $\omega_i$'s and $\tilde\omega_j$'s are chosen in an i.i.d.\ structure, we can rename the random iterates, \ie we lose no information in writing
\begin{equation*}
({\mathcal U}^t_{\varepsilon} \psi)(x)=\displaystyle\int
\I_{\{\varphi\circ f^{t}_{\o}(x),\,\ldots\, ,\, \varphi\circ f^{t+\ell-1}_{\o}(x)
\leq u_n\}} \dif\theta^ \N_\varepsilon (\o).
\end{equation*}

Therefore, we get
\begin{eqnarray*}
\int\phi(x)\, ({\mathcal U}^t_{\varepsilon} \psi)(x)
\,\dif\mu_\varepsilon=\int\mu_\varepsilon\Big(\varphi(x)>u_n,\varphi\circ f^{t}_{\o}(x)\leq
u_n,\,\ldots,\varphi\circ f^{t+\ell-1}_{\o}(x) \leq u_n\Big)\dif\theta^
\N_\varepsilon(\o).
\end{eqnarray*}

On the other hand,
\begin{eqnarray*}
\int\phi(x) \dif\mu_\varepsilon&=& \mu_\varepsilon(X_0(x)>u_n)=\int \mu_{\varepsilon}(X_0(x)>u_n)\, \dif\theta^\N_\varepsilon (\o) \\
\int\psi(x) \dif\mu_{\varepsilon}&=& \displaystyle\int \Big(\int \I_{\{\varphi(x), \varphi\circ f_{\omega_1}(x),\,\ldots\,,\varphi\circ f^{\ell-1}_{\o}(x)\leq u_n \}} \dif\mu_{\varepsilon}\Big)\dif\theta^{\N}_{\varepsilon} (\o)\\
&=&\displaystyle\int \mu_{\varepsilon}\Big(\varphi(x)\leq u_n, \varphi\circ f_{\omega_1}(x)\leq u_n,\,\ldots\,,\varphi\circ f^{\ell-1}_{\o}(x)\leq u_n \Big)\dif\theta^{\N}_{\varepsilon} (\o).
\end{eqnarray*}

Now, the decay of correlations can be written as
\begin{multline*}
\hspace{-.2cm}\Big|\int\mu_\varepsilon\big(X_0(x)>u_n,\varphi\circ f^{t}_{\o}(x)\leq u_n,\,\ldots\,, \varphi\circ f^{t+\ell-1}_{\o}(x) \leq u_n \big)\, \dif\theta^ \N_\varepsilon (\o)-\\
\int\mu_{\varepsilon}(\varphi(x)>u_n)\, \dif\theta^ \N_\varepsilon (\o) \int\mu_{\varepsilon}\big(\varphi(x)\leq u_n, \varphi\circ f_{\omega_1}(x)\leq u_n,\,\ldots\,, \varphi\circ f_{\o}^{\ell-1}(x)\leq u_n \big)\, \dif\theta^ \N_\varepsilon (\o)\Big|\\
\leq C \|\phi\|_{\mathcal C}\|\psi\|_{L^1(\mu_\varepsilon)} t^{-2}
\end{multline*}

which leads us to the conclusion that the condition $D_2(u_n)$ holds with
\begin{equation}
\gamma(n,t)=\gamma(t)=C^{*}t^{-2}
\end{equation}
for some $C^{*}>0$ and $t_n=n^{\beta}$, with $1/2<\beta<1$.

For proving $D'(u_n)$, the basic idea is to use the fact that we have decay of correlations against $L^1$ as in Theorem ~\ref{thm:DC+R=>EVL} and then to show that except for a small set of $\o$'s, $R^{\o}(U_n)$ grows at a sufficiently fast rate. Hence, we split $\Omega$ into two parts: the $\o$'s for which $R^{\o}(U_n)>\alpha_n$, where $(\alpha_n)_n$ is some sequence such that
\begin{equation}
\label{eq:alpha-n}
\alpha_n\to\infty\qquad\mbox{and }\qquad\alpha_n=o(\log k_n),
\end{equation} which is designed, on one hand, to guarantee that for the $\o$'s for which $R^{\o}(U_n)>\alpha_n$, the argument using decay of correlations against $L^1$ is still applicable and, on the other hand, the set of the $\o$'s for which $R^{\o}(U_n)\leq\alpha_n$ has $\theta_\eps^\N$ small measure. To show the latter we make an estimate on the $\o$'s that take the orbit of $\zeta$ too close to itself.

First, note that since $f$ is continuous (which implies that $f_{\o}^j$ is also continuous for all $j\in\N$) and $\eta$ is the highest rate at which points can separate, the diameter of $f_{\o}^j(U_n)$ grows at most at a rate given by $\eta^j$, so, for any $\o\in\Omega$ we have $|f_{\o}^j(U_n)|\leq \eta^j |U_n|$. This implies that
\begin{equation}
\label{eq:fact1}
\mbox{if $\dist(f_{\o}^j(\zeta),\zeta)>2\eta^j|U_n|>|U_n|+\eta^j|U_n|$ then $f_{\o}^j(U_n)\cap U_n=\emptyset$}.
\end{equation}
Note that, by equation \eqref{eq:fact1}, if for all $j=1,\ldots,\alpha_n$ we have $\dist(f_{\o}^j(\zeta),\zeta)>2\eta^j|U_n|$ then clearly $R^{\o}(U_n)>\alpha_n$. Hence, we may write that $\big\{\o:\; R^{\o}(U_n)\leq\alpha_n\big\}\subset \bigcup_{j=1}^{\alpha_n} \big\{\o:\; f_{\o}^j(\zeta)\in B_{2\eta^j|U_n|}(\zeta)\big\}.$
It follows that, there exists some $C>0$ such that
\begin{align*}
\theta_\eps^\N\big(\big\{\o:\; R^{\o}(U_n)&\leq\alpha_n\big\}\big)\leq \sum_{j=1}^{\alpha_n} \int \theta_\eps\Big(\Big\{\omega_j: f\left(f_{\o}^{j-1}(\zeta)\right)+\omega_j\in  B_{2\eta^j|U_n|}(\zeta) \Big\}\Big)\dif\theta^\N_\eps\\
&= \sum_{j=1}^{\alpha_n}\int\theta_\eps\Big(\Big\{\omega_j: \omega_j\in  B_{2\eta^j|U_n|}(\zeta)-f\left(f_{\o}^{j-1}(\zeta)\right)\Big\}\Big)\dif\theta^\N_\eps\\
&= \sum_{j=1}^{\alpha_n} \iint_{B_{2\eta^j|U_n|}(\zeta)-f\left(f_{\o}^{j-1}(\zeta)\right)}g_\eps(x) \dif\l\,\dif\theta^\N_\eps\\
&\leq \sum_{j=1}^{\alpha_n} \overline{g_\eps}\l\left(B_{2\eta^j|U_n|}(\zeta)\right)=\sum_{j=1}^{\alpha_n} \overline{g_\eps} C\eta^j\l(U_n)\leq C\overline{g_\eps} \l(U_n)\frac{\eta}{\eta-1}\eta^{\alpha_n}.
\end{align*}

%\textbf{HA: (39) the second and the third inequalities must be equalities? Also, I changed the size of the curly brackets.}

Now, observe that
\begin{align*}
n\sum_{j=1}^{\lfloor n/k_n \rfloor}\p(U_n\cap f_{\o}^{-j}(U_n))&%=%n\sum_{j=1}^{\lfloor n/k_n \rfloor}\p\Big(\big\{(x,\o):x\in U_n,\, f_{\o}^j(x)\in U_n\big\}\Big)\\
%&=n\sum_{j=1}^{\lfloor n/k_n \rfloor}\p\Big(\big\{(x,\o):x\in U_n,\, f_{\o}^j(x)\in U_n, R^{\o}(U_n)>\alpha_n\big\}\Big)\\
%&\quad+n\sum_{j=1}^{\lfloor n/k_n \rfloor}\p\Big(\big\{(x,\o):x\in U_n,\, f_{\o}^j(x)\in U_n, R^{\o}(U_n)\leq\alpha_n\big\}\Big)\\
\leq n\sum_{j=\alpha_n}^{\lfloor n/k_n \rfloor}\p\Big(\big\{(x,\o):x\in U_n,\, f_{\o}^j(x)\in U_n\big\}\Big)\\
&\quad+n\sum_{j=1}^{\lfloor n/k_n \rfloor}\p\Big(\big\{(x,\o):x\in U_n, R^{\o}(U_n)\leq\alpha_n\big\}\Big):=I+I\!I.%\textbf{HA: (40) $=I+II$}
\end{align*}
Let us start by estimating $I$, %\textbf{HA: (40) $I$ instead of "the first term"??}, 
which will be dealt as in Section~\ref{sec:dichotomy}. Taking $\psi=\phi=\I_{U_n}$ in \eqref{RDC:L12} and since $\|\I_{U_n}\|_{\mathcal C}\leq C'$ we get
\begin{align}
\label{eq:estimate2}
\p\big(\{(x,\o):x\in U_n,\, f_{\o}^j(x)\in\ U_n\}\big) &\le
 (\sm (U_n))^2+C \left\| \I_{U_n}\right\|_{\mathcal C} \left\| \I_{U_n}\right\|_{L^1(\mu_\varepsilon)} j^{-2}\notag\\&\leq  (\sm(U_n))^2+C^*\mu_{\varepsilon}(U_n)j^{-2},
\end{align}
where $C^*=CC'>0$.
Now observe that by definition of $U_n$ and \eqref{eq:un}, we have that $\mu_\epsilon(U_n)\sim \tau/n$. Using this observation together with the definition of $R^{\o}_n$ and the estimate \eqref{eq:estimate2}, it follows that there exists some constant $D>0$ such that
\begin{align*}
 n\sum_{j=\alpha_n}^{\lfloor n/k_n \rfloor}& \p\big(\{(x,\o):x\in U_n,\, f_{\o}^j(x)\in\ U_n
 %,R^{\o}(U_n)>\alpha_n
 \}\big)\leq n\big\lfloor\tfrac {n}{k_n}\big\rfloor\sm(U_n)^2 +n\,C^*\mu_{\varepsilon}(U_n) \sum_{j=\alpha_n}^{\lfloor n/k_n \rfloor}j^{-2}
\\& \leq \frac{(n\sm(U_n))^2}{k_n} +n\,C^*\mu_{\varepsilon}(U_n) \sum_{j=\alpha_n}^{\infty}j^{-2}\leq D \left(\frac{\tau^2}{k_n}+\tau \sum_{j=\alpha_n}^{\infty}j^{-2} \right)\xrightarrow[n\to\infty]{}0.
\end{align*}
For the term $I\!I$, %\textbf{HA: (40) $II$ "the second term"??, Jorge, please change this if you think that it is necessary, did not seem so to me}, 
as $\mu_\eps(U_n)\sim \tau/n$ and since $d\mu_\eps/d\l$ is bounded below and above by positive constants, there exists  some positive constant $C^*>0$ so that
\begin{eqnarray}
n\sum_{j=1}^{\lfloor n/k_n \rfloor}\p\big(\{(x,\o):x\in U_n, R^{\o}(U_n)\leq\alpha_n\}\big)&\leq&  \frac{n^2}{k_n}\mu_{\eps}(U_n)C\overline{g_\eps} \l(U_n)\frac{\eta}{\eta-1}\eta^{\alpha_n}\nonumber \\ \label{eq:final-estimate}
&\leq& C^* \frac{\eta^{\alpha_n}}{k_n}\xrightarrow[n\to\infty]{}0\quad \mbox{by \eqref{eq:alpha-n}}.
\end{eqnarray}
\end{proof}

\begin{proof}[Proof of Corollary~\ref{cor:random-EVL=>Poisson}]
The only extra step we need to do is to check that $D_3(u_n)$ also holds. To do that we just have to change slightly the definition of $\psi$ that we used to prove $D_2(u_n)$ by using \eqref{RDC:L12}. Let $A\in\mathcal R$.
We set:
$$
\psi(x)=\displaystyle\int \I_{\bigcap_{i\in A\cap\N}\{f^{i}_{\tilde\o}(x)\leq u_n\}}\, \dif\theta_\varepsilon^{\N} (\tilde\o).
$$
The rest of the proof follows exactly as in the proof of $D_2(u_n)$ in the proof of Theorem~\ref{thm:random-EVL}.
\end{proof}

\subsection{Laws of rare events for specific randomly perturbed systems}

\subsubsection{Expanding and piecewise expanding maps on the circle with a finite number of discontinuities}
 We give a general definition from \cite{V97} of piecewise expanding maps on the circle which also includes the particular case of the continuous expanding maps:
\begin{itemize}
\item[(1)] there exist $\ell\in\N_0$ and $0=a_0<a_1<\cdots<a_{\ell}=1=0=a_0$ for which the restriction of $f$ to each ${\Xi_i}=(a_{i-1},a_i)$ is of class $C^1$, with $|Df(x)|>0$ for all $x\in\Xi_i$ and $i=1,\ldots,\ell$. In addition, for all $i=1,\ldots,\ell$, $g_{\Xi_{i}}=1/|Df|_{\Xi_i}|$ has bounded variation for $i=1,\ldots,\ell$.
\end{itemize}
We assume that $(f|_{\Xi_i})$ and $g_{\Xi_i}$ admit continuous extensions to $\Xi_i=[a_{i-1},a_i]$, for each $i=1,\ldots,\ell$. Since modifying the values of a map over a finite set of points does not change its statistical properties, we may assume that $f$ is either left-continuous or right-continuous (or both) at $a_i$, for each $i=1,\ldots,\ell$ (possibly for all $i$'s at the same time). Then let $\mathcal{P}^{(1)}$ be some partition of ${\mathcal S}^1$ into intervals $\Xi$ such that $\Xi\subset\Xi_i$ for some $i$ and $(f|\Xi)$ is continuous. Furthermore, for $n\geq 1$, let $\mathcal{P}^{(n)}$ be the partition of ${\mathcal S}^1$ such that $\mathcal{P}^{(n)}(x)=\mathcal{P}^{(n)}(y)$ if and only if $\mathcal{P}^{(1)}(f^j(x))=\mathcal{P}^{(1)}(f^j(y))$ for all $0\leq j< n$. Given $\Xi\in\mathcal{P}^{(n)}$, denote $g_{\Xi}^{(n)}=1/|Df^n|_{\Xi}|$;
\begin{itemize}
\item[(2)] there exist constants $C_1>0,\lambda_1<1$ such that $\sup {g^{(n)}_{\Xi}}\leq C_1\lambda^{n}_1$ for all $\Xi\in \mathcal{P}^{(n)}$ and all $n\geq 1$;
\item[(3)] for every subinterval $J$ of ${\mathcal S}^1$, there exists some $n\geq 1$ such that $f^n(J)={\mathcal S}^1$.
\end{itemize}

According to \cite[Proposition~3.15]{V97}, one has exponential decay of correlations for randomly perturbed systems derived from maps satisfying conditions $(1)-(3)$ above, taking $\mathcal{C}$ as the space of functions with bounded variation ($BV$), \ie  given $\varphi$ in $BV$
and $\psi\in L^1(\l)$,
\begin{equation}
\label{RDC:L13} \left|\int (\mathcal{U}_\eps \psi)\varphi\, \dif\l- \int\psi\, \dif\sm\int \varphi\, \dif\l \right|
 \leq C \lambda^n \|\varphi\|_{BV}\|\psi\|_{L^1(\l)},
\end{equation}
where $0<\lambda<1$ and $C>0$ is a constant independent of both $\varphi, \psi$.

Hence, in the particular case of $f$ being a continuous expanding map of the circle, (\ref{RDC:L13}), Theorem~\ref{thm:random-EVL}, Corollaries~\ref{cor:random-EVL=>HTS}  and \ref{cor:random-EVL=>Poisson}   allow us to obtain
\begin{corollary}
\label{for:random-EVL-expanding-circle}
Let $f:{\mathcal S}^1\to {\mathcal S}^1$ be a continuous expanding map satisfying $(1)-(3)$ above, which is randomly perturbed as in \eqref{eq:add-perturbation} with noise distribution given by \eqref{eq:noise-distribution}. For any point $\zeta\in\mathcal M$, consider that $X_0, X_1,\ldots$ is defined as in \eqref{eq:def-rand-stat-stoch-proc-RDS2} and let $u_n$ be such that \eqref{eq:un} holds. Then the stochastic process $X_0, X_1,\ldots$ satisfies $D_2(u_n)$, $D_3(u_n)$ and $D'(u_n)$, which implies that we have an EVL for $M_n$ such that $\bar H(\tau)=\e^{-\tau}$ and we have exponential HTS/RTS for balls around $\zeta$. Moreover, the REPP $N_n$ defined
  in \eqref{eq:def-REPP} is such that
  $N_n\xrightarrow[]{d}N$, as $n\rightarrow\infty$, where $N$
  denotes a Poisson Process with intensity $1$.
\end{corollary}

In the proof of Theorem~\ref{thm:random-EVL}, we used the continuity of the map, in particular, in \eqref{eq:fact1}. However, we can adapt the argument in order to allow a finite number of discontinuities for expanding maps of the circle.

\begin{proposition}
Let  $f:{\mathcal S}^1\to {\mathcal S}^1$ be a map satisfying conditions $(1)-(3)$ above, which is randomly perturbed as in \eqref{eq:add-perturbation} with noise distribution given by \eqref{eq:noise-distribution}. For any point $\zeta\in\mathcal M$, consider that $X_0, X_1,\ldots$ is defined as in \eqref{eq:def-rand-stat-stoch-proc-RDS2} and let $u_n$ be such that \eqref{eq:un} holds. Then the stochastic process $X_0, X_1,\ldots$ satisfies $D_2(u_n)$, $D_3(u_n)$ and $D'(u_n)$, which implies that we have an EVL for $M_n$ such that $\bar H(\tau)=\e^{-\tau}$ and we have exponential HTS/RTS for balls around $\zeta$. Moreover, the REPP $N_n$ defined
  in \eqref{eq:def-REPP} is such that
  $N_n\xrightarrow[]{d}N$, as $n\rightarrow\infty$, where $N$
  denotes a Poisson Process with intensity $1$.
\end{proposition}

\begin{proof}

The proof of $D_2(u_n)$ follows from (\ref{RDC:L13}) as in the continuous case. Regarding the proof of $D'(u_n)$, in order to use the same arguments as in the continuous case, we want to avoid coming close to the discontinuity points along the random orbit of $\zeta$ (up to time $\alpha_n$). Since there are finitely many discontinuity points, say $\xi_i$'s for $i=1,\ldots,\ell$, we can control this by asking for some "safety regions" around each of them. By doing so, we ensure that the random orbit of $\zeta$ is sufficiently far away from $\xi_i$'s so that the iterates of $U_n$ consist of only one connected component. We can formulate these ``safety regions'' as
\begin{equation}
\label{eq:safetybox}
\dist(f_{\o}^j(\zeta),\xi_i)>2\eta^j|U_n| \textrm{ for all } i=1,\dots,\ell.
\end{equation}

Now, we make an estimate on the $\o$'s that take the orbit of $\zeta$ too close to the discontinuity points as well as close to $\zeta$ itself and our aim is to show that the $\theta_\eps^\N$ measure of this set is small. Let us set $\xi_0=\zeta$ to simplify the notation. Then,
%\begin{equation*}
$\big\{\o:\; R^{\o}(U_n)\leq\alpha_n\big\}\subset \bigcup_{j=1}^{\alpha_n} \bigcup_{i=0}^{\ell} \big\{\o:\; f_{\o}^j(\zeta)\in B_{2\eta^j|U_n|}(\xi_i)\big\}.$
%\end{equation*}
Thus, we have
\begin{align*}
\theta_\eps^\N\big(\big\{\o:\; &R^{\o}(U_n)\leq\alpha_n\big\}\big)\leq \sum_{i=0}^{\ell} \sum_{j=1}^{\alpha_n}\int \theta_\eps\Big(\Big\{\omega_j: f\left(f_{\o}^{j-1}(\zeta)\right)+\omega_j\in  B_{2\eta^j|U_n|}(\xi_i) \Big\}\Big)\dif\theta^\N_\eps\\
%&= \sum_{i=0}^{\ell} \sum_{j=1}^{\alpha_n}\int \theta_\eps\Big(\Big\{\omega_j: \omega_j\in  B_{2\eta^j|U_n|}(\xi_i)-f\left(f_{\o}^{j-1}(\zeta)\right) \Big\}\Big) \dif\theta^\N_\eps\\
%&= \sum_{i=0}^{\ell} \sum_{j=1}^{\alpha_n} \iint_{B_{2\eta^j|U_n|}(\xi_i)-f\left(f_{\o}^{j-1}(\zeta)\right)}g_\eps(x) dx\dif\theta^\N_\eps\\
&\leq \sum_{i=0}^{\ell} \sum_{j=1}^{\alpha_n} \overline{g_\eps}\left|B_{2\eta^j|U_n|}(\xi_i)\right|=\sum_{i=0}^{\ell} \sum_{j=1}^{\alpha_n} \overline{g_\eps} 4\eta^j|U_n|
\leq 4(\ell+1)\overline{g_\eps} |U_n|\frac{\eta}{\eta-1}\eta^{\alpha_n}.
\end{align*}

%\textbf{HA: (44) again the second and the third inequalities should be equalities?}

The proof now follows the same lines as the proof of Theorem~\ref{thm:random-EVL} and Corollary~\ref{cor:random-EVL=>Poisson}.
\end{proof}

\subsubsection{Expanding and piecewise expanding maps in higher dimensions}
Let us now consider the multidimensional piecewise expanding systems defined in Section~\ref{subset:example-Saussol} but only with a finite number, $K$, of domains of local injectivity;  moreover let us restrict ourselves to a mixing component which, for simplicity, we will take as the whole space $Z$; we will take $\mu$ as the unique absolutely continuous invariant measure with density $h$. In addition, we ask each $\partial Z_i$ to be included in piecewise $C^1$ codimension-$1$ embedded compact submanifolds and for $Z(f)=\sup\limits_{x}\sum_{i=1}^K \#\{\textrm{ smooth pieces intersecting } \partial Z_i \textrm{ containing }x\}$
\begin{equation}\label{sc}
s^\alpha+\frac{4s}{1-s}Z(f)\frac{\gamma_{N-1}}{\gamma_N}<1,
\end{equation}
where $\gamma_N$ is the N-volume of the N-dimensional unit ball of $\R^N$. Then, we know that by Lemma 2.1.\ in \cite{S00}, item (4) in Definition~\ref{def:MPE} is satisfied \footnote{The inequality \eqref{sc} ensures that for the unperturbed map the quantity $\eta(\eps_1)<1$; see the definition of this quantity after the formula \eqref{GG}. The value of $\eta(\eps_1)$ is one of the constants in the Lasota-Yorke inequality, see item (i) below, and we will require that it will be independent of the noise. This will be the case for the additive noise since the determinant of the perturbed maps will not change and this is what is used in \eqref{GG} to control the Lebesgue measure of $f_i^{-1}B_{\eps}(\partial f Z_i)$. The other factor in the Lasota-Yorke inequality is also given in terms of the quantity \eqref{GG}. }. Notice that formula \eqref{sc} gives exponential decay of correlations for the adapted pair:  $L^1$ functions against functions in the quasi-H\"older space ~$V_{\alpha}$.\\
We perturb again this kind of maps with  additive noise by asking that the image of $Z$ is strictly included in $Z$. We will also require that the density $h$ is bounded from below by the positive constant $h_m$. We will now prove the exponential decay of correlations for the random evolution operator ${\mathcal U}_{\eps}$, by using the perturbation theory in \cite{KL09}, which we will also quote and use later on in Section \ref{subsec:lim-dist}. This theory ensures us that the perturbed Perron-Frobenius  operator ${\mathcal P}_{\epsilon}$ is mixing on the adapted pair $(L^1, V_{\alpha})$ whenever we have:\\
(i) a uniform Lasota-Yorke inequality for ${\mathcal P}_{\epsilon}$, \ie all the constants in that inequality are independent of the noise $\eps$,\\
(ii) the closeness property (see also hypothesis {\bf H4} in Section \ref{subsec:lim-dist} below): there exists a monotone upper semi-continuous function $p:\Omega\rightarrow [0,\infty)$ such that $\lim_{\eps\rightarrow 0}p_{\eps}=0$ and $\forall \varphi \in V_{\alpha}$ , $\forall \eps\in \Omega: ||{\mathcal P}\varphi-{\mathcal P}_{\eps}\varphi||_1\le p_{\eps}||\varphi ||_{\alpha}$.

Condition (i) follows easily by observing that  the derivative of the original and of the perturbed maps are the same, and this does not change the contraction factor $s$, and the multiplicity of the boundaries' intersection, $Z(f)$,  is invariant too. Finally we invoke the observation written in the preceding footnote.  Therefore the Perron-Frobenius operators $\mathcal {P}_{\omega}$ associated to the perturbed maps $f_{\omega}$ verify the same Lasota-Yorke inequality  and therefore the same is true for $\mathcal{P}_{\eps}$. \\ Our next step is to prove condition (ii), in particular we have
\begin{proposition} 
\label{prop:DCL1-PE}
There exists a constant $C$ such that for any $\varphi\in V_{\alpha}$ we have $$\Vert\mathcal{P}\varphi-\mathcal{P}_{\eps}\varphi\Vert_1\leq C\eps^{\alpha}||\varphi||_{\alpha}.$$
\end{proposition}
\begin{proof}
We have
\begin{align*}
\Vert \mathcal{P}\varphi-\mathcal{P}_{\eps}\varphi\Vert_1 \leq \int_Z\int_{\Omega}\vert\mathcal{P}_{\omega}\varphi(x)-\mathcal{P}\varphi(x)\vert \dif\theta_{\eps}(\omega)\dif x.
\end{align*}
Putting $G(x)=\frac{1}{|\det Df(x)|}$, we can write
\begin{align}
\begin{split}
\vert\mathcal{P}_{\omega}\varphi(x)-\mathcal{P}\varphi(x)\vert \leq \sum_{Z_i, i=1,\ldots, K}\vert \varphi(f_{i}^{-1}x) G(f_{i}^{-1}x)\I_{fZ_{i}}(x)-\varphi(f_{\omega,i}^{-1}x) G(f_{i}^{-1}x)\I_{f_{\omega}Z_{i}}(x)\vert\label{IN}\\
+\sum_{Z_i, i=1,\ldots, K}\vert \varphi(f_{\omega,i}^{-1}x)\vert \vert G(f_{i}^{-1}x)-G(f_{\omega,i}^{-1}x)\vert \I_{f_{\omega}Z_{i}}(x)\\
:=I+I\!I
\end{split}
\end{align}
where $f_{\omega}(x)=f(x)+\omega$ and $\omega$ is a vector in $\mathbb{R}^N$ with each component being less than $\eps$ in modulus. Moreover $f_{\omega,i}^{-1}$ denotes the inverse of the restriction of $f_{\omega}$ to $Z_i$, which is denoted by $f_{\omega,i}$ itself. We now bound the first sum, $I$, in \eqref{IN}  by considering two cases:\\
(i) Let us suppose first that $x\in fZ_i\cap f_{\omega,i}Z_i$. Then since both $f$ and $f_{\omega,i}$ are injective, there will be two points, $y_i$ and $y_{\omega,i}$ in $Z_i$ such that $x=f(y_i)=f_{\omega,i}(y_{\omega,i})=f(y_{\omega,i})+\omega$. This immediately implies that $\dist(y_i,y_{\omega,i})\le s \sqrt{N} \eps$, if $\dist$ is the Euclidean distance. For such an $x$ we continue to bound the first summand, I, in \eqref{IN} as:
$$
I\leq \sum_{Z_i, i=1,\ldots,K} G(f_{i}^{-1}x)\mbox{osc}(\varphi, B_{s \sqrt{N} \eps}(f_{i}^{-1}(x)))\I_{fZ_{i}}(x)
$$
By integrating over $Z$ we get
\begin{eqnarray*}
\int_Z\Big( \sum_{Z_i, i=1,\ldots,K} G(f_{i}^{-1}x)\mbox{osc}(\varphi, B_{s \sqrt{N} \eps}(f_{i}^{-1}(x)))\I_{fZ_{i}}(x)\Big) \dif x=\int_Z  \mathcal{P}(\mbox{osc}(\varphi, B_{s \sqrt{N} \eps}(x)))\dif x\\=\int_Z  \mbox{osc}(\varphi, B_{s \sqrt{N} \eps}(x))\dif x\leq (s \sqrt{N} \eps)^{\alpha} |\varphi|_{\alpha}.
\end{eqnarray*}

(ii) We now consider the case when $x\in fZ_i\Delta f_{\omega,i}Z_i$; the Lebesgue measure of this last set is bounded by $\eps$ times the codimension-1 volume of $\partial f Z_i$: let $r$ denote the maximum of those volumes for $i=1,\cdots,k$. Thus we get
\begin{equation}\label{FG}
\int_Z \vert\mathcal{P}_{\omega}\varphi(x)-\mathcal{P}\varphi(x)\vert\dif x\leq r \eps \Vert \varphi\Vert_{\infty}\Vert\mathcal{P}1\Vert_{\infty}.
\end{equation} We notice that the inclusion $V_{\alpha}\hookrightarrow L^{\infty}_m$ is bounded, namely there exists $c_v$ such that $\Vert \varphi\Vert_{\infty}\le c_v \Vert \varphi\Vert_{\alpha}.$ We therefore continue (\ref{FG}) as
$$
(\ref{FG})\leq r \eps c_v \Vert \varphi\Vert_{\alpha}\Vert \mathcal{P}\frac{h}{h}\Vert_{\infty}\le  r \eps c_v \Vert \varphi\Vert_{\alpha}\frac{\Vert h\Vert_{\infty}}{h_m}
$$
We now come to the second summand, $II$, in \eqref{IN}; we begin by observing that
\begin{align*}
 \vert G(f_{i}^{-1}x)-G(f_{\omega,i}^{-1}x)\vert &= \Bigg\vert\frac{1}{|\det Df(f^{-1}_i x)|}-\frac{1}{|\det Df(f^{-1}_{\omega,i}x)\vert}\Bigg\vert\\
%&=\Bigg\vert\frac{1}{|\det Df(y_i)\vert}-\frac{1}{|\det Df(y_{\omega,i})\vert}\Bigg\vert\\
&= \Big\vert |\det  Df_i^{-1}(x)\vert -|\det Df_i^{-1}(z)\vert\Big\vert
\leq \vert \det  Df_i^{-1}(x) -\det  Df_i^{-1}(z)\vert
\end{align*}
where $z=f(y_{\omega,i})$ and dist$(x,z)\le \sqrt{N}\eps$. By using the H\"older assumption (2) in Definition~\ref{def:MPE}, we have
\begin{align*}
%\sum_{Z_i, i=1,\ldots, K}\vert \varphi(f_{\omega,i}^{-1}x)\vert &\vert G(f_{i}^{-1}x)-G(f_{\omega,i}^{-1}x)\vert\I_{f_{\omega}Z_{i}}(x)\\
I\!I &\leq c (\sqrt{N}\eps)^{\alpha} \sum_{Z_i, i=1,\ldots, K}\vert \varphi(f_{\omega,i}^{-1}x)\vert  |\det Df_i^{-1}(z)|\I_{f_{\omega}Z_{i}}(x)\\
%&\leq  c (\sqrt{N}\eps)^{\alpha}\sum_{Z_i, i=1,\ldots, K}\vert \varphi(f_{\omega,i}^{-1}x)\vert  \frac{1}{|\det Df(f^{-1}_i(z))|}\I_{f_{\omega}Z_{i}}(x)\\
&\leq c (\sqrt{N}\eps)^{\alpha}\sum_{Z_i, i=1,\ldots, K}\vert \varphi(f_{\omega,i}^{-1}x)\vert  \frac{1}{|\det Df(f^{-1}_{\omega,i}(x))|}\I_{f_{\omega}Z_{i}}(x).
\end{align*}
 By integrating over $Z$ we get the contribution
 $$
 c (\sqrt{N}\eps)^{\alpha}\int_Z {\mathcal P}_{\omega}(|\varphi|)\dif x\le  c (\sqrt{N}\eps)^{\alpha}\int_Z |\varphi|\dif x\le c (\sqrt{N}\eps)^{\alpha}||\varphi||_{L^1(\l)}.
 $$

In conclusion we get $\Vert\mathcal{P}\varphi-\mathcal{P}_{\eps}\varphi\Vert_1\leq C\eps^{\alpha}\Vert \varphi\Vert_{\alpha}$, where the constant $C$ collects the various constants introduced above.
\end{proof}

As a consequence of Proposition~\ref{prop:DCL1-PE} we obtain exponential decay of correlations of quasi-H\"older functions (in $V_\alpha$), against $L^1$ functions, in particular, for uniformly expanding maps on the torus $\mathbb T^d$. Since, $\I_{U_n}\in V_\alpha$, $\|\I_{U_n}\|_\alpha$ is uniformly bounded by above, then it follows by Theorem~\ref{thm:random-EVL} and Corollary~\ref{cor:random-EVL=>HTS} that
\begin{corollary}
\label{for:random-EVL-expanding-torus}
Let $f:{\mathbb T}^d\to {\mathbb T}^d$ be a $C^2$ uniformly expanding map on $\mathbb T^d$, which is randomly perturbed as in \eqref{eq:add-perturbation} with noise distribution given by \eqref{eq:noise-distribution}. %For any point $\zeta\in \mathbb T^d$, consider that $X_0, X_1,\ldots$ is defined as in \eqref{eq:def-rand-stat-stoch-proc-RDS2}, let $u_n$ be such that \eqref{eq:un} holds and assume that $U_n$, defined in \eqref{def:Un}, is such that \eqref{def:zeta} holds. Then the stochastic process $X_0, X_1,\ldots$ satisfies $D_2(u_n)$ and $D'(u_n)$, which implies that we have an EVL for $M_n$ such that $\bar H(\tau)=\e^{-\tau}$ and we have exponential HTS/RTS for balls around $\zeta$.
For any point $\zeta\in\mathcal M$, consider that $X_0, X_1,\ldots$ is defined as in \eqref{eq:def-rand-stat-stoch-proc-RDS2} and let $u_n$ be such that \eqref{eq:un} holds. Then the stochastic process $X_0, X_1,\ldots$ satisfies $D_2(u_n)$, $D_3(u_n)$ and $D'(u_n)$, which implies that we have an EVL for $M_n$ such that $\bar H(\tau)=\e^{-\tau}$ and we have exponential HTS/RTS for balls around $\zeta$. Moreover, the REPP $N_n$ defined
  in \eqref{eq:def-REPP} is such that
  $N_n\xrightarrow[]{d}N$, as $n\rightarrow\infty$, where $N$
  denotes a Poisson Process with intensity $1$.

\end{corollary}
 
As in the previous case of maps on the circle, we may adapt the argument used in the continuous case to consider more general piecewise expanding maps of Definition \ref{def:MPE}, as long as there is a finite number of domains of local injectivity.
\begin{proposition}
\label{prop:random-multi-dimensional-EVL}
Suppose that $(Z,f,\mu)$ is a topologically mixing multidimensional piecewise expanding system as in Definition~\ref{def:MPE}, $\mu$ is the a.c.i.p.\ with a Radon-Nikodym density bounded away from 0.  We assume that there are $K\in\N$ domains of injectivity of the map and there exists $\eta>1$ such that for all $i=1,\ldots, K$ and all $x,y\in Z_i$ we have $\dist(f(x),f(y))\leq \eta\,\dist(x,y)$. Consider that such a map is randomly perturbed with additive noise as in \eqref{eq:add-perturbation}  with noise distribution given by \eqref{eq:noise-distribution} and such that the image of $Z$ is strictly included in $Z$. %For any point $\zeta\in Z$, consider that $X_0, X_1,\ldots$ is defined as in \eqref{eq:def-rand-stat-stoch-proc-RDS2}, let $u_n$ be such that \eqref{eq:un} holds and assume that $U_n$, defined in \eqref{def:Un}, is such that \eqref{def:zeta} holds. Then the stochastic process $X_0, X_1,\ldots$ satisfies $D_2(u_n)$ and $D'(u_n)$, which implies that we have an EVL for $M_n$ such that $\bar H(\tau)=\e^{-\tau}$ and we have exponential HTS/RTS for balls around $\zeta$.
For any point $\zeta\in\mathcal M$, consider that $X_0, X_1,\ldots$ is defined as in \eqref{eq:def-rand-stat-stoch-proc-RDS2} and let $u_n$ be such that \eqref{eq:un} holds. Then the stochastic process $X_0, X_1,\ldots$ satisfies $D_2(u_n)$, $D_3(u_n)$ and $D'(u_n)$, which implies that we have an EVL for $M_n$ such that $\bar H(\tau)=\e^{-\tau}$ and we have exponential HTS/RTS for balls around $\zeta$. Moreover, the REPP $N_n$ defined
  in \eqref{eq:def-REPP} is such that
  $N_n\xrightarrow[]{d}N$, as $n\rightarrow\infty$, where $N$
  denotes a Poisson Process with intensity $1$.

\end{proposition}
\begin{proof}

Previously, for maps on the circle, by putting some ``safety regions'' around the discontinuity points we guaranteed that the iterates of $f_{\o}^j(U_n)$, $j=0,1,\ldots, \alpha_n$ had one connected component. Since in this case the border of the domains of injectivity are codimension-1 submanifolds instead of single points (as in the 1-dimensional case), we must proceed to a more thorough analysis. To that end, for each $\o$, for $j=1$ let $1\leq l_1\leq K$ be the number of  intersections with non-empty interior between  $f_{\o}(U_n)$ and $Z_i$, with $i=1,\ldots,K$. For each $\ell=1,\ldots, l_1$, let $i_\ell$ denote the index of the partition element $Z_{i_\ell}$ for which such intersection has non-empty interior, define $U_n^{(1,\ell)}:=f_{\o}(U_n)\cap Z_{i_\ell}$ and let $\zeta_{1,\ell}$ be a point in the interior of $U_n^{(1,\ell)}$. For any $j=2,\ldots, \alpha_n$, given the sets $U_n^{(j-1,k)}$, with $k=1,\ldots, l_{j-1}$, let $l_j$ be the total number of intersections of non-empty interior between $f_{\sigma^{j-1}(\o)}\left(U_n^{(j-1,k)}\right)$ and $Z_i$, with $i=1,\ldots,K$. For each $\ell=1,\ldots, l_j$, let $i_\ell$ denote the index of the partition element $Z_{i_\ell}$ and $k_\ell$ the super index of the sets $U_n^{(j-1,k)}$ for which the intersection between $f_{\sigma^{j-1}(\o)}\left(U_n^{(j-1,k_\ell)}\right)$ and $Z_{i_\ell}$ has non-empty interior, define $U_n^{(j,\ell)}=f_{\sigma^{j-1}(\o)}\left(U_n^{(j-1,k_\ell)}\right)\cap Z_{i_\ell}$ and let $\zeta_{j,\ell}$ be a point in the interior of $U_n^{(j,\ell)}$.  

In order to avoid the first return time to $U_n$ occurring before $\alpha_n$ in a similar way to the previous proofs, we require that:
\begin{equation}
\label{eq:piecewise-multidimensional-return}
\dist(f_{\sigma^{j-1}(\o)}(\zeta_{j-1,\ell}),\zeta)>2\eta^j|U_n| \textrm{ for all } j=2,\ldots,\alpha_n ,\; \ell=1,\ldots, l_{j-1}.
\end{equation}
Note that, as in the proof of Theorem~\ref{thm:random-EVL}, for any $\o\in\Omega$, we have $|U_n^{(j,\ell)}|\leq \eta^j |U_n|$. This implies that
\begin{equation}
\label{eq:fact3}
\mbox{if $\dist(f_{\sigma^{j-1}(\o)}(\zeta_{j-1,\ell}),\zeta)>2\eta^j|U_n|>|U_n|+\eta^j|U_n|$ then $U_n^{(j,\ell)}\cap U_n=\emptyset$}.
\end{equation}
Note that, by equation \eqref{eq:fact3}, if \eqref{eq:piecewise-multidimensional-return} holds then clearly $R^{\o}(U_n)>\alpha_n$. Hence, letting $l_0=1$ and $\zeta_{0,1}=\zeta$, we may write that
$$\big\{\o:\; R^{\o}(U_n)\leq\alpha_n\big\}\subset \bigcup_{j=1}^{\alpha_n} \bigcup_{\ell=1}^{l_{j-1}} \big\{\o:\; f_{\sigma^{j-1}(\o)}(\zeta_{j-1,\ell})\in B_{2\eta^j|U_n|}(\zeta)\big\}.$$
Recalling that $l_j\leq K^j$, for all $j=1,\ldots,\alpha_n$, it follows that, there exists some $C>0$ such that
\begin{align*}
\theta_\eps^\N&\big(\big\{\o:\; R^{\o}(U_n)\leq\alpha_n\big\}\big)\leq \sum_{j=1}^{\alpha_n} \sum_{\ell=1}^{l_{j-1}}\int \theta_\eps\Big(\Big\{\omega_j: f\left(\zeta_{j-1,\ell}\right)+\omega_j\in  B_{2\eta^j|U_n|}(\zeta) \Big\}\Big) \dif\theta^\N_\eps\\
%&\leq& \sum_{j=1}^{\alpha_n}\sum_{\ell=1}^{l_{j-1}}\theta_\eps\Big(\Big\{\omega_j: \omega_j\in  B_{2\eta^j|U_n|}(\zeta)-f\left(\zeta_{j-1,\ell}\right) \Big\}\Big)\\
&\leq \sum_{j=1}^{\alpha_n} \sum_{\ell=1}^{l_{j-1}} \overline{g_\eps}\l\left(B_{2\eta^j|U_n|}(\zeta)\right)\leq\sum_{j=1}^{\alpha_n} K^j\overline{g_\eps} C\eta^j\l(U_n)
\leq C\overline{g_\eps} \l(U_n)\frac{\eta K}{\eta K-1}(\eta K)^{\alpha_n}
\end{align*}
Now, the proof follows exactly in the same way as the proof of Theorem~\ref{thm:random-EVL} and Corollary~\ref{cor:random-EVL=>Poisson}, except that in the final estimate (\ref{eq:final-estimate}), $\eta$ should be replaced by $\eta K$, which will not make any difference by the choice of $\alpha_n$ defined in (\ref{eq:alpha-n}).
\end{proof}

\section{Extremes for random dynamics from a spectral approach}
\label{sec:Keller}
In this section, we want to prove our results for the random case using another approach introduced by Keller in \cite{K12}.  His technique is based on an eigenvalue perturbation formula which was given in \cite{KL09} under a certain number of  assumptions that we recall in the first subsection and adapt to our situation. We check those assumptions in Section~\ref{subsec:checking-assumptions} for a large class of  maps of the interval whose properties are listed in the conditions ({\bf H1}-{\bf H5}). Possible generalisations deserve to be investigated and we point out here a major difficulty in  higher dimensions. In this case one should control  (any kind of) variation/oscillation on the boundaries of the preimages of the complement of balls (the set $U^c_m$ in the proof of Proposition 5.2 below; it is important that such variation/oscillation grows at most sub-exponentially). 
%The second difficulty holds in the 1-D case as well and it %forced us to assume  H6 below:  we called it {\em random %uniform expectation of cylinders}.
To sum up, the direct technique introduced in Section 4 and the spectral one in this section are complementary. The direct technique is easily adapted to higher dimensions but it requires assumptions on the noise in order to control the short returns (see the quantity $R^{\underline{\omega}}(U_n)$ in Proposition 4.5), which follows easily for additive noise. The spectral technique is an alternative method and for the moment particularly adapted to the 1-D case and, as we will see in a moment, the noise could be chosen in a quite general way to prove the existence of the EI, formula (\ref{AEA}). Instead, if we want to characterise such an EI and show that it is always equal to $1$, we need to consider special classes of uniformly expanding maps and particularly the noise should be chosen as additive and with a continuous distribution (Proposition \ref{calq}). The fact that the existence of EI follows for general classes of noises is clear by looking at the proof of Proposition \ref{KUA}. Indeed what is really necessary is that the derivatives of the randomly chosen maps are close enough to each other in order to guarantee the uniformity of the Lasota-Yorke inequality for the perturbed Perron-Frobenius operator. This could be achieved quite widely and with discrete distributions as well. Nevertheless, in order to make the exposition simpler and coherent with the previous sections, we will consider additive noise, together with {\em any kind of  distribution} to prove Proposition \ref{KUA} and with  {\em absolutely continuous distributions} to prove Proposition \ref{calq}.

\subsection{The setting}\label{subsec:REPFO}

Given a Banach space $(V,\|\cdot\|)$, and a set
 of parameters $E$ which is equipped with some topology,
  let us suppose there are $\lambda_\eps\in\C$,
   $\varphi_\eps\in V$, $\nu_\eps\in V'$ ($V'$ denotes the dual of V) and linear operators $P_{\eps}, Q_{\eps}: V\to V$ such that
\begin{gather}\label{Keller-condition-1}
\lambda_\eps^{-1} P_{\eps}=\varphi_\eps\otimes\nu_\eps+
Q_{\eps} \textrm{ (assume $\lambda_0=1$) },\\
P_{\eps}(\varphi_\eps)=\lambda_\eps\varphi_\eps, \nu_\eps P_{\eps}=
\lambda_\eps\nu_\eps,\, Q_{\eps}(\varphi_\eps)=0,\, \nu_\eps Q_{\eps}=0,\\
\displaystyle\sum_{n=0}^{\infty}\sup_{\eps\in E}\|Q_{\eps}^n\|
=: C_1< \infty,\\
\exists C_2>0,\forall\eps\in E: \nu_0(\varphi_\eps)=1 \textrm{ and } \|\varphi_\eps\|\leq C_2 < \infty,\\
\displaystyle\lim_{\eps\to 0}\|\nu_0(P_0-P_{\eps})\|=0,\\
\|\nu_0(P_0-P_{\eps})\|\cdot\|(P_0-P_{\eps})\varphi_0\| \leq
\textrm{const}\cdot |\Delta_\eps| \label{Keller-condition-6}
\end{gather}
where $$\Delta_\eps:=\nu_0((P_0-P_{\eps})(\varphi_0)).$$

Under these assumptions, Keller and Liverani got the following formula as  the main result in \cite{KL09}:
\begin{equation}\label{AEA}
1-\lambda_\eps=\Delta_\eps\vartheta(1+o(1)) \textrm{ in the limit as } \eps\to 0
\end{equation}
where $\vartheta$ is said to be a constant to take care of short time correlations, which is later identified as the extremal index in extreme value theory context as mentioned in \cite[Section~1.2]{K12}. Actually $\vartheta$ is  given by an explicit and, in some cases, computable formula, and, in fact, we will be able to compute it for our random systems. This formula is the content of Theorem 2.1 in \cite{KL09} and states that under the above assumptions, in particular when $\Delta_{\eps}\neq 0$, for $\eps$ small enough, and whenever the following limit exists
\begin{equation}\label{eq:qk}
q_k:=\lim_{\eps\rightarrow 0}q_{k,\eps}:= \lim_{\eps\rightarrow
\infty}\frac{\nu_0((P_0-P_{\eps})P_{\eps}^k(P_0-P_{\eps})(\varphi_0))}
{\Delta_\eps},
\end{equation}
we have
\begin{equation}\label{eq:qk2}
\lim_{\eps\rightarrow 0}\frac{1-\lambda_{\eps}}{\Delta_\eps}=\vartheta:=1-\sum_{k=0}^{\infty}q_k.
\end{equation}
We now state equivalent ways to verify assumptions \eqref{Keller-condition-1}-\eqref{Keller-condition-6}, we refer to \cite{K12} for the details.

\begin{itemize}
\item[\bf (A1)] \ There are constants $ A>0, B>0, D>0$ and a second norm $|\cdot|_{\omega}\leq \|\cdot\|$ on $V$ (it is enough to be a seminorm) such that:
  \begin{gather}
  \forall\eps\in E, \forall \psi \in V, \forall n \in \N: |P_{\eps}^n \psi|_{\omega}\leq D |\psi|_{\omega}\label{eq:asmp2}\\
  \exists\alpha\in (0,1), \forall\eps \in E, \forall \psi\in V, \forall n
  \in \N: \|P_{\eps}^n \psi\| \leq A\alpha^n \|\psi\| +B |\psi|_{\omega}\label{eq:asmp3}
  \end{gather}
  Moreover the closed unit ball of $(V, \left\|\cdot\right\|)$, is $|\cdot|_\omega$-compact.
  \item[{\bf (A2)}] The unperturbed operator verifies the mixing condition $$
  P=\varphi\otimes\nu+Q_0\textrm{ (assume $\lambda_0=1$) }
  $$
\item [\bf (A3)] $\exists C>0$ such that
\begin{equation}
\eta_{\eps}:=\sup_{\left\|\psi\right\|\le1}\left|\int (P_0-P_{\eps})\psi
\,\dif\nu_0\right|\rightarrow 0, \ \mbox{as} \ \eps \rightarrow 0
\end{equation}
\item [\bf (A4)] and
\begin{equation}
\eta_{\eps} \ \left\|(P_0-P_{\eps})\varphi_0\right\|\le C \
\Delta_{\eps}
\end{equation}
\end{itemize}

 Keller called this framework {\em Rare events Perron-Frobenius operators}, REPFO. We will construct a {\em perturbed} Perron-Frobenius operator which satisfies the previous assumptions and which will give us information on extreme value distributions and statistics of first returns to small sets.

Before continuing, we should come back to our extreme distributions, namely to the quantity
$\{M_m\leq u_m\}=\{r_{\{\phi>u_m\}}>m\}$  where  $U_m:=\{\phi>u_m\}$
is a topological ball shrinking to the point $\zeta$ (see \eqref{def:Un}: we changed $U_n$ into $U_m$ here).
Now we consider the first time $r_{U_m}^{\o}(x)$ the point $x$ enters $U_m$ under the realization $\o$, namely under the composition  $\cdots\circ f_{\omega_k}\circ f_{\omega_{k-1}}\circ\cdots \circ f_{\omega_1}(x)$. For simplicity we indicate it by $r_{m}^{\o}(x)$ and consider its annealed distribution:
\begin{equation}
\label{eq:meas}
(\mu_{\eps} \text{ x } \nm)((x,\o): r_{m}^{\o}(x) > m)=(\mu_{\eps} \text{ x } \nm)(M_m\leq u_m)
\end{equation}
Let us write the measure on the left-hand side of ~\eqref{eq:meas} in terms of integrals: it is given by
\begin{equation}
\label{eq:intmeas}
\iint\limits_{\{r_{m}^{\o}>m\}} \dif(\mu_{\eps} \text{ x } \nm)=\iint h_{\eps} \I_{U_m^c}(x)\I_{U_m^c}(f_{\omega_1}x)\cdots \I_{U_m^c}(f_{\omega_{m-1}}\circ\cdots\circ f_{\omega_1}x)\, \dif\l\, \dif\nm
\end{equation}
which is in turn equal to
\begin{equation}\label{eq:intOp}
\int_M\widetilde{\P}_{\eps,m}^m h_{\eps}(x)\,\dif\l
\end{equation}
where we have now defined
\begin{equation}\label{rpfo}
\widetilde{\P}_{\eps,m}\psi(x):=\P_{\eps}({\bf 1}_{U_m^c}\psi)(x).
\end{equation}

Let us note that the operator $\widetilde{\P}_{\eps,m}$ depends on $m$ via the set $U_m$, and not on $\eps$ which is kept fixed and that $\widetilde{\P}_{\eps,m}$ ``reduces'' to $\P_{\eps}$ as $m\rightarrow \infty$. It is therefore tempting to consider $\widetilde{\P}_{\eps,m}$ as a small perturbation of $\P_{\eps}$ when $m$ is large and to check if it shares the spectral properties of a REPFO operator. We will show in a moment that it will be the case; let us now see what that implies for our theory.
\subsection{Limiting distributions}
\label{subsec:lim-dist}

We now indicate the correspondences between the general notations of
Keller's results and our own quantities:
\begin{gather*}
P_0\Rightarrow {\mathcal P}_{\eps}\\
P_{\eps}\Rightarrow
\widetilde{\P}_{\eps,m}; \ Q_{\eps}\Rightarrow Q_{\eps, m}\\
\varphi_{\eps}\Rightarrow \varphi_{\eps,m}; \ \varphi_0°\Rightarrow h_{\eps}\\
\lambda_{\eps}\Rightarrow \lambda_{\eps,m}\\
\nu_{\eps} \Rightarrow \nu_{\eps,m}; \ \nu_0\Rightarrow \mbox{Leb}\\
\Delta_{\eps}\Rightarrow
\Delta_{\eps,m}=\mu_{\eps}(U_m)=\l((\P_{\eps}-\widetilde{\P}_{\eps,m})h_{\eps})
\end{gather*}
The framework for which we will prove the assumptions {\bf (A1)-(A4)} for our REPFO $\widetilde{\P}_{\eps,m}$ are those behind the system and its perturbations which we introduced in the previous sections and which we summarize here:\\

{\bf Hypotheses on the system and its perturbations}\\
We consider piecewise expanding maps $f$ of the circle or of the interval $I$ which verify:\\
{\bf H1} The map $f$ admits a (unique) a.c.i.p.\ which is mixing.\\
{\bf H2} %We perturb $f$ with additive noise and we write $f_{\omega}$ for the family of perturbed maps. 
We will require that
\begin{equation}\label{du1}
 \inf_{x\in I}\left|Df(x)\right|\ge \beta >1.
\end{equation}
and
\begin{equation}\label{du2}
 \sup_{x\in I²}\left|\frac{D^2f(x)}{Df(x)}\right|\leq C_1<\infty,
 \end{equation}
whenever the first and the second derivatives are defined.\\
{\bf H3} The couple of adapted spaces upon which the REPFO operators will act are:  the space of functions of bounded variation (as in Definition~\ref{def:variation}, we will indicate with ${\rm Var}$ the total variation), and $L^1(\l)$, with norm $\left\|\cdot\right\|_1$; this time, we will write $\left\|\cdot\right\|_{BV}=\V(\cdot)+\left\|\cdot\right\|_1$ for the associated Banach norm.\\
{\bf H4} There exists a monotone upper semi-continuous function $p:\Omega\rightarrow [0,\infty)$ such that $\lim_{\eps\rightarrow 0}p_{\eps}=0$ and $\forall f\in$ BV, $\forall \eps\in \Omega: ||{\mathcal P}f-{\mathcal P}_{\eps}f||_1\le p_{\eps}||f||_{BV}$\footnote{This condition can be checked in several cases. We did it, for instance, in the previous section 4.2.2. A general theorem is presented in Lemma 16 in \cite{K82}  for piecewise expanding maps of the interval endowed with our pair of adapted spaces and with the noise given by a convolution kernel. This means that $\theta_{\eps}$ is absolutely continuous with respect to Lebesgue on the space $\Omega$ with density $s_{\eps}$, and our two operators are related by the convolution formula ${\mathcal P}_{\eps}g(x)=\int_{\Omega}({\mathcal P}g)(x-\omega)s_{\eps}(\omega)\dif\omega$, where $g\in BV$. In the case of additive noise, it is straightforward to check that the previous formula is equivalent to ${\mathcal P}_{\eps}g(x)=\int_{\Omega}({\mathcal P}_{\omega}g)(x)s_{\eps}(\omega)\dif\omega$, where ${\mathcal P}_{\omega}$ is the Perron-Frobenius operator of the transformation $f_{\omega}$.}.\\
{\bf H5} The density $h_{\eps}$ of the stationary measure is bounded from below $\l$-a.e.\ and we call this bound $\underline{h}_{\eps}$.\\

{\bf Extreme values}\\
Let us therefore write $\widetilde{\P}_{\eps,m}\varphi_{\eps,m}=\lambda_{\eps, m}\varphi_{\eps,m}$, $\nu_{\eps,m}\widetilde{\P}_{\eps,m}=\lambda_{\eps,m}\nu_{\eps,m}$, and $\lambda_{\eps,m}^{-1}\widetilde{\P}_{\eps,m}=\varphi_{\eps,m}\otimes\nu_{\eps,m}+Q_{\eps,m}$.

Then formula (\ref{AEA}) implies that
$1-\lambda_{\eps,m}=\Delta_{\eps,m} \vartheta_{\eps} (1+o(1))$. We can
therefore write:
\begin{align*}
(\mu_{\eps} \text{ x } \nm)(M_m\leq u_m)&=\int_M\widetilde{\P}_{\eps,m}^mh_{\eps}(x) \,\dif\l=\lambda_{\eps,m}^m\int h_{\eps}\,\dif\nu_{\eps,m}+\lambda_{\eps,m}^m\int Q_{\eps,m}h_{\eps}\,\dif\l\\
&=e^{-(\vartheta_{\eps} m\,\mu_{\eps}(U_m)+m\, o(\mu_{\eps}(U_m)))}\int h_{\eps}\,\dif\nu_{\eps,m}+ {\mathcal O}(\lambda_{\eps,m}^m \left\|Q_{\eps,m}\right\|_{BV})
\end{align*}
Remember that we are under the assumption that $m\, (\mu_{\eps} \text{ x } \nm)(\phi>u_m)=  m\, \mu_{\eps} (\phi>u_m)= m\,\mu_{\eps}(U_m)\rightarrow \tau$, when $m\rightarrow \infty$; moreover it follows from the theory of \cite{KL09} that $\int h_{\eps}\, \dif\nu_{\eps,m}\rightarrow \int h_{\eps}\, \dif\l=1$, as $m$ goes to infinity. In conclusion we get
$$
(\mu_{\eps} \text{ x } \nm)(M_m\leq u_m)=e^{-\tau \vartheta_{\eps}}(1+o(1))
$$
in the limit $m\rightarrow \infty$ and where $\vartheta_{\eps}$ will be the extremal index and this will be explicitly computed later on for some particular maps thanks to formula \eqref{eq:qk2} and shown to be equal to $1$ for any point $\zeta$, see Proposition \ref{calq} below.\\

{\bf Random hitting times}\\
Let us denote again with $r_{U_m}^{\o}(x)$ the first entrance into the ball $U_m$. A direct application of \cite[Proposition~2]{K12} and which is true for REPFO operators, allows to get the following result, which we  adapted to our situation and which provides an explicit formula for the statistics of the first hitting times in the annealed case. Notice that this result strenghtens our Corollary~\ref{cor:random-EVL=>HTS} since it provides the error in the convergence to the exponential law.
\begin{proposition}
\label{prop:hts-error-terms}
For the REPFO  $\widetilde{\P}_{\eps,m}$ which verifies the hypotheses {\bf H1}-{\bf H5},  and using the notations introduced above,  there exists a constant $C>0$ such that for all $m$ big enough there exists $\xi_m>0$ s.t. for all $t>0$
$$
\left|(\mu_{\eps} \times \nm)\left\{r_{U_m}^{\o}>\frac{t}{\xi_m\,\mu_{\eps}(U_m)}\right\}-e^{-t}\right|\leq C \delta_m (t\vee1)e^{-t}
$$
where $\delta_m=O(\eta_m \log \eta_m)$,
$$
\eta_m=\sup \left\{\int_{U_m}\psi\, \dif\l; \| \psi \|_{BV} \leq 1\right\}
$$
and $\xi_m$ goes to $\vartheta_{\eps}$ as $m\rightarrow \infty$.
\end{proposition}
\subsection{Cheking assumptions (A1)-(A4)}
\label{subsec:checking-assumptions}

\begin{proposition}\label{KUA}
For the REPFO  $\widetilde{\P}_{\eps,m}$ which verifies the hypotheses {\bf H1}-{\bf H5},  the assumptions {\bf (A1)}-{\bf (A4)} hold.
\end{proposition}
\begin{proof}
 Condition ({\bf A1}) means to prove the Lasota-Yorke inequality  for the operator $\widetilde{\P}_{\eps,m}$. We recall that the constants $A$ and $B$ there must independent of the perturbation parameter which in our case is $m$.  We begin with the total variation.

The structure of $\widetilde{\P}_{\eps,m}$'s iterates is
\begin{equation}\label{FP}
(\widetilde{\P}_{\eps,m}^n \psi)=\int\cdots\int \P_{\omega_n}(\Iuc \P_{\omega_{n-1}} (\Iuc\cdots \P_{\omega_1}(\psi\,\Iuc)))\,\dif\theta_\eps(\omega_1)\cdots \dif\theta_\eps(\omega_n).
\end{equation}

Let us call $A_{l,\omega_j}$ the $l$-domain of injectivity of the map $f_{\omega_j}$ and call $f^{-1}_{l, \omega_j}$ the inverse of $f_{\omega_j}$ restricted to $A_{l,\omega_j}$.  We have:
\begin{multline*}
\Upsilon_{\omega_1,\ldots,\omega_n}:=\P_{\omega_n}(\Iuc \P_{\omega_{n-1}} (\Iuc \cdots\P_{\omega_1}(\psi\,\Iuc)))(x)\\
=\sum_{k_n,\ldots,k_1}\frac{(\psi\cdot \Iuc \cdot \Iuc\circ f_{\omega_1}\cdots\Iuc \circ f_{\omega_{n-1}}\circ\cdots\circ f_{\omega_1})((f^{-1}_{k_1, \omega_1}\circ\cdots\circ f^{-1}_{k_n, \omega_n})(x))}{|D(f_{\omega_n}\circ\cdots\circ f_{\omega_1})((f^{-1}_{k_1, \omega_1}\circ\cdots\circ f^{-1}_{k_n, \omega_n})(x))|}\\ \times\I_{f_{\omega_n}\circ\cdots\circ f_{\omega_1}}\Omega_{\omega_1,\ldots,\omega_n}^{k_1,\ldots,k_n}(x)
\end{multline*}
The sets
\begin{multline*}
\Omega_{\omega_1,\ldots,\omega_n}^{k_1,\ldots,k_n}:=f^{-1}_{k_1, \omega_1}\circ\cdots\circ f^{-1}_{k_{n-1}, \omega_{n-1}}A_{k_n,\omega_n}\cap f^{-1}_{k_1, \omega_1}\circ\cdots\circ f^{-1}_{k_{n-2}, \omega_{n-2}}A_{k_{n-1},\omega_{n-1}}\cap\cdots\\ \cdots\cap f^{-1}_{k_1, \omega_1} A_{k_2,\omega_2}\cap A_{k_1,\omega_1}
\end{multline*}
are intervals and they give a mod-$0$ partition of $I=[0,1]$; moreover the image $H_{\omega_1,\ldots,\omega_n}^{k_1,\ldots,k_n}:=f_{\omega_n}\circ\cdots\circ f_{\omega_1} \Omega_{\omega_1,\ldots,\omega_n}^{k_1,\ldots,k_n}$ for a given $n$-tuple $\{k_n,\ldots,k_1\}$ is a connected interval.
We also note for future purposes that we could equivalently write:
$$
g_n:=\Iuc \cdot \Iuc\circ f_{\omega_1}\cdot\ldots\cdot\Iuc f_{\omega_{n-1}}\circ\cdots\circ f_{\omega_1}= \I_{U_m^c\cap f^{-1}_{\omega_1}U_m^c\cap\cdots\cap f^{-1}_{\omega_1}\circ\cdots\circ f^{-1}_{\omega_{n-1}} U_m^c}
$$
Now we observe that the set
$$U^c_{m}(n):=U_m^c\cap f^{-1}_{\omega_1}U_m^c\cap f^{-1}_{\omega_1}\circ f^{-1}_{\omega_{2}}U_m^c\cap\cdots\cap f^{-1}_{\omega_1}\circ\cdots\circ f^{-1}_{\omega_{n-1}}U_m^c\cap \Omega_{\omega_1,\ldots,\omega_n}^{k_1,\ldots,k_n}$$
is actually given by:
$$
U_m^c(n):=U_m^c\cap f^{-1}_{k_1,\omega_1}U_m^c\cap f^{-1}_{k_1,\omega_1}\circ f^{-1}_{k_{2},\omega_{2}}U_m^c\cap\cdots\cap f^{-1}_{k_1,\omega_1}\circ\cdots\circ f^{-1}_{k_{n-1},\omega_{n-1}}U_m^c\cap \Omega_{\omega_1,\ldots,\omega_n}^{k_1,\ldots,k_n}
$$
Since $U_m^c$ is the disjoint union of two connected intervals, the number of connected intervals in $U_m^c(n)$ is  bounded from above by $n+1$ and it is important that it grows linearly with $n$.  We now take the total variation $\V(\Upsilon_{\omega_1,\ldots,\omega_n})$. We begin to remark that, by standard techniques:
\begin{multline*}
\V\left(\frac{(\psi g_n)((f^{-1}_{k_1, \omega_1}\circ\cdots\circ f^{-1}_{k_n, \omega_n})(x))}{|D(f_{\omega_n}\circ\cdots\circ f_{\omega_1})((f^{-1}_{k_1, \omega_1}\circ\cdots\circ f^{-1}_{k_n, \omega_n})(x))|}\I_{f_{\omega_n}\circ\cdots\circ f_{\omega_1}}\Omega_{\omega_1,\ldots,\omega_n}^{k_1,\ldots,k_n}(x)\right)\\
\leq 2\V_{H_{\omega_1,\ldots,\omega_n}^{k_1,\ldots,k_n}}\left(\frac{(\psi g_n)((f^{-1}_{k_1, \omega_1}\circ\cdots\circ f^{-1}_{k_n, \omega_n})(x))}{|D(f_{\omega_n}\circ\cdots\circ f_{\omega_1})((f^{-1}_{k_1, \omega_1}\circ\cdots\circ f^{-1}_{k_n, \omega_n})(x))|}\right)\\
 +\frac{2}{\beta^n}\frac{1}{\l(\Omega_{\omega_1,\ldots,\omega_n}^{k_1,\ldots,k_n})}\int_{\Omega_{\omega_1,\ldots,\omega_n}^{k_1,\ldots,k_n}}|\psi g_n|\, \dif\l
\end{multline*}
where $\beta$ is given by (\ref{du1}) in {\bf H2}.\\
The variation above can be further estimated by standard techniques:
\begin{multline}\label{V1}
\leq \frac{2}{\beta^n}\V_{\Omega_{\omega_1,\ldots,\omega_n}^{k_1,\ldots,k_n}}(\psi g_n)+ \frac{2}{\beta^n}\frac{1}{\l(\Omega_{\omega_1,\ldots,\omega_n}^{k_1,\ldots,k_n})}\int_{\Omega_{\omega_1,\ldots,\omega_n}^{k_1,\ldots,k_n}}|\psi g_n|\,\dif\l \\ +2\sup_{\zeta, \omega_1,\ldots,\omega_n}\frac{|D^2(f_{\omega_n}\circ\cdots\circ f_{\omega_1})(\zeta)|}{[D(f_{\omega_n}\circ\cdots\circ f_{\omega_1})(\zeta)]^2} \int_{\Omega_{\omega_1,\ldots,\omega_n}^{k_1,\ldots,k_n}}|\psi g_n|\,\dif\l
\end{multline}
We now have:
\begin{align*}
\frac{|D^2(f_{\omega_n}\circ\cdots\circ f_{\omega_1})(\zeta)|}{[D(f_{\omega_n}\circ\cdots\circ f_{\omega_1})(\zeta)]^2}=\sum_{k=0}^{n-1}\frac{ D^2f_{\omega_{n-k}}\left(\prod\limits_{l=1}^{n-1-k}T_{\omega_{n-l}}(\zeta)\right)}{\left[Df_{\omega_{n-k}}\left(\prod\limits_{l=1}^{n-1-k}f_{\omega_{n-l}}(\zeta)\right)\right]^2\prod\limits_{j=0}^{k}Df_{\omega_{n-j+1}}\left(\prod\limits_{l=1}^{n-j}f_{\omega_{n-l}}(\zeta)\right)}
\end{align*}
By (\ref{du2}) in {\bf H2} 
 and using again  (\ref{du1}),
 the previous sum will be bounded by $C_1$ times the sum of a geometric series of reason $\beta^{-1}$: we call $C$ the upper bound thus found.
 Our variation above is therefore bounded by:
\begin{equation}\label{V2}
\mbox{(\ref{V1})}\leq
\frac{2}{\beta^n}\V_{\Omega_{\omega_1,\ldots,\omega_n}^{k_1,\ldots,k_n}}(\psi g_n)+
\frac{2}{\beta^n}\frac{1}{\l(\Omega_{\omega_1,\ldots,\omega_n}^{k_1,\ldots,k_n})}\int_{\Omega_{\omega_1,\ldots,\omega_n}^{k_1,\ldots,k_n}}|\psi g_n|\dif\l
+ 2 C\int_{\Omega_{\omega_1,\ldots,\omega_n}^{k_1,\ldots,k_n}}|\psi g_n|\dif\l
\end{equation}
Now:
\begin{align}
\V_{\Omega_{\omega_1,\ldots,\omega_n}^{k_1,\ldots,k_n}}(\psi g_n)&\leq \V_{\Omega_{\omega_1,\ldots,\omega_n}^{k_1,\ldots,k_n}}(\psi)+2(n+1)
\sup_{\Omega_{\omega_1,\ldots,\omega_n}^{k_1,\ldots,k_n}}|\psi|\nonumber\\
&\leq [2(n+1)+1]\V_{\Omega_{\omega_1,\ldots,\omega_n}^{k_1,\ldots,k_n}}(\psi)+\frac{1}{\l(\Omega_{\omega_1,\ldots,\omega_n}^{k_1,\ldots,k_n})}\int_{\Omega_{\omega_1,\ldots,\omega_n}^{k_1,\ldots,k_n}}|\psi|\,\dif\l\label{V3}
\end{align}
where $2(n+1)$ is an estimate from above of the number of jumps of
$g_n$.
We now observe that for a finite realization of length $n$, $\omega_1,\ldots,\omega_n$, the quantity $\Psi_{n,\omega_1,\ldots,\omega_n}=\displaystyle\inf_{k_1,\ldots,k_n}\l(\Omega_{\omega_1,\ldots,\omega_n}^{k_1,\ldots,k_n})$,
where each $k_j$ runs over the finite branches of $f_{\omega_j}$, is surely strictly positive and this will also implies that  $
\Psi_n^{-1}:=\int \Psi_{n, \omega_1,\ldots,\omega_n}^{-1}\dif\theta_{\eps}^{\N}>0.$
We now replace (\ref{V3}) into (\ref{V2}), we sum over the ${k_1,\ldots,k_n}$ and we integrate w.r.t.\ $\theta_{\eps}^{\N}$; finally we get
$$
\V(\widetilde{\P}_{\eps,m}^n \psi)\leq \frac{2}{\beta^n}(2n+3)\V(\psi)+
[\frac{4}{\beta^n}\frac{1}{\Psi_n}+2C]\int_{I}|\psi|\,\dif\l
$$
In
order to get the Lasota-Yorke inequality one should get a certain
$n_0$ and a number $\beta >\kappa>1$ and such that
\begin{equation}\label{no}
\frac{2}{\beta^{n_0}}(2n_0+3)<\kappa^{-\n_0};
\end{equation}
the Lasota-Yorke inequality \eqref{eq:asmp2} will then follow  with standard arguments. \footnote{ By defining $A=2(2n_0+3)$ and $B=\left[\frac{4}{\Psi_{n_0}}+2C\right]\frac{2}{1-\kappa^{-n_0}}$, we have
$$
\V(\widetilde{\P}_{\eps,m}^n \psi)\leq A \kappa^{-n}\V( \psi)+ B \int_I |\psi|\,\dif\l .$$}

We now compute the $L^1$-norm of our operator. We have to compute $\|\widetilde{\P}_{\eps,m}^n \psi \|_1$; by splitting $\psi$ into the sum of its positive and negative parts and by using the linearity of the transfer operator, we may suppose that $\psi$ is non-negative. This allows us to interchange the integrals w.r.t.\ the Lebesgue measure and $\theta_{\eps}^{\N}$ and to use duality for each of the ${\mathcal P}_{\omega}$. In conclusion we get
$$
\| \widetilde{\P}_{\eps,m}^n \psi \|_1\le \int |\psi|h_{\eps} \I_{U_m^c}(x)\I_{U_m^c}(f_{\omega_1}x)\cdots \I_{U_m^c}(f_{\omega_{n-1}}\circ\cdots\circ f_{\omega_1}x)\, \dif\l\le \|\psi\|_1.
$$
This concludes the proof of the Lasota-Yorke inequality, ({\bf A1}).  We have now to show that 
%Besides the latter, the uniformity w.r.t.\ the noise  comes %from the assumptions on the uniform bound from below for the %first derivative and the uniform bound from above for the %second derivative.  With these two assumptions it is %straightforward to get the Lasota-Yorke inequality with the %{\em same} factors $A$ and $B$ for 
the operator ${\mathcal P}_{\eps}$, which is the unperturbed operator w.r.t.\ $\widetilde{\P}_{\eps,m}$, verifies the mixing condition ({\bf A2}). Now the  Perron-Frobenius operator ${\mathcal P}$ for the original map $f$, which is in turn the unperturbed operator w.r.t.\ ${\mathcal P}_{\eps}$, is  mixing  ($1$ is the only eigenvalue of finite multiplicity on the unit circle), since our original map $f$ was chosen to be mixing (hypothesis {\bf H1}),  and therefore, by the perturbation theory in \cite{KL09} and  the closeness of the two operators expressed by assumption {\bf H4},  also ${\mathcal P}_{\eps}$ is a mixing operator.  Let us discuss the assumption ({\bf A3}).

Let us bound the following quantity, for any $\psi$ of bounded variation and of total variation less than or equal to $1$:
\begin{multline*}
\left|\int_I (\widetilde{\P}_{\eps,m} \psi(x)-\P_{\eps} \psi(x))\,\dif\l(x)\right|=\left|\int_I \P_\eps(\I_{U_m}\psi)(x)\, \dif\l(x)\right|\\
\leq\left|\int\left(\int_I\P_{\omega}(\I_{U_m}\psi)\, \dif\l\right)\dif\theta_{\eps}(\omega)\right|\le \|\psi\|_{\infty} \ \l(U_m)
\end{multline*}
where $\|\psi\|_{\infty}\le \|\psi\|_{BV}$ and $\l(U_m)$ goes to zero when $m$ goes to infinity.

We now check assumption ({\bf A4}) under the hypothesis {\bf H5}.\\We have:
$$
\|(\widetilde{\P}_{\eps,m} -\P_{\eps})h_{\eps} \|_{BV}=\|\P_{\eps}({\bf 1}_{U_m}h_{\eps})\|_{BV}\le A\kappa^{-1}\|{\bf 1}_{U_m}h_{\eps}\|_{BV}+B\|{\bf 1}_{U_m}h_{\eps}\|_{1}
$$
The right hand side is bounded by a constant $C^*$ which is independent of $m$. We recall that in our case $\Delta_{\eps,m}=\mu_{\eps}(U_m)$ and that
$$\eta_{\eps,m}:=\sup_{\|\psi\|_{BV}\le 1}\left|\int_I (\widetilde{\P}_{\eps,m} \psi(x)-\P_{\eps} \psi(x))\, \dif\l(x)\right|\le \l(U_m)$$ (see computation above).
Then
$$
\|(\widetilde{\P}_{\eps,m} -\P_{\eps})h_{\eps} \|_{BV}\le C^*\frac{\mu_\eps(U_m)}{\underline{h}_{\eps}\l(U_m)}\le C^* \frac{\Delta_{\eps,m}}{\eta_{\eps,m}}
$$
\end{proof}

\subsection{Extremal index}
\label{subsec:q_k}

In this part, we investigate the quantity, see \eqref{eq:qk} and \eqref{eq:qk2}:
$$
q_{k,m}=\frac {\l(( \P_{\eps}-\widetilde{\P}_{\eps,m})\widetilde{\P}_{\eps,m}^k(\P_{\eps}-\widetilde{\P}_{\eps,m})(h_\eps))}{\mu_{\eps}(U_m)}
$$
We recall that $U_m:= U_m(\zeta)$ represents a ball around the point $\zeta$.
Our result is the following.
\begin{proposition}\label{calq}
Let us suppose that $f$ is either a $C^2$ expanding map of the circle or a piecewise expanding map of the circle 
with finite  branches and verifying hypotheses {\bf H1}-{\bf H4}. Then for each $k$,
$$
\lim_{m\rightarrow \infty}q_{k,m}\equiv 0,
$$
\ie the limit in the definition of $q_k$ in \eqref{eq:qk} exists and equals zero. Also the extremal index verifies
$
\vartheta=1-\sum_{k=0}^{\infty}q_k=1
$
and this is independent of the point $\zeta$, the center of the ball $U_m$.
\end{proposition}

\begin{proof}
Let us define $G_{k,m}\equiv \int ( \P_{\eps}-\widetilde{\P}_{\eps,m})\widetilde{\P}_{\eps,m}^k(\P_{\eps}-\widetilde{\P}_{\eps,m})h_\eps \, \dif\l$.

As $(\P_{\eps}-\widetilde{\P}_{\eps,m})\psi=\Pe(\I_{U_m} \psi)$, we may write $G_{k,m}=\int \I_{U_m}(x)\widetilde{\P}_{\eps,m}^k(\Pe-\widetilde{\P}_{\eps,m})h_\eps\, \dif\l.$\\ By using (\ref{FP}) we  get
$$G_{k,m} =\iint \I_{U_m}(f_{\omega_{k+1}}\circ f_{\omega_k}\circ\cdots\circ f_{\omega_1}x)\Iuc(f_{\omega_k}\circ\cdots\circ f_{\omega_1}x)\dots\Iuc(f_{\omega_1}x)\I_{U_m}(x)h_\eps(x)\, \dif\l\, \dif\nm.
$$
In order to simplify the notation let us put $$\psi_{k,U_m,\underline\omega}(x)=\I_{U_m}(f_{\omega_{k+1}}\circ f_{\omega_{k}}\circ\cdots\circ f_{\omega_1}x)\I_{U_m^c}(f_{\omega_{k}}\circ\cdots\circ f_{\omega_1}x)\ldots \I_{U_m^c}(f_{\omega_1}x)\I_{U_m}(x).$$
Now let us prove that $q_{k,m}$ converges to 0. Our approach is very similar to what we did to prove $D'(u_m)$ and we now split the proof according to the regularity of the map.\\

(i) Suppose that $f:\mathcal S^1\to \mathcal S^1$ is a $C^2$, expanding map, \ie there exists $|Df(x)|>\lambda>1$, for all $x\in\S^1$.
First, note that since $S^1$ is compact and $f$ is $C^2$, there exists $\sigma>1$ such that $|Df(x)|\leq \sigma$. Hence the set $U_m$  grows at most at a rate given by $\sigma$, so, for any $\o\in\Omega^\N$ we have $|f_{\o}^j(U_m)|\leq \sigma^j |U_m|$. This implies that
\begin{equation}
\label{eq:fact}
\mbox{if $\dist(f_{\o}^j(\zeta),\zeta)>2\sigma^j|U_m|>|U_m|+\sigma^j|U_m|$ then $f_{\o}^j(U_m)\cap U_m=\emptyset$}.
\end{equation}
Note that, by inequality \eqref{eq:fact}, if for all $j=1,\ldots,k+1$ we have $\dist(f_{\o}^j(\zeta),\zeta)>2\sigma^j|U_m|$, then clearly $\psi_{k,B_m,\o}(x)=0$, for all $x$. We define
\begin{equation}
\label{def:W}
W_{k,m}=\bigcap_{j=1}^{k+1}\big\{\o\in(-\eps,\eps)^\N: \dist(f_{\o}^j(\zeta),\zeta)>2\sigma^j|U_m| \big\}.
\end{equation}
Note that on $W_{k,m}$ we have $\psi_{k,U_m,\o}=0$. We want to compute the measure of $W_{k,m}^c$.

Observe that
$W_{k,m}^c\subset \bigcup_{j=1}^{k+1} \big\{\o:\; f_{\o}^j(\zeta)\in B_{2\sigma^j|U_m|}(\zeta)\big\}.$
Hence, we have
\begin{align*}
\theta_\eps^\N(W_{k,m}^c)&\leq \sum_{j=1}^{
k+1}\int \theta_\eps\Big(\Big\{\omega_j: f\left(f_{\o}^{j-1}(\zeta)\right)+\omega_j\in  B_{2\sigma^j|U_m|}(\zeta) \Big\}\Big)\dif\theta^\N_\eps\\
%&=\sum_{j=1}^{k+1}\int\theta_\eps\Big(\Big\{\omega_j: \omega_j\in  B_{2\sigma^j|U_m|}(\zeta)-f\left(f_{\o}^{j-1}(\zeta)\right) \Big\}\Big)\dif\theta^\N_\eps\\
%&=\sum_{j=1}^{k+1}\iint_{B_{2\sigma^j|U_m|}(\zeta)-f\left(f_{\o}^{j-1}(\zeta)\right)}g_\eps(x) \,\dif x\dif\theta^\N_\eps\\
&\leq \sum_{j=1}^{k+1} \overline{g_\eps}\left|B_{2\sigma^j|U_m|}(\zeta)\right|=\sum_{j=1}^{k+1} \overline{g_\eps} 4\sigma^j|U_m|
\leq 4\overline{g_\eps} |U_m|\frac{\sigma}{\sigma-1}\sigma^{k+1}.
\end{align*}

%\textbf{HA: (69) again, the second and the third inequalities should be equalities?}

Using this estimate we obtain:
\begin{align*}
G_{k,m}&=\int_{W_{k,m}}\int\psi_{k,U_m,\underline\omega}(x)h_{\eps}(x)\,\dif\l\,\dif\theta_\eps^\N+\int_{W_{k,m}^c}\int\psi_{k,U_m,\underline\omega}(x)h_{\eps}(x)\,\dif\l\,\dif\theta_\eps^\N\\
&=0+\int_{W_{k,m}^c}\int\psi_{k,U_m,\underline\omega}(x)h_{\eps}(x)\,\dif\l\,\dif\theta_\eps^\N \quad\mbox{ and because $\psi_{k,U_m,\underline\omega}(x)\leq \I_{U_m}(x) $, we have: }\\
&\leq  \int_{W_{k,m}^c}\int \I_{U_m}(x)h_{\eps}(x)\,\dif\l\,\dif\theta_\eps^\N\leq \mu_{\eps}(U_m)\,\theta_\eps^\N(W_{k,m}^c)\\
&\leq \mu_{\eps}(U_m) 4\overline{g_\eps} |U_m|\frac{\sigma}{\sigma-1}\sigma^{k+1}
\end{align*}

Now recall that $q_{k,m}=\frac{G_{k,m}}{\mu_\eps(U_{m})}$. It follows that
\begin{align*}
q_{k,m}&\leq \frac{ \mu_{\eps}(U_m) 4\overline{g_\eps} |U_m|\frac{\sigma}{\sigma-1}\sigma^{k+1}
}{\mu_\eps(U_m)}
\leq  4\overline{g_\eps} |U_m|\frac{\sigma}{\sigma-1}\sigma^{k+1}\xrightarrow[m\to\infty]{}0.
\end{align*}

(ii) Using the same ideas as in the previous section, we can extend this result to the piecewise expanding maps with finite branches. Recall that we need to define some 'safety boxes' in order to use the same arguments as in the continuous case. So, if for all $j=1,\ldots, k+1$ and $i=1,\dots,\ell$, where $\ell$ stands for the number of discontinuity points, we have
\begin{equation}
\label{eq:safetybox1}
\mbox{$\dist(f_{\o}^j(\zeta),\xi_i)>2\sigma^j|U_m|$},
\end{equation}
then the set $U_m$ consists of one connected component at each iteration, and also we have $f_{\o}^j(U_m)\cap U_m=\emptyset$ which means $\psi_{k,U_m,\o}(x)=0$, for all $x$. Now let us define
\begin{equation}
\label{def:W2}
W_{k,m}=\bigcap_{j=1}^{k+1}\bigcap_{i=0}^{\ell}\big\{\o\in(-\eps,\eps)^\N: \dist(f_{\o}^j(\zeta),\xi_i)>2\sigma^j|U_m| \big\}.
\end{equation}

Observe that in this case
$W_{k,m}^c\subset \bigcup_{j=1}^{k+1}\bigcup_{i=0}^{\ell} \big\{\o:\; f_{\o}^j(\zeta)\in B_{2\sigma^j|U_m|}(\xi_i)\big\}.$
Hence, we have
\begin{align*}
\theta_\eps^\N(W_{k,m}^c)&\leq \sum_{i=0}^{\ell}\sum_{j=1}^{
k+1}\int \theta_\eps\Big(\Big\{\omega_j: f\left(f_{\o}^{j-1}(\zeta)\right)+\omega_j\in  B_{2\sigma^j|U_m|}(\xi_i) \Big\}\Big)\dif\theta^\N_\eps\\
%&= \sum_{i=0}^{\ell}\sum_{j=1}^{k+1}\int\theta_\eps\Big(\Big\{\omega_j: \omega_j\in  B_{2\sigma^j|U_m|}(\xi_i)-f\left(f_{\o}^{j-1}(\zeta)\right)\Big\}\Big)\dif\theta^\N_\eps\\
%&= \sum_{i=0}^{\ell}\sum_{j=1}^{k+1} \iint_{B_{2\sigma^j|U_m|}(\xi_i)-f\left(f_{\o}^{j-1}(\zeta)\right)}g_\eps(x) dx\,\dif\theta^\N_\eps\\
&\leq \sum_{i=0}^{\ell}\sum_{j=1}^{k+1} \overline{g_\eps}\left|B_{2\sigma^j|U_m|}(\xi_i)\right|=\sum_{i=0}^{\ell}\sum_{j=1}^{k+1} \overline{g_\eps} 4\sigma^j|U_m|
\leq 4(\ell+1)\overline{g_\eps} |U_m|\frac{\sigma}{\sigma-1}\sigma^{k+1}
\end{align*}

%\textbf{HA: (70) again, the second and the third inequalities should be equalities?}

Using this estimate we obtain:
\begin{align*}
G_{k,m}&=\int_{W_{k,m}}\int\psi_{k,U_m,\underline\omega}(x)h_{\eps}(x)\,\dif\l\,\dif\theta_\eps^\N+\int_{W_{k,m}^c}\int\psi_{k,U_m,\underline\omega}(x)h_{\eps}(x)\,\dif\l\,\dif\theta_\eps^\N\\
&=0+\int_{W_{k,m}^c}\int\psi_{k,U_m,\underline\omega}(x)h_{\eps}(x)\,\dif\l\,\dif\theta_\eps^\N \quad\mbox{ and because $\psi_{k,U_m,\underline\omega}(x)\leq \I_{U_m}(x) $, we have: }\\
&\leq  \int_{W_{k,m}^c}\int \I_{U_m}(x)h_{\eps}(x)\,\dif\l\,\dif\theta_\eps^\N\leq \mu_{\eps}(U_m)\theta_\eps^\N(W_{k,m}^c)
\leq \mu_{\eps}(U_m) 4(\ell+1)\overline{g_\eps} |U_m|\frac{\sigma}{\sigma-1}\sigma^{k+1}
\end{align*}

Since $q_{k,m}=\frac{G_{k,m}}{\mu_\eps(U_{m})}$, we get
\begin{align*}
q_{k,m}&\leq \frac{ \mu_{\eps}(U_m) 4(\ell+1)\overline{g_\eps} |U_m|\frac{\sigma}{\sigma-1}\sigma^{k+1}
}{\mu_\eps(U_m)}
\leq  4(\ell+1)\overline{g_\eps} |U_m|\frac{\sigma}{\sigma-1}\sigma^{k+1}\xrightarrow[m\to\infty]{}0.
\end{align*}
\end{proof}

\begin{remark}
\label{prop:D'=>q_k}
Let us note that $D'(u_m)$ implies that all $q_k$'s are well defined and equal to 0. Assume that there exists $k\in\N$ and a subsequence $(m_i)_{i\in\N}$ such that $\lim_{j\to\infty} \frac{G_{k,{m_j}}}{\mu_\eps(U_{m_j})}=\alpha>0$. Let us prove that $D'(u_m)$ does not hold in this situation. Recall that if $D'(u_m)$ holds then $$\lim_{m\rightarrow\infty}\,m\sum_{j=1}^{\lfloor m/k_m \rfloor}\mu_\eps\times\theta_\eps^\N( X_0>u_m,X_j>u_m)=0,$$
where $k_m$ (which should not be confused with $k$, here) is a sequence diverging to $\infty$ but slower than $m$, which implies that $\lfloor m/k_m \rfloor\to\infty$, as $m\to\infty$. Hence, let $M_0$ be sufficiently large so that for all $m>M_0$ we have $\lfloor m/k_m \rfloor>k$. Hence, for $i$ sufficiently large so that $m_i>M_0$, we may write
\begin{align*}
m_i\sum_{j=1}^{\lfloor m_i/k_{m_i} \rfloor}\mu_\eps\times\theta_\eps^\N( X_0>u_{m_i},X_j>u_{m_i})&\geq {m_i}\, \mu_\eps\times\theta_\eps^\N( X_0>u_{m_i},X_{k+1}>u_{m_i})\\
&\geq {m_i}\,G_{k,m_i}\sim  \frac{\tau\,G_{k,m_i}}{\mu_\eps(U_{m_i})}\to \tau \alpha>0,\mbox{ as $i\to\infty$},
\end{align*}
since $B_m$ is such that $m \mu_\eps(U_m)\to\tau$, as $m\to\infty$. This implies that $D'(u_m)$ does not hold.
\end{remark}

\section*{Acknowledgements}
We wish to thank Ian Melbourne for helpful comments and, in particular, for calling us the attention for the paper \cite{D98} in this context. We are grateful to  Wael Bahsoun and Manuel Stadlbauer
for fruitful conversations and comments. We are also grateful to Mike Todd for careful reading and many suggestions in order to improve the paper. 

\appendix
\section{Clustering and periodicity}
\label{sec:periodicity}

Condition $D'(u_n)$ prevents the existence of clusters of exceedances, which implies  that the EVL is  standard exponential $\bar H(\tau)=\e^{-\tau}$. However, when $D'(u_n)$ fails, clustering of exceedances is responsible for the appearance of a parameter $0<\vartheta<1$ in the EVL, called the EI, which implies that, in this case, $\bar H(\tau)=\e^{-\vartheta \tau}$. In \cite{FFT12},  the authors established a connection between the existence of an EI less than 1 and periodic behaviour. This was later generalised for REPP in \cite{FFT12a}. Namely, this phenomenon of clustering appeared when $\zeta$ is a repelling periodic point. We assume that the invariant measure $\p$ and the observable $\varphi$ are sufficiently regular so that besides \ref{item:U-ball}, we also have that

\begin{enumerate}
\item[\namedlabel{item:repeller}{(R2)}]  If $\zeta\in \X$ is a repelling periodic point, of prime period\footnote{i.e.,  the \emph{smallest} $n\in \N$ such that $f^n(\zeta)=\zeta$.  Clearly $f^{ip}(\zeta)=\zeta$ for any $i\in \N$. }  $p\in\N$, then we have that the periodicity of $\zeta$ implies that for all large $u$, $\{X_0>u\}\cap f^{-p}(\{X_0>u\})\neq\emptyset$ and the fact that the prime period is $p$ implies that $\{X_0>u\}\cap f^{-j}(\{X_0>u\})=\emptyset$ for all $j=1,\ldots,p-1$. Moreover,
 the fact that $\zeta$ is repelling means that we have backward contraction which means that there exists $0<\vartheta<1$ so that $\bigcap_{j=0}^i f^{-jp}(X_0>u)$ corresponds to another ball of smaller radius around $\zeta$ with $\p\left(\bigcap_{j=0}^i f^{-jp}(X_0>u)\right)\sim(1-\vartheta)^i\p(X_0>u),$ for all $u$ sufficiently large. 
 
\end{enumerate}

The main obstacle when dealing with periodic points is that they create plenty of dependence in the short range. In particular, using \ref{item:repeller} we have that for all $u$ sufficiently large
$
\p(\{X_0>u\}\cap \{X_p>u\})\sim(1-\vartheta)\p(X_0>u)
$
which implies that $D'(u_n)$ is not satisfied, since for the levels $u_n$ as in \eqref{eq:un} it follows that
$
n\sum_{j=1}^{[n/k_n]}\p(X_0>u_n, X_j>u_n)\geq n\p(X_0>u_n, X_p>u_n)\xrightarrow[n\to\infty]{}(1-\vartheta)\tau.
$
To overcome this difficulty around periodic points we make a key observation that roughly speaking tells us that around periodic points one just needs to replace the topological ball $\{X_0>u_n\}$ by the topological annulus \begin{equation}
\label{eq:def-Qp}
Q_p(u):=\{X_0>u, \;X_p\leq u\}.
\end{equation}  Then much of the analysis works out as in the absence of clustering. 
Note that $Q_p(u)$ is obtained by removing from $U(u)$ the points that were doomed to return after $p$ steps, which form the smaller ball $U(u)\cap f^{-p}(U(u))$. Then, the  crucial observation  is that  the limit law corresponding to no entrances up to time $n$ into the ball $U(u_n)$ is equal to the limit law corresponding to no entrances into the annulus $Q_p(u_n)$ up to time $n$.

In what follows for every $A\in\mathcal B$, we denote the complement of $A$ as $A^c:=\mathcal X\setminus A$.
For $s\leq\ell\in \N_0$, we define
\begin{equation}
\label{eq:no-entrance-annulus}
 \QQ_{p,s,\ell}(u)=\bigcap_{i=s}^{s+\ell-1} f^{-i}(Q_p(u))^c,
 \end{equation}
which corresponds to no entrances in the annulus from time $s$ to $s+\ell-1$. Sometimes to abbreviate we also write: $\QQ_\ell(u):=\QQ_{p,0,\ell}(u)$.

\begin{theorem}[{\cite[Proposition~1]{FFT12}}]
\label{thm:ball-annulus} Let $X_0, X_1,,\ldots$ be a stochastic process defined by \eqref{eq:def-stat-stoch-proc-DS} where $\varphi$ achieves a global maximum at a repelling periodic point $\zeta\in \X$, of prime period $p\in\N$, so that conditions  \ref{item:U-ball} and \ref{item:repeller} above hold. Let $(u_n)_n$ be a sequence of levels such that \eqref{eq:un} holds. Then,
$
\lim_{n\to\infty}\p(M_n\leq u_n)=\lim_{n\to\infty}\p(\mathscr Q_n(u_n)).
$   
\end{theorem}

Hence, the idea to cope with clustering caused by periodic points is to adapt conditions $D_2(u_n)$ and $D'(u_n)$, letting annuli replace balls. %In order to make the theory as general as possible, motivated by the above considerations for stochastic processes generated by dynamical systems around periodic points, some   abstract conditions were given in \cite{FFT12} to prove the existence of an EI less than 1 for general stationary stochastic processes.  
%
% The first one establishes exactly the type of periodic behaviour assumed, namely:
%\begin{condition}[$\spp(u_n)$]\label{cond:SP} We say that  $X_0,X_1,X_2,\ldots$ satisfies condition $\spp(u_n)$ for $p\in \N$ and $\vartheta\in [0,1]$ if
%\begin{equation}
%\label{cond:periodicity}
%\lim_{n\to \infty}\sup_{1\le j<p}\p(X_j>u_n| X_0>u_n)=0 \quad \mbox{ and}\quad \lim_{n\to \infty} \p(X_p>u_n| X_0>u_n)\to (1-\vartheta)
%\end{equation}
%and moreover
%\begin{equation}
%\label{cond:summability}
%\lim_{n\to \infty}\sum_{i=0}^{[\frac{n-1}p]} \p(X_0>u_n, X_p>u_n, X_{2p}>u_n,\ldots,X_{ip}>u_n)=0.
%\end{equation}
%\end{condition}
%Condition \eqref{cond:periodicity}, when $\vartheta<1$, imposes some sort of periodicity of period $p$ among the exceedances of high levels $u_n$, since if at some point the process exceeds the high level $u_n$, then, regardless of how high $u_n$ is, there is always a strictly positive probability of another exceedance occurring at the (finite) time $p$.  In fact, if the process is generated by a deterministic dynamical system $f:\X\to \X$ and $f$ is continuous then \eqref{cond:periodicity} implies that $\zeta$ is a periodic point of period $p$, \ie $f^p(\zeta)=\zeta$.
%

%The next two conditions concern to the dependence structure of $X_0, X_1,\ldots$ and can be described as being obtained from $D_2(u_n)$ and $D'(u_n)$ by replacing balls by annuli. 
\begin{condition}[$D^p(u_n)$]\label{cond:Dp}We say that $D^p(u_n)$ holds
for the sequence $X_0,X_1,X_2,\ldots$ if for any integers $\ell,t$
and $n$
\( \left|\p\left(Q_{p,0}(u_n)\cap
  \QQ_{p,t,\ell}(u_n)\right)-\p(Q_{p,0}(u_n))
  \p(\QQ_{p,0,\ell}(u_n))\right|\leq \gamma(n,t),
\)
where $\gamma(n,t)$ is non increasing in $t$ for each $n$ and
$n\gamma(n,t_n)\to0$ as $n\rightarrow\infty$ for some sequence
$t_n=o(n)$.  %, which means that $t_n/n\to0$ as $n\to \infty$.
\end{condition}
As with $D_2(u_n)$, the main advantage of this condition when compared to Leadbetter's $D(u_n)$ (or others of the same sort) is that it follows directly from sufficiently fast decay of correlations as observed in \cite[Section~5.1]{F12}, on the contrary to $D(u_n)$. 

Assuming $D^p(u_n)$ holds let $(k_n)_{n\in\N}$ be a sequence of integers such that \eqref{eq:kn-sequence-1} holds.
\begin{condition}[$D'_p(u_n)$]\label{cond:D'p} We say that $D'_p(u_n)$
holds for $X_0,X_1,X_2,\ldots$ if there exists a sequence $(k_n)_{n\in\N}$ satisfying \eqref{eq:kn-sequence-1} and such that
%\begin{equation}
%\label{eq:D'rho-un}
$\lim_{n\rightarrow\infty}\,n\sum_{j=1}^{[n/k_n]}\p( Q_{p,0}(u_n)\cap
Q_{p,j}(u_n))=0.$
%\end{equation}
\end{condition}

One of the main results in \cite{FFT12} is:
\begin{theorem}[{\cite[Theorem~1]{FFT12}}]
  \label{thm:existence-EI}
  Let $(u_n)_{n\in\N}$ be such that \eqref{eq:un} holds.
  Consider a stationary stochastic process $X_0, X_1,\ldots$ be a stochastic process defined by \eqref{eq:def-stat-stoch-proc-DS} where $\varphi$ achieves a global maximum at a repelling periodic point $\zeta\in \X$, of prime period $p\in\N$, so that conditions  \ref{item:U-ball} and \ref{item:repeller} above hold.
  Assume further that conditions
  $D^p(u_n)$ and $D'_p(u_n)$ hold. Then
%  \begin{equation}
%  \label{eq:max-ei}
   $\lim_{n\to\infty}\p(M_n\leq u_n)=\lim_{n\to\infty}\p(\QQ_{p,0,n}(u_n))=\e^{-\vartheta \tau}$.
 % \end{equation}
\end{theorem}

Regarding the convergence of the REPP, when there is clustering, one cannot use the aforementioned criterion of Kallenberg  because the point processes are not simple anymore and possess multiple events. This means that a much deeper analysis must be done in order to obtain convergence of the REPP. We carried this out in \cite{FFT12a} and we describe below the main results and conditions needed.  
First, we define the sequence $\left(U^{(\kappa)}(u)\right)_{\kappa\geq0}$ of nested balls centred at $\zeta$ given by: 
\begin{equation}
\label{eq:Uk-definition}
U^{(0)}(u)=U(u)  
 \quad\text{and}\quad U^{(\kappa)}(u)=f^{-p}(U^{(\kappa-1)}(u))\cap U(u), \quad\text{for all $\kappa\in\N$.} 
 \end{equation} 
For $i,\kappa,\ell,s\in\N\cup\{0\}$, we define the following events:
\begin{equation}\label{eq:Q-definition}
Q_{p,i}^\kappa(u):=f^{-i}\left(U^{(\kappa)}(u)-U^{(\kappa+1)}(u)\right).
\end{equation}

Observe that for each $\kappa$, the set $Q_{p,0}^\kappa(u)$ corresponds to an annulus centred at $\zeta$. Besides, 
$U(u)=\bigcup_{\kappa=0}^\infty Q_{p,0}^\kappa(u),$
which means that the ball centred at $\zeta$ which corresponds to $U(u)$ can be decomposed into a sequence of disjoint annuli where $Q_{p,0}^0(u)$ is the most outward ring and the inner ring $Q_{p,0}^{\kappa+1}(u)$ is sent outward by $f^p$ to the ring $Q_{p,0}^\kappa(u)$, i.e., 
$f^{p}(Q_{p,0}^{\kappa+1}(u))=Q_{p,0}^\kappa(u).$

We are now ready to state:\begin{condition}[$D_p(u_n)^*$]\label{cond:Dp*}We say that $D_p(u_n)^*$ holds
for the sequence $X_0,X_1,X_2,\ldots$ if for any integers $t, \kappa_1,\ldots,\kappa_\varsigma$, $n$ and
 any $J=\cup_{i=2}^\varsigma I_j\in \mathcal R$ with $\inf\{x:x\in J\}\ge t$, 
 \[ \left|\p\left(Q_{p,0}^{\kappa_1}(u_n)\cap \left(\cap_{j=2}^\varsigma \nn_{u_n}(I_j)=\kappa_j \right) \right)-\p\left(Q_{p,0}^{\kappa_1}(u_n)\right)
  \p\left(\cap_{j=2}^\varsigma \nn_{u_n}(I_j)=\kappa_j \right)\right|\leq \gamma(n,t),
\]
where for each $n$ we have that $\gamma(n,t)$ is nonincreasing in $t$  and
$n\gamma(n,t_n)\to0$  as $n\rightarrow\infty$, for some sequence
$t_n=o(n)$. 
\end{condition}
This mixing condition is stronger than $D^p(u_n)$ because it requires a uniform bound for all $\kappa_1$, nonetheless, it still is much weaker than the original $D(u_n)$ from Leadbetter \cite{L73} or any of the kind. As all the other preceding conditions ($D_2,D_3,D^p$) it can be easily verified for systems with sufficiently fast decay of correlations (see \cite[Section~5.1]{F12}).     

In \cite{FFT12a}, for technical reasons only, we also introduced a slight modification to $D'_p(u_n)$. The new condition was denoted by $D'_p(u_n)^*$ and the difference is that we require that 
%\begin{equation}
%\label{eq:D'rho-un*}
$\lim_{n\rightarrow\infty}\,n\sum_{j=1}^{[n/k_n]}\p( Q_{p,0}(u_n)\cap
\{X_j>u_n\})=0$ holds.
%\end{equation}

We can now state the main theorem in \cite{FFT12a}. 
\begin{theorem}[{ \cite[Theorem~1]{FFT12a}}]
\label{thm:convergence-point-process} Let $X_0, X_1, \ldots$ be given by \eqref{eq:def-stat-stoch-proc-DS}, where $\varphi$ achieves a global maximum at the repelling periodic point $\zeta$, of prime period $p$, and conditions \ref{item:U-ball} 
 and \ref{item:repeller} hold. 
 Let
$(u_n)_{n\in\N}$ be a sequence satisfying \eqref{eq:un}.  
Assume that conditions 
$D_p(u_n)^*$, $D'_p(u_n)^*$ hold.
Then the REPP $N_n$ converges in distribution to a compound Poisson process $N$ with intensity $\vartheta$ and multiplicity d.f. $\pi$ given by
$
\pi(\kappa)=\vartheta(1-\vartheta)^\kappa,
$
for every $\kappa\in\N_0$, where the extremal index $\vartheta$ is given by the expansion rate at $\zeta$ stated in \ref{item:repeller}.
\end{theorem}

\section{Computing the EI and the multiplicity distribution}
\label{sec:EI-multiplicity-formulas}

In order to prove the existence of an EI around a repelling periodic point, we may use Theorem~\ref{thm:existence-EI} and, basically, observe that, once conditions $D^p(u_n)$ and $D'_p(u_n)$ are verified, by \ref{item:repeller} the EI may be computed from the formula:
\begin{equation}
\label{eq:EI-formula}
\vartheta=\lim_{n\to\infty}\frac{\p(Q_{p,0}(u_n))}{\p(U_n)}.
\end{equation}
In order to compute the multiplicity distribution of the limiting compound Poisson process for the REPP, when $\zeta$ is a repelling periodic point, we can use the following estimate :
\begin{lemma}[{\cite[Corollary~2.4]{FFT12a}}]
\label{lem:k-ring}
Assuming that $\varphi$ achieves a global maximum at the repelling periodic point $\zeta$, of prime period $p$, and conditions \ref{item:U-ball} and \ref{item:repeller} hold, there exists $C>0$ depending only on $\vartheta$ given by property \ref{item:repeller} such that for any $s,\kappa\in\N$ and $u$ sufficiently close to $u_F=\varphi(\zeta)$ we have for $\kappa>0$
\begin{multline*}
\left|\p\big(\nn_{u,0}^{s+1}=\kappa\big)-s\left(\p(Q_{p,0}^{\kappa-1}(u))-\p(Q_{p,0}^\kappa(u))\right)\right|\\ \leq 4s\sum_{j=p+1}^s\p(Q_{p,0}^0(u)\cap \{X_j>u\})+2C\,\p(X_0>u_n),
\end{multline*}
and in the case $\kappa=0$
$$
\left|\p\big(\nn_{u,0}^{s+1}=0\big)-\left(1-s\p(Q_{p,0}^0(u))\right)\right|\leq 2s\sum_{j=p+1}^s\p(Q_{p,0}^0(u)\cap \{X_j>u\})+C\,\p(X_0>u).
$$
\end{lemma}
The idea then is to realise that in the proof of Theorem~\ref{thm:convergence-point-process} one splits the first $n$ r.v. $X_0, \ldots, X_{n-1}$ into blocks of size $\lfloor n/k_n\rfloor$ with a time gap of size $t_n$ between them. Then using the asymptotic ``independence'' obtained from $D^p(u_n)^*$ and $D'_p(u_n)^*$ we get the compound Poisson limit with multiplicity distribution determined by the distributional limit of the number of exceedances in each block of size  $\lfloor n/k_n\rfloor$, given that at least one exceedance occurs. Hence, we need to compute, for all $\kappa\in\N$:
$\lim_{n\to\infty}\p\left(\nn_{u_n,0}^{\lfloor n/k_n\rfloor+1}=\kappa | \nn_{u_n,0}^{\lfloor n/k_n\rfloor+1}>0\right)$. Since, by $D'_p(u_n)^*$, we have that 
$\lfloor n/k_n\rfloor\sum_{j=p+1}^{\lfloor n/k_n\rfloor}\p(Q_{p,0}^0(u_n)\cap \{X_j>u_n\})=o(1/k_n)$, then by Lemma~\ref{lem:k-ring} we have that, for every $\kappa\in\N$,
\begin{equation}
\label{eq:multiplicity}
\pi(\kappa)=\lim_{n\to\infty}\p\left(\nn_{u_n,0}^{\lfloor n/k_n\rfloor+1}=\kappa | \nn_{u_n,0}^{\lfloor n/k_n\rfloor+1}>0\right)=\lim_{n\to\infty}\frac{%\frac{n}{k_n}
\left(\p(Q_{p,0}^{\kappa-1}(u_n))-\p(Q_{p,0}^\kappa(u_n))\right)}{%\frac{n}{k_n}
\p(Q_{p,0}^0(u_n))}.
\end{equation}

\bibliographystyle{novostyle}
\bibliography{RandomExtremes}

\end{document}